\documentclass[a4paper,10pt,leqno]{amsart}
        \title{Approximately fibering a manifold over an aspherical one}

       \author{Tom Farrell}
       \address{Yau Mathematical Sciences Center and Department of Mathematics\\Tsinghua University\\ Beijing, China.}
       \email{farrell@math.tsinghua.edu.cn}
       \author{Wolfgang L\"uck}
       \address{Universit\"at Bonn\\
               Mathematisches Institut\\
               Endenicher Allee 60,
               D-53115 Bonn, Germany}
       \email{wolfgang.lueck@him.uni-bonn.de}
       \urladdr{http://www.math.uni-muenster.de/u/lueck}
       \author{Wolfgang Steimle}
       \address{Institut f\"ur Mathematik\\
				Universit\"at Augsburg\\
				D-86135 Augsburg, Germany}
       \email{wolfgang.steimle@math.uni-augsburg.de}
	  \urladdr{http://www.math.uni-augsburg.de/prof/diff/arbeitsgruppe/steimle/}
       \date{July, 2018}
  
 

\usepackage[active]{srcltx}
\usepackage{pdfsync}
\usepackage{color}
\usepackage{hyperref}
\usepackage{calc}
\usepackage{enumerate,amssymb}
\usepackage{graphicx}
\usepackage[arrow,curve,matrix,tips,2cell]{xy}
  \SelectTips{eu}{10} \UseTips
  \UseAllTwocells
\usepackage{tikz}
\usetikzlibrary{decorations.pathmorphing,snakes}

\DeclareMathAlphabet{\matheurm}{U}{eur}{m}{n}
\newcommand{\NK}{N\!K}
\newcommand{\Or}{\matheurm{Or}}
\newcommand{\Spaces}{\matheurm{Spaces}}
\newcommand{\Spectra}{\matheurm{Spectra}}

\DeclareMathOperator{\asmb}{asmb}

\DeclareMathOperator{\cok}{cok}

\DeclareMathOperator{\cyl}{cyl}

\DeclareMathOperator{\fib}{fib}

\DeclareMathOperator{\GL}{GL}

\DeclareMathOperator{\id}{id}
\DeclareMathOperator{\im}{im}

\DeclareMathOperator{\map}{map}

\DeclareMathOperator{\NWh}{NWh}

\DeclareMathOperator{\pt}{pt}
\DeclareMathOperator{\pr}{pr}

\DeclareMathOperator{\Wh}{Wh}

\newcommand{\VCyc}{{\mathcal{VC}\text{yc}}}
\newcommand{\Linfty}{L^{\langle -\infty\rangle}}
\newcommand{\bfLinfty}{\mathbf{L}^{\langle -\infty\rangle}}

  \newcommand{\IC}{\mathbb{C}}
  
  \newcommand{\IE}{\mathbb{E}}

  \newcommand{\IH}{\mathbb{H}}

  \newcommand{\IP}{\mathbb{P}}
  
  \newcommand{\IR}{\mathbb{R}}

  \newcommand{\IZ}{\mathbb{Z}}

  \newcommand{\cala}{\mathcal{A}}
  
  \newcommand{\calc}{\mathcal{C}}
  
  \newcommand{\cale}{\mathcal{E}}

  \newcommand{\calm}{\mathcal{M}}
  \newcommand{\caln}{\mathcal{N}}

  \newcommand{\calp}{\mathcal{P}}

  \newcommand{\cals}{\mathcal{S}}
  \newcommand{\calt}{\mathcal{T}}
  
  \newcommand{\calv}{\mathcal{V}}
  \newcommand{\calw}{\mathcal{W}}

  \newcommand{\bfA}{{\mathbf A}}
  \newcommand{\bfa}{{\mathbf a}}
  \newcommand{\bfasmb}{{\mathbf a}{\mathbf s}{\mathbf m}{\mathbf b}}
  
  \newcommand{\bfb}{{\mathbf b}}

  \newcommand{\bfE}{{\mathbf E}}

  \newcommand{\bfK}{{\mathbf K}}
  \newcommand{\bfL}{{\mathbf L}}

  \newcommand{\bfu}{{\mathbf u}}

  \newcommand{\bfWh}{{\mathbf W}{\mathbf h}}

\newcommand{\inv}{^{-1}}
\newcommand{\eps}{\varepsilon}

\newcounter{commentcounter}

\definecolor{lightgray}{RGB}{200,200,200}

\theoremstyle{plain}
\newtheorem{theorem}{Theorem}[section]

\newtheorem{lemma}[theorem]{Lemma}
\newtheorem{corollary}[theorem]{Corollary}
\newtheorem{proposition}[theorem]{Proposition}
\newtheorem{conjecture}[theorem]{Conjecture}

\newtheorem{definition}[theorem]{Definition}
\theoremstyle{definition}

\newtheorem{addendum}[theorem]{Addendum}

\newtheorem{remark}[theorem]{Remark}
\newtheorem{notation}[theorem]{Notation}

\theoremstyle{remark}

\makeatletter\let\c@equation=\c@theorem\makeatother

\newcommand{\version}[1]                       
{\begin{center} Last edited on #1\\
Last compiled on \today\\
name of the file: \jobname
\end{center}
}

\newcommand{\EGF}[2]{E_{#2}(#1)}

\newcommand{\intgf}[2]{\int_{#1} #2}

\newcommand{\xycomsquare}[8]                % commutative square (xy-Version)
{$$\xymatrix{#1 \ar[r]^-{#2} \ar[d]^{#4} &
    #3 \ar[d]^{#5}  \\
    #6\ar[r]^-{#7} & #8 }$$}

\begin{document}

\begin{abstract}
  The paper is devoted to the problem when a map from some closed connected manifold to an
  aspherical closed manifold approximately fibers, i.e., is homotopic to Manifold
  Approximate Fibration. We define obstructions in algebraic $K$-theory.  Their vanishing
  is necessary and under certain conditions sufficient.  Basic ingredients are Quinn's
  Thin $h$-Cobordism Theorem and End Theorem, and knowledge about the Farrell-Jones
  Conjectures in algebraic $K$- and $L$-theory and the MAF-Rigidity Conjecture by
  Hughes-Taylor-Williams.
\end{abstract}

\maketitle

%%%%%%%%%%%%%%%%%%%%%%%%%%%%%%%%%%%%%%%%%%%%%%%%%%%%%%%%%%%%%%%%%%%%%%%%%%%%%%%
%%%%%%%%%%%%%%%%%%%%%%%%%%%%%%%%%%%%%%%%%%%%%%%%%%%%%%%%%%%%%%%%%%%%%%%%%%%%%%%
%%%%%%%%%%%%%%%%%%%%%%%%%%%%%%%%%%%%%%%%%%%%%%%%%%%%%%%%%%%%%%%%%%%%%%%%%%%%%%%

\section*{Introduction}

%%%%%%%%%%%%%%%%%%%%%%%%%%%%%%%%%%%%%%%%%%%%%%%%%%%%%%%%%%%%%%%%%%%%%%%%%%%%%%%

\subsection{Manifold Approximate Fibrations}
\label{subsec:Manifold_Approximate_Fibrations}

A Manifold Approximate Fibration (MAF) is a map between closed
topological manifolds which is an approximate fibration in the sense
that it has an ``approximate lifting property'', generalizing the
usual lifting property of Hurewicz fibrations, see Definition~\ref{def:MAF}.
This notion of an approximate fibration was introduced by 
Coram-Duvall~\cite{Coram-Duvall(1977)}.  
The theory of MAFs is a
``bundle theory'', see~\cite{Hughes-Taylor-Williams(1990)}, which plays a prominent role in the study of topological 
manifolds and controlled topology.
For instance, 
Edwards~\cite{Edwards(2009)} proved in the early 1970s that a locally flat submanifold of a
topological manifold of dimension greater than five has a mapping
cylinder neighborhood, where the projection map is a {MAF}.

The investigations of group actions on topological manifold
of Quinn~\cite{Quinn(1982a),Quinn(1987),Quinn(1988)} 
rely on the theory of approximate fibrations. Further papers dealing with
approximate fibrations are
\cite{Chapman-Ferry(1983),Ferry(1977a),Hughes(1997),Hughes-Khan(2010),
Hughes(1985mani), Hughes(1985spaces),
Hughes(1988),Hughes-Taylor-Williams(1991),Hughes-Taylor-Williams(1995),
Lacher(1977)}. 

A map $p\colon M\to B$ between closed manifolds is a MAF if and only if the 
preimage of
any open ball in $B$ has the weak homotopy type of the homotopy fiber of $p$. It 
is also
true that the restriction of a MAF to an open subset of the base space is again 
an
approximate fibration, see~\cite[Corollary~12.14 and 
Theorem~12.15]{Hughes-Taylor-Williams(1990)}.

In this paper we treat the following question: 
\begin{quote}
Given a closed
aspherical manifold $B$ and a  map $p\colon M\to B$ from some closed
connected manifold, when is $p$ homotopic to {MAF}?
\end{quote}

We emphasize that we do not want to change the source or target, only the map 
within its homotopy class.

%%%%%%%%%%%%%%%%%%%%%%%%%%%%%%%%%%%%%%%%%%%%%%%%%%%%%%%%%%%%%%%%%%%%%%%%%%%%%%%

\subsection{Two appetizers}
\label{subsec:Two_appetizers}

As an illustration we present two easy to formulate prototypes of more general 
results we
will later prove.

\begin{theorem}[Special case I] \label{the:main_easy_to_state_result_I} Let $B$ 
be a closed connected aspherical PL manifold with hyperbolic fundamental group, e.g., a   
  closed connected smooth manifold with Riemannian metric of negative sectional 
curvature. Let
  $M$ be a closed connected manifold of dimension $\neq 4$. Suppose $\pi_1(M)$ 
is
  torsionfree.  Moreover, assume that $\pi_1(M)$ is a hyperbolic group, 
CAT(0)-group, a
  solvable group, arithmetic group, or a lattice in an almost connected Lie
  group. (Actually, we only need that $\pi_1(M)$ satisfies the $K$- and 
$L$-theoretic
  Farrell-Jones Conjecture.)

  Then a map $M \to B$ is homotopic to a MAF if and only if the homotopy fiber 
of $p$ is
  finitely dominated.
\end{theorem}

\begin{theorem}[Special case II] \label{the:main_easy_to_state_result_II}
  Let $B$ be an aspherical closed connected PL manifold with 
hyperbolic fundamental group, e.g., a   
  closed connected smooth manifold with Riemannian metric of negative sectional 
curvature. 
  Let $M$ be a closed aspherical manifold of dimension
  $\neq 4$. Let $p \colon M \to B$ be a map such
that the kernel of  $\pi_1(p) \colon \pi_1(M) \to \pi_1(B)$ is 
poly-cyclic and its image has finite index.
  
  Then $p \colon M \to B$ is homotopic to a \text{MAF}.
\end{theorem}

%%%%%%%%%%%%%%%%%%%%%%%%%%%%%%%%%%%%%%%%%%%%%%%%%%%%%%%%%%%%%%%%%%%%%%%%%%%%%%%

\subsection{The Naive Conjecture}
\label{subsec:The_Naive_Conjecture}

\emph{Throughout this introduction, we will make the following standing assumptions:}

$p\colon M\to B$ is a continuous map, where
\begin{enumerate}
\item $M$ is a connected closed (topological) manifold;
\item $B$ is an aspherical  closed (topological) manifold which admits a PL structure
  (e.g., is smooth);
\item The homotopy fiber of $p$ is homotopy finite, i.e., has the homotopy type of a finite CW complex;
\item $p$ induces a surjection on fundamental groups.
\end{enumerate}
(But see also Section~\ref{sec:non_finite_fiber} for a discussion of the case of 
a non-finite fiber, in particular see
Theorem~\ref{the:main_result_3}.)

In this situation, we will show in 
Theorem~\ref{thm:factorization_to_an_approximate_fibration} that
there exists a factorization up to homotopy
\begin{equation}\label{diag:factorization_property}
  \xymatrix{M \ar[rr]^f \ar[rd]_p && E \ar[ld]^q \\ & B}
\end{equation}
where $f$ is a homotopy equivalence, $q$ is an approximate fibration,
and $E$ is a compact ENR.

Given such a factorization~\eqref{diag:factorization_property}, denote by 
$\tau(f)\in
\Wh(\pi_1(E))$ the Whitehead torsion of the homotopy equivalence $f$.  Its image 
in the
cokernel $\NWh(p)$ of the assembly map
\[
H_1^{\pi}(\widetilde{B},\Wh(p))   \to \pi_1(\bfWh(M)) = \Wh(\pi_1(M)),
\]
see Section~\ref{sec:Assembly_maps}, depends only on the homotopy class of $p$. 
We call it \emph{tight torsion} and denote it by
\[
N\tau(p)\in \NWh(p).
\]
See Section~\ref{sec:tight_torsion} for more details and for the relation to
the fibering obstruction $\tau_{\fib}(p)$ defined by the authors 
in~\cite{Farrell-Lueck-Steimle(2010)}.

Clearly, if $p$ is homotopic to a MAF, then $N\tau(p)=0$. It is implicit
in~\cite{Farrell(1967)} that if $B=S^1$ and $\dim(M)\geq 6$, then the converse 
is also
true.  From this example, one may be tempted to propose the following

\begin{conjecture}[Naive Conjecture]\label{con:Naive_Conjecture}
Under the above assumptions (i) to (iv) , the map $p$ is homotopic to a MAF if 
and only if $N\tau(p)=0$.
\end{conjecture}

%%%%%%%%%%%%%%%%%%%%%%%%%%%%%%%%%%%%%%%%%%%%%%%%%%%%%%%%%%%%%%%%%%%%%%%%%%%%%%%

\subsection{The main results}
\label{subsec:The_main_results}

This Naive Conjecture is wrong in general: In 
Section~\ref{sec:counterexample_to_Naive_Conjecture} we give a
counterexample where $B$ is a product of the Klein bottle with the
circle $S^1$.  
However, the Naive Conjecture holds in remarkable
generality:

\begin{theorem}[Hyperbolic or CAT(0)-fundamental group and the Klein bottle 
condition]
\label{the:main_result_1}
In addition to the standing assumptions above, assume the following conditions about $\pi_1(B)$:

  \begin{enumerate} 
  \item The group $\pi_1(B)$ is hyperbolic. (This is the case if
  $B$ is a connected  smooth  manifold with Riemannian metric of  negative 
sectional curvature);

  \item The group $\pi_1(B)$ is a CAT(0)-group. (This is the case if
  $B$ is a connected   closed smooth  manifold with Riemannian metric of  
non-positive sectional curvature).
  Moreover, $\pi_1(B)$ satisfies the ``Klein bottle
  condition'' that it does not contain the fundamental group $K=\IZ\rtimes \IZ$ 
of the Klein bottle.
  
  \end{enumerate}
  Suppose $\dim(M) \not= 4$. 

  Then the Naive Conjecture is true.\footnote{In the special case where $B$ is a (real)
hyperbolic manifold, $\dim M>\dim B+4$, and $\Wh(\pi_1(M)\times \IZ^n)=0$ for
  all $n\geq 0$, the conclusion of Theorem~\ref{the:main_result_1} was
  proven by Farrell-Jones in~\cite[Theorem~10.7]{Farrell-Jones(1989)}.} 
\end{theorem}

   We show in Proposition~\ref{prop:klein_bottle_and_locally_symmetric_space}
  that if $B$ is a non-positively curved
locally symmetric space, then up to passage to a finite cover, the
Klein bottle condition is also satisfied.

In the case where the Klein bottle condition fails, we are able to
identify additional obstructions. Recall that the Whitehead group
$\Wh(\pi_1(M))$ carries a $w_1(M)$-twisted involution $a\mapsto \overline{a}$. 
It 
induces an involution on $\NWh(p)$. Recall also that
if $A$ is an abelian group with involution, and $s\in \IZ$, then the
$s$-th Tate cohomology group of $\IZ/2$ with coefficients in $A$ is
defined by
\[
\widehat H^s(\IZ/2;A)=\frac{ \{x\in A \;\vert\; x=(-1)^s 
\cdot\overline{x}\}}{\{x+(-1)^s \cdot \overline{x}\;\vert\; x\in A\}}.
\]

\begin{theorem}[Non-positive sectional curvature in general] 
\label{the:main_result_2}

In addition to the standing assumptions above, assume that  $B$ is smooth and 
admits a Riemannian metric with non-positive sectional curvature.
Suppose 
   $n = \dim M\neq 4$ and  $N\tau(p)=0$.

 Then there is an
  integer $s$ and a sequence of obstructions
  \[
  \kappa_i\in \widehat H^{n+i}(\IZ/2;\NWh(p\times\id_{T^i})),\quad (i=0,\dots, 
s)
  \]
  where $\kappa_{i-1}$ is defined whenever $\kappa_i$ vanishes.  The map $p$ is 
homotopic
  to a MAF if and only if all the $\kappa_i$ vanish.
\end{theorem}

Notice that there are many more closed aspherical manifolds whose
fundamental groups are hyperbolic or CAT(0) respectively than there
are closed connected smooth manifolds with Riemannian metric of
negative or non-positive respectively sectional  curvature,
see~\cite{Davis-Januszkiewicz(1991)}. The reason why in
Theorem~\ref{the:main_result_2} the existence of a smooth structure
together with a Riemannian metric of non-negative sectional curvature
occurs is that we need for its proof to know the MAF-Rigidity Conjecture, see
Section~\ref{sec:Proof_of_Theorem_ref(the:obstruction_to_destabilization)}.

Weinberger~\cite[14.4.4 on page~263]{Weinberger(1994)}  conjectured that if the 
homotopy
fiber of $p$ is homotopy finite, the map $p$ approximately fibers
``modulo a Nil-obstruction''. Theorem~\ref{the:main_result_2} actually
shows that there is a number of different Nil-obstructions whose
vanishing, in the non-positively curved case, is necessary and
sufficient for $p$ to approximately fiber. In many situations, the number $s$ 
can be specified, 
see Remark~\ref{rem:number_of_stabilizations_needed}.

The reason why the $\kappa_i$-obstructions do not show up in 
Theorem~\ref{the:main_result_1} is that the Klein bottle condition implies
the vanishing of the corresponding Tate cohomology groups.
On the other hand Theorem~\ref{the:main_result_1} is a consequence of a  more 
general statement. To
formulate it, we introduce the following terminology:

\begin{definition}[Orientable cyclic subgroups] 
 \label{def:orientable_cyclic_subgroups}
  Given a torsionfree group $G$, we say that \emph{the cyclic subgroups of $G$ 
are
  orientable} if there is a choice $g_C$ of a generator for each non-trivial
  cyclic subgroup $C$, such that whenever $f\colon C\to C'$ is an
  inclusion or a conjugation by some element of $G$, the element
  $f(g_C)$ is a positive power of $g_{C'}$.
\end{definition}

\begin{remark}[FJC]\label{not:FJC}
We will say that a group \emph{$G$ satisfies FJC} if it satisfies
both the $K$- and $L$-theoretic Farrell-Jones Conjecture with  coefficients in 
additive categories.
A discussion of the FJC and a description of groups for which it has been proven 
is given in 
Section~\ref{sec:Brief_review_of_the_Farrell-Jones_Conjecture_with_coefficients_in_additive_categories}.
These include hyperbolic groups, CAT(0)-groups, solvable groups, arithmetic 
groups, and 
lattices in almost connected Lie groups.
\end{remark}

\begin{theorem}[Aspherical base manifold, FJC, and orientable cyclic subgroups]
\label{the:main_result_1_generalized}
  If $\dim M\neq 4$, the Naive Conjecture is true, provided $\pi_1(B)$ satisfies 
FJC
  and the cyclic subgroups of $\pi_1(B)$ are orientable. 
\end{theorem}

Since the Naive Conjecture is wrong in general, an immediate question
is whether there is a geometric interpretation of the vanishing of the
tight torsion $N\tau(p)$. We are grateful to
Bruce Williams for suggesting to us an interpretation in terms of
$Q$-manifold theory. In fact the following holds, where
$Q=\prod_{i=1}^\infty I$ is the Hilbert cube:

\begin{theorem}[Hilbert cube 
version]\label{the:interpretation_in_terms_of_Q_manifolds}
  Under the standing assumptions above, assume that $N\tau(p)=0$. 
  Then the composite $M\times Q\to B$ of the projection
  map $M\times Q\to M$ with $p$ is homotopic to an approximate
  fibration. 
\end{theorem}

%%%%%%%%%%%%%%%%%%%%%%%%%%%%%%%%%%%%%%%%%%%%%%%%%%%%%%%%%%%%%%%%%%%%%%%%%%%%%%%

\subsection{Block bundles}
\label{subsec:Block_bundles}

Every block bundle is homotopic to a MAF, see 
Quinn~\cite[Lemma~3.3.1]{Quinn(1979a)}.
The converse is not true: In Section~\ref{sec:An_example_on_block_fibering} 
we give  an example of a map to a $2$-torus which
satisfies the conditions of the Naive Conjecture and therefore
approximately fibers, but does not block fiber. However, if certain middle and 
lower
$K$-groups vanish, then by~\cite[Theorem~3.3.2]{Quinn(1979a)} a MAF is 
homotopic to a block bundle, provided the difference $\dim(M) - \dim(B)$  is 
greater or equal to five.  

The lower $K$-groups under consideration vanish in the situation
of Theorem~\ref{the:main_easy_to_state_result_I} and 
Theorem~\ref{the:main_easy_to_state_result_II}, essentially
since there $\pi_1(M)$ is torsionfree and satisfies FJC by assumption.
Hence we can replace in Theorem~\ref{the:main_easy_to_state_result_I} 
and Theorem~\ref{the:main_easy_to_state_result_II}
MAF by the stronger conclusion that $p$ is homotopic to a block bundle,
provided the difference $\dim(M) - \dim(B)$  is greater or equal to five.

%%%%%%%%%%%%%%%%%%%%%%%%%%%%%%%%%%%%%%%%%%%%%%%%%%%%%%%%%%%%%%%%%%%%%%%%%%%%%%%

\subsection{Locally trivial bundles}
\label{subsec:locally_trivial_bundles}

Of course it would be desirable if  $p$ is homotopic to the projection of a 
locally trivial bundle 
which is a very restrictive condition. This problem leads to new obstructions in 
the \emph{higher} Whitehead groups
and one can hope for a satisfactory answer under  certain conditions on the 
dimensions of $B$ and $M$.

There is one favorite case. Namely, if $f \colon M \to B$ is a homotopy 
equivalence
of closed aspherical manifolds of dimension $\not= 4$ and $\pi_1(B)$ satisfies 
FJC, then
the Borel Conjecture is known to be true that predicts that 
$f$ is homotopic to a homeomorphism and in particular to the projection of a 
fiber bundle.  Modulo the 
orientability assumption and in the high-dimensional case, this also follows 
from Theorem~\ref{the:main_result_1_generalized}, because a manifold approximate 
fibration which is a homotopy equivalence is a CE map, which was shown to be 
homotopic to a homeomorphism by Siebenmann~\cite{Siebenmann(1972)} in dimensions at 
least 5. 

More generally, if $M$ is a closed aspherical manifold of dimension $\ge 5$ such 
that
$\pi_1(M)$ is a direct product $G_1 \times G_2$ and the cohomological dimension 
of $G_i$
is not equal to $3,4,5$ for $i = 1,2$, then there are closed aspherical 
manifolds $M_1$ and $M_2$ with
$G_i = \pi_1(M_i)$ together with a homeomorphism $f \colon M \xrightarrow{\cong} 
M_1\times
M_2$ implementing on fundamental groups the given isomorphism $\pi_1(M) \cong 
G_1 \times
G_2$, see~\cite[Theorem~6.1]{Lueck(2010asph)}. The manifolds $M_i$ are unique up 
to
homeomorphism.  This implies (by inspecting the proof to  make sure that 
$\dim(N) = 3,4,5$ is not needed)
for a $\pi_1$-surjective map $p \colon M \to N$ of closed
aspherical manifolds that $p$ is homotopic to the projection of a trivial 
bundle, if
$\dim(M) \ge 5$ and $\dim(M) - \dim(N) \not= 3,4,5$, the inclusion of
the kernel of $\pi_1(p) \colon \pi_1(M) \to \pi_1(N)$ into $\pi_1(M)$ has a 
retraction, and
$\pi_1(M)$ satisfies FJC.

%%%%%%%%%%%%%%%%%%%%%%%%%%%%%%%%%%%%%%%%%%%%%%%%%%%%%%%%%%%%%%%%%%%%%%%%%%%%%%%

\subsection{Acknowledgements}
\label{subsec:Acknowledgments}
 This paper is financially supported by the Leibniz-Preis of the second
author, granted by the Deutsche Forschungsgemeinschaft {DFG}. Moreover the first named author was partially supported by NSF grant DMS-1206622. The paper is financially supported by the ERC 
Advanced Grant ``KL2MG-interactions'' (no.
662400) of the second author granted by the European Research Council.
The third named author was partly supported by the 
ERC Advanced Grant 288082 and by the Danish National Research Foundation through 
the Centre for Symmetry and Deformation (DNRF92). We thank the referee
for his detailed report and very valuable comments.

\tableofcontents

%%%%%%%%%%%%%%%%%%%%%%%%%%%%%%%%%%%%%%%%%%%%%%%%%%%%%%%%%%%%%%%%%%%%%%%%%%%%%%%
%%%%%%%%%%%%%%%%%%%%%%%%%%%%%%%%%%%%%%%%%%%%%%%%%%%%%%%%%%%%%%%%%%%%%%%%%%%%%%%
%%%%%%%%%%%%%%%%%%%%%%%%%%%%%%%%%%%%%%%%%%%%%%%%%%%%%%%%%%%%%%%%%%%%%%%%%%%%%%%

\section{Organization of the proofs}\label{sec:organization_of_the_proofs}

In this section, after recalling some definitions, we describe our main 
stabilization-destabilization technique. This technique is essentially the 
combination of Theorem~\ref{the:theorem_is_stably_true} and
Theorem~\ref{the:stably_implies_unstably} whose proof will be given in later 
sections. In this section we explain how to derive the results stated in the 
introduction from this technique.

>From a logical perspective, parts of this section 
should be read at the very end of the article. We decided to place it here 
because it summarizes the main strategy.

%%%%%%%%%%%%%%%%%%%%%%%%%%%%%%%%%%%%%%%%%%%%%%%%%%%%%%%%%%%%%%%%%%%%%%%%%%%%%%%

\subsection{Review of MAF}
\label{subsec:Review_of_MAF}

We first recall the definition of an $\eps$-fibration, approximate fibration 
and MAF.

\begin{definition}[$\eps$-fibration]
  \label{def:epsilon-fibration}
  Let $E$ be a topological space, $B$ be a metric space, and $p \colon E \to B$ be 
  a continuous map. We call
  $p$ an \emph{$\eps$-fibration}, if for any homotopy $h \colon X \times 
[0,1] \to B$ and map
  $f \colon X \to E$ with $p \circ f = h_0$ for $h_0(x) := h(x,0)$ there is a 
map $H
  \colon X \times [0,1] \to E$ such that $H_0 = f$ and for all $(x,t) \in X 
\times [0,1]$ we have
  $d(p\circ H(x,t),h(x,t)) \le \eps$, in other words, we can solve the 
lifting
  problem below up to an error bounded by $\eps$ when measured in $B$.
  \[
  \xymatrix{X \ar[r]^-f \ar[d]_{i_0} & E \ar[d]^p
    \\
    X \times [0,1] \ar[r]_-h \ar@{.>}[ur]^{H} & B }
  \]
\end{definition}

An $\eps$-fibration for $\eps = 0$ is the same as a (Hurewicz) fibration 
in the classical sense.

\begin{definition}[Approximate fibration]
  \label{def:approximate_fibration}
  Let $p\colon E\to B$ be a continuous map of topological spaces with $B$ compact metric. 
  We call
  $p$ an \emph{approximate fibration} if it is an $\eps$-fibration for every $\eps > 0$.
\end{definition}

This definition is independent of the choice of metric on $B$.
If one wants to consider
non-compact base spaces, one should use a slightly more complicated definition in terms of open coverings.

\begin{definition}[MAF]
  \label{def:MAF}
  A \emph{Manifold Approximate Fibration} (MAF) is an approximate fibration with 
closed manifolds
  as source and target.
\end{definition}

%%%%%%%%%%%%%%%%%%%%%%%%%%%%%%%%%%%%%%%%%%%%%%%%%%%%%%%%%%%%%%%%%%%%%%%%%%%%%%%

\subsection{Low dimensions}
\label{subsec:Low-dimensions}

The results of the introduction are true by inspection if the dimension of $M$ 
is 1 or
2. In Section~\ref{sec:Three-dimensional_manifolds}, which is independent of the 
rest
of the paper, we treat the case where $M$ has dimension three.

%%%%%%%%%%%%%%%%%%%%%%%%%%%%%%%%%%%%%%%%%%%%%%%%%%%%%%%%%%%%%%%%%%%%%%%%%%%%%%%

\subsection{Stabilization-Destabilization Strategy}
\label{subsec:Stabilization-Destabilization_Strategy}

The following line of
argument applies in the high-dimensional case $\dim(M)\geq 5$.

\begin{notation}
  Consider a map $p \colon M \to B$. Let $T^s$ denote the $s$-torus, i.e., 
$T^s=S^1\times\dots\times S^1$
  ($s$ factors), and
  \[
   p_s\colon M\times T^s\to B\times T^s
   \] 	
   be the map $p\times\id_{T^s}$. 
\end{notation}

The following is the starting point for our results. It will be proven
in Section~\ref{sec:stably_true_finite_fiber} for homotopy finite homotopy fiber 
and finally for 
finitely dominated homotopy fiber in 
Subsection~\ref{subsec:stably_true_in_general}.

\begin{theorem}[Stabilization Theorem]\label{the:theorem_is_stably_true}
  Let $p \colon M \to B$ be a $\pi_1$-surjective 
  map of closed manifolds such 
that $B$ is PL and 
  aspherical and the homotopy fiber $F_p$ is finitely dominated. Suppose that 
the $L$-theoretic FJC  is true for the group
  $\pi_1(B)$. 

  Then there  exists a natural number $s$ such that $p_s\colon M\times T^s\to 
B\times T^s$
  is homotopic to a MAF.
\end{theorem}

Theorem~\ref{the:theorem_is_stably_true} suggests the following
strategy to obtain an approximate fibration $M\to B$ homotopic to $p$:
First choose an approximate fibration $M\times T^s\to B\times T^s$
homotopic to $p_s$ and try to ``split''. The next theorem yields  a sufficient
condition to get from $s$ one step down to $(s-1)$. It will be proven in 
Section~\ref{sec:Proof_of_the_Destabilization_Theorem_ref(the:stably_implies_unstably)}.

\begin{theorem}[Splitting Theorem]\label{the:stably_implies_unstably}
   Let $p \colon M \to B$ be a  $\pi_1$-surjective
   map of closed manifolds such 
that $B$ is PL and 
  aspherical.
  \begin{enumerate}
  \item \label{the:stably_implies_unstably:fac_prop} 
    Suppose that $n=\dim M\geq 5$, the $K$-theoretic FJC 
   is true for the group $\pi_1(B)$ and $p_1\colon M\times  S^1\to B\times S^1$ 
is homotopic to a MAF.

   Then  $p$ is $h$-cobordant to $q\colon N\to B$ where $q$ is a MAF.
  
  \item \label{the:stably_implies_unstably_approx}
  If, additionally, $N\tau(p)=0$ and $\widehat H^n(\IZ/2;\NWh(\pi_1
    M))=0$, then $p\colon M\to B$ is homotopic to a MAF.
  \end{enumerate}
\end{theorem}

%%%%%%%%%%%%%%%%%%%%%%%%%%%%%%%%%%%%%%%%%%%%%%%%%%%%%%%%%%%%%%%%%%%%%%%%%%%%%%%

\subsection{Proof of Theorem~\ref{the:main_result_1_generalized}}
\label{subsec:Proof_of_Theorem_ref(the:main_result_1_generalized)}

  We need to show that if $N\tau(p)=0$ and $\dim(M)\geq 5$, 
  then $p$ is homotopic to a MAF. 
   
    By Theorem~\ref{the:theorem_is_stably_true} the map
  \[
   p_s\colon M\times T^s\to B\times T^s
   \]
  is homotopic to a MAF for some $s\geq 1$.    Now we want to apply 
  Theorem~\ref{the:stably_implies_unstably} in the case, where $M$, $p$ and
  $p_1$ are specialized to be $M\times T^{s-1}$, $p_{s-1}$ and $p_s$
  respectively. 
 
   The orientability assumption for $\pi_1(B)$ implies the orientability
   assumption for $\pi_1(B\times T^{n-1})$, see 
Lemma~\ref{lem:extensions_and_orientable_subgroups}.
  As mentioned in the introduction,
  this orientability assumption implies that the Tate
  cohomology groups $\widehat H^{n+s-1}(\IZ/2;\NWh(\pi_1(M \times T^{s-1})))$ 
  vanish, see~Theorem~\ref{the:FJC_and_Tate_cohomology}.

  Let $p\simeq q\circ f$ be a factorization as provided by 
Theorem~\ref{thm:factorization_to_an_approximate_fibration}. Then
  \[
  q_{s-1}:=q\times\id_{T^{s-1}}\colon E\times T^{s-1}\to B\times
  T^{s-1}
  \] 
   is also an approximate fibration. Also
  \[
   f_{s-1}:=f\times\id_{T^{s-1}}\colon M\times T^{s-1}\to E\times
  T^{s-1}
   \] 
   is a homotopy equivalence and $p_{s-1}\simeq q_{s-1}\circ
  f_{s-1}$. So we can compute $N\tau(p_{s-1})$ from 
  $\tau(f_{s-1})=\chi(T^{s-1})\,\tau(f)$.

  There are two cases to consider: $s-1>0$ and $s-1=0$. When $s-1>0$,
  $\chi(T^{s-1})=0$ and hence $N\tau(p_{s-1})=0$. When $s-1=0$,
  $p_{s-1}=p$ and $N\tau(p_{s-1})=0$ by assumption. Hence all conditions of 
  Theorem~\ref{the:stably_implies_unstably} hold in both cases, and we
  conclude that $p_{s-1}$ is homotopic to an approximate
  fibration. Continuing this argument inductively, we finally arrive
  at the desired result, namely $p\colon M\to B$ is homotopic to an
  approximate fibration. This finishes the proof of 
Theorem~\ref{the:main_result_1_generalized}.

%%%%%%%%%%%%%%%%%%%%%%%%%%%%%%%%%%%%%%%%%%%%%%%%%%%%%%%%%%%%%%%%%%%%%%%%%%%%%%%

\subsection{Proof of Theorem~\ref{the:main_result_1}}
\label{subsec:Proof_of_Theorem_ref(the:main_result_1)}
In the situation of Theorem~\ref{the:main_result_1} the fundamental group 
$\pi_1(B)$ satisfies
FJC by Theorem~\ref{the:status_of_FJC} and 
has  orientable cyclic subgroups by 
Theorem~\ref{the:fundamental_groups_of_npc_manifolds_are_orientable}.
Hence Theorem~\ref{the:main_result_1} follows from 
Theorem~\ref{the:main_result_1_generalized}.

%%%%%%%%%%%%%%%%%%%%%%%%%%%%%%%%%%%%%%%%%%%%%%%%%%%%%%%%%%%%%%%%%%%%%%%%%%%%%%%

\subsection{Proof of Theorem~\ref{the:main_result_2}}
\label{subsec:Proof_of_Theorem_ref(the:main_result_2)}

The statement of Theorem~\ref{the:stably_implies_unstably} becomes
wrong if we drop the assumption on $\widehat H^n(\IZ/2, \NWh(p))$;
see Section~\ref{sec:counterexample_to_Naive_Conjecture} for a
counterexample. However, in the case where $B$ is non-positively curved
(or more generally, for all manifolds $B$ that satisfy the MAF
Rigidity Conjecture~\ref{con:MAF_rigidity_conjecture}), we
are able to identify a specific element
\[
\kappa_0\in \widehat H^n(\IZ/2, \NWh(p)),
\]
the \emph{splitting obstruction}, such that:

\begin{theorem}[MAF-Rigidity Conjecture and the $\kappa_0$-obstruction]
\label{the:obstruction_to_destabilization}
  Suppose that the MAF Rigidity Conjecture holds for $B$. In the
  situation of Theorem~\ref{the:stably_implies_unstably}, suppose that
  the conditions of (i) hold, and that $N\tau(p)=0$. Then $\kappa_0=0$
  if and only if $p$ is homotopic to an approximate fibration.
\end{theorem}

See 
Section~\ref{sec:Proof_of_Theorem_ref(the:obstruction_to_destabilization)} for a 
proof of this
result. It implies Theorem~\ref{the:main_result_2} in the same way as
Theorem~\ref{the:main_result_1_generalized} followed from 
Theorem~\ref{the:stably_implies_unstably} (and recalling that the FJC is 
known to hold in this situation).

%%%%%%%%%%%%%%%%%%%%%%%%%%%%%%%%%%%%%%%%%%%%%%%%%%%%%%%%%%%%%%%%%%%%%%%%%%%%%%%

\begin{remark}\label{rem:pi_1_surjective}
  In practice, the $\pi_1$-surjectivity condition on $p\colon M\to B$ is not
  restrictive. Indeed, if the homotopy fiber is finitely dominated, then the image of
  $\pi_1(M)$ in $\pi_1(B)$ has finite index so $p$ lifts along a finite covering
  $\overline B\to B$ of $B$, such that the new map $p'\colon M\to \overline B$ is
  $\pi_1$-surjective. It follows from the definitions that $p$ is a MAF if and only if
  $p'$ is a MAF.
\end{remark}

\subsection{Proof of Theorem~\ref{the:main_easy_to_state_result_I}}
\label{subsec:Proof_of_Theorem_ref(the:main_easy_to_state_result_I)}

Using Remark~\ref{rem:pi_1_surjective}, we can reduce to the special case where $p$ is $\pi_1$-surjective.
We want to apply 
Theorem~\ref{the:main_result_1} 
and therefore have to 
show that all conditions are satisfied. By Theorem~\ref{the:status_of_FJC}  the
$K$-theoretic FJC holds for $\pi_1(M)$, for $\pi_1(M) \times \IZ$, and for
$\pi_1(F_p)$. 
Notice that the Farrell-Jones Conjecture with coefficients in an additive
category
implies the original Farrell-Jones Conjecture.  By assumption $\pi_1(M)$ and
therefore also $\pi_1(M) \times \IZ$ and $\pi_1(F_p)$ are torsionfree.
Hence the Whitehead groups  $\Wh(\pi_1(M))$ and $\Wh(\pi_1(M \times S^1))$ and
the reduced projective class group $\widetilde{K}_0(\IZ[\pi_1(F_p)])$ vanish, 
see
for 
instance~\cite[Conjecture~1.1 on page~652, Conjecture~1.3 on page~653,
Corollary~2.3 on page~685]{Lueck-Reich(2005)}. 
Therefore the finiteness obstruction of $F_p$ is zero and hence $F_p$ is
homotopy equivalent to a finite $CW$-complex. Moreover $N\tau(p)$ vanishes. We
conclude $K\nsubseteq \pi_1(B)$ from the fact
that a hyperbolic group does not contain $\IZ \oplus \IZ$ as subgroup. Hence, by
Theorem~\ref{the:main_result_1}, $p$ is homotopic to a {MAF}.
This finishes the proof of 
Theorem~\ref{the:main_easy_to_state_result_I}.

%%%%%%%%%%%%%%%%%%%%%%%%%%%%%%%%%%%%%%%%%%%%%%%%%%%%%%%%%%%%%%%%%%%%%%%%%%%%%%%

\subsection{Proof of Theorem~\ref{the:main_easy_to_state_result_II}}
\label{subsec:Proof_of_Theorem_ref(the:main_easy_to_state_result_II)}

Again by Remark~\ref{rem:pi_1_surjective}, we only need to consider the case where $p$ is $\pi_1$-surjective.
Since $M$ and $B$ are aspherical, $F_p$ is aspherical with a torsionfree 
poly-cyclic
fundamental group. Any torsionfree poly-cyclic group has a finite model for its
classifying space since it is an iterated extension by an infinite cyclic group. 
 Hence
$F_p$ is homotopy equivalent to a finite $CW$-complex. Since an extension of a 
poly-cyclic
group by an infinite cyclic group is again poly-cyclic, we conclude from
Theorem~\ref{the:status_of_FJC} that $\pi_1(M)$ and $\pi_1(M) \times \IZ$ 
satisfies
{FJC}. Now proceed as in the proof of 
Theorem~\ref{the:main_easy_to_state_result_I}.

%%%%%%%%%%%%%%%%%%%%%%%%%%%%%%%%%%%%%%%%%%%%%%%%%%%%%%%%%%%%%%%%%%%%%%%%%%%%%%%

\subsection{Proof of Theorem~\ref{the:interpretation_in_terms_of_Q_manifolds}}
\label{subsec:Proof_of_Theorem_ref(the:interpretation_in_terms_of_Q_manifolds)_announced}

The proof of Theorem~\ref{the:interpretation_in_terms_of_Q_manifolds}, finally, 
will be given 
Subsection~\ref{subsec:Proof_of_Theorem_ref(the:interpretation_in_terms_of_Q_manifolds)}.
It does not depend on the stabilization-destabilization technique.

To summarize, we still have to prove Theorem~\ref{the:theorem_is_stably_true},
Theorem~\ref{the:stably_implies_unstably}, 
Theorem~\ref{the:obstruction_to_destabilization},
and Theorem~\ref{the:interpretation_in_terms_of_Q_manifolds}.

%%%%%%%%%%%%%%%%%%%%%%%%%%%%%%%%%%%%%%%%%%%%%%%%%%%%%%%%%%%%%%%%%%%%%%%%%%%%%%%
%%%%%%%%%%%%%%%%%%%%%%%%%%%%%%%%%%%%%%%%%%%%%%%%%%%%%%%%%%%%%%%%%%%%%%%%%%%%%%%
%%%%%%%%%%%%%%%%%%%%%%%%%%%%%%%%%%%%%%%%%%%%%%%%%%%%%%%%%%%%%%%%%%%%%%%%%%%%%%%

\section{Brief review of the Farrell-Jones Conjecture\\with coefficients in 
additive categories}
\label{sec:Brief_review_of_the_Farrell-Jones_Conjecture_with_coefficients_in_additive_categories}

The original statement of the Farrell-Jones Conjecture can be found in~\cite[1.6 
on page 257]{Farrell-Jones(1993a)}.  
We will focus on the formulation with coefficients in additive categories as it 
is for instance
given for $K$-theory in~\cite[Conjecture~2.3]{Bartels-Reich(2007coeff)} and 
for $L$-theory in~\cite[Definition~0.2]{Bartels-Lueck(2009coeff)}.

\begin{definition}[$K$-theoretic Farrell-Jones Conjecture]
\label{def:K-theoretic_Farrell-Jones_Conjecture}
  A group $G$ satisfies the \emph{$K$-theoretic Farrell-Jones Conjecture with 
coefficients in additive categories} 
  if     for any additive $G$-category $\cala$ the assembly map
\begin{eqnarray*}
& \asmb^{G,\cala}_n \colon H_n^G\bigl(\EGF{G}{\VCyc};\bfK_{\cala}\bigr) \to
 H_n^G\bigl(\pt;\bfK_{\cala}\bigr)
= K_n\left(\intgf{G}{\cala}\right)
&
\end{eqnarray*}
induced by the projection $\EGF{G}{\VCyc} \to \pt$ is bijective for all $n \in 
\IZ$.
\end{definition}

\begin{definition}[$L$-theoretic Farrell-Jones Conjecture]
\label{def:L-theoretic_Farrell-Jones_Conjecture}
  A group $G$ satisfy the \emph{$L$-theoretic Farrell-Jones Conjecture with 
coefficients in additive categories} 
  if   for any additive $G$-category with involution $\cala$ the assembly map
\begin{eqnarray*}
& \asmb^{G,\cala}_n \colon H_n^G\bigl(\EGF{G}{\VCyc};\bfL_{\cala}^{\langle - 
\infty\rangle}\bigr) \to
 H_n^G\bigl(\pt;\bfL_{\cala}^{\langle - \infty\rangle}\bigr)
= L_n^{\langle - \infty\rangle}\left(\intgf{G}{\cala}\right)
&
\end{eqnarray*}
induced by the projection $\EGF{G}{\VCyc} \to \pt$ is bijective for all $n \in 
\IZ$.
\end{definition}

For more information about these conjecture we refer for instance to the survey
article~\cite{Lueck-Reich(2005)}.  In this paper we will use it as a black box 
and it is
not necessary to understand the details of the construction of the assembly maps 
for
equivariant additive categories (with involutions). It is more important to know 
for which
groups these conjectures are known.

\begin{theorem}[Status of the Farrell-Jones Conjecture]
\label{the:status_of_FJC}
The class of groups  for which both the $K$-theoretic and the $L$-theoretic 
Farrell-Jones
Conjectures~\ref{def:K-theoretic_Farrell-Jones_Conjecture} 
and~\ref{def:L-theoretic_Farrell-Jones_Conjecture}
are known has the following properties:
\begin{enumerate}

\item It contains all hyperbolic groups;

\item It contains all CAT(0)-groups;

\item It contains all solvable groups,

\item It contains all lattices in almost connected Lie groups;

\item It contains all arithmetic groups;

\item It  contains all fundamental groups of (not necessarily compact) 
$3$-manifolds (possibly with boundary);

\item It is closed under taking subgroups;

\item It is closed under finite free products and finite direct products;

\item It is closed under directed colimits over directed systems (with arbitrary 
 structure maps);

\item Let $1 \to H \to G \xrightarrow{p} Q \to 1$ be an extension of groups.
Suppose that $Q$ and $p^{-1}(V)$ for any virtually cyclic subgroup $V \subseteq 
Q$ belong to this class of groups,
then also $G$ does.

\end{enumerate}
\end{theorem}
\begin{proof} The proofs can be found 
in~\cite{Bartels-Echterhoff-Lueck(2008colim),
Bartels-Farrell-Lueck(2014),
Bartels-Lueck(2012annals),
Bartels-Lueck-Reich(2008hyper), Bartels-Lueck-Reich-Rueping(2014), 
Kammeyer-Lueck-Rueping(2016),
Rueping(2016_S-arithmetic), Wegner(2012), Wegner(2015)}.
\end{proof}

%%%%%%%%%%%%%%%%%%%%%%%%%%%%%%%%%%%%%%%%%%%%%%%%%%%%%%%%%%%%%%%%%%%%%%%%%%%%%%%
%%%%%%%%%%%%%%%%%%%%%%%%%%%%%%%%%%%%%%%%%%%%%%%%%%%%%%%%%%%%%%%%%%%%%%%%%%%%%%%
%%%%%%%%%%%%%%%%%%%%%%%%%%%%%%%%%%%%%%%%%%%%%%%%%%%%%%%%%%%%%%%%%%%%%%%%%%%%%%%

\section{Assembly maps}
\label{sec:Assembly_maps}

%%%%%%%%%%%%%%%%%%%%%%%%%%%%%%%%%%%%%%%%%%%%%%%%%%%%%%%%%%%%%%%%%%%%%%%%%%%%%%%

\subsection{Fibered assembly maps}\label{subsec:fibered_assembly_maps}

We will use the setup for fibered assembly maps 
from~\cite[Section~11]{Lueck-Steimle(2016splitasmb)}, 
which we recall here for the reader's convenience.

Let
\[
\bfE \colon \Spaces \to \Spectra
\]
be a covariant functor which is homotopy invariant, i.e., sends
weak equivalences of spaces to weak equivalences of spectra. Our main examples 
are
\[
\bfK_{\IZ\Pi}, \bfA, \bfWh, \bfA^\% \colon \Spaces \to \Spectra
\]
which are defined as follows: Given a space $Y$, let $\bfK_{\IZ\Pi}(Y)$ be the 
non-connective $K$-theory
spectrum associated to the additive category of finitely generated projective 
$\IZ \Pi(Y)$-modules,
let $\bfA(Y)$ be the non-connective $A$-theory spectrum associated to $Y$, let 
$\bfWh(Y)$ be the 
associated non-connective PL Whitehead spectrum and let $\bfA^\%(Y)$ be  $Y_+ 
\wedge \bfA(\pt)$. 
By definition there is a cofibration sequence of spectra, natural in $Y$
\begin{eqnarray}
& \bfA^\%(Y)  \xrightarrow{\bfasmb} \bfA(Y) \to \bfWh(Y). &
\label{cofiber_sequence_A_Wh}
\end{eqnarray}
We emphasize that we use non-connective $K$-theory spectra. More information 
about these spectra 
can be found for instance 
in~\cite{Lueck(1989),Lueck-Steimle(2014delooping),Pedersen-Weibel(1985),
Vogell(1990),Vogell(1991)}.

Let $p \colon X \to B$ be a map of path connected spaces which induces an 
epimorphism on fundamental groups.
Suppose that $B$ admits a universal covering $q \colon \widetilde{B} \to B$.

Choose base points $x_0 \in X$, $b_0 \in B$ and $\widetilde{b_0} \in 
\widetilde{B}$
satisfying $p(x_0) = b_0 = q(\widetilde{b_0})$. We will abbreviate $\Gamma = 
\pi_1(X,x_0)$ and
$\pi = \pi_1(B,b_0)$. The free right proper $\pi$-action on $\widetilde{B}$
and $q$ induces a homeomorphism $\widetilde{B}/\pi \xrightarrow{\cong} B$. For 
a subgroup $H \subseteq \pi$ denote by $q(\pi/H) \colon \widetilde{B} 
\times_{\pi} \pi/H = \widetilde{B}/H \to B$
the obvious covering induced by $q$. The pullback construction yields a 
commutative square of spaces
\[
\xymatrix@!C= 7em{
X(\pi/H) \ar[r]^-{\overline{q}(\pi/H)} \ar[d]_{\overline{p}(\pi/H)}
&
X \ar[d]^{p}
\\
\widetilde{B} \times_{\pi} \pi/H \ar[r]_-{q(\pi/H)}
&
B
}
\]
where $\overline q(\pi/H)$ is again a covering. 
This yields covariant functors from the orbit category of $\pi$ to the category 
of topological spaces
\begin{eqnarray*}
\underline{B}  \colon \Or(\pi) & \to & \Spaces, \quad \pi/H \mapsto 
\widetilde{B} \times_{\pi} \pi/H;
\\
\underline{X} \colon \Or(\pi)  & \to & \Spaces, \quad \pi/H \mapsto X(\pi/H).
\end{eqnarray*}

Let
\[
\bfE(p):= \bfE\circ \underline X\colon \Or(\pi)\to \Spectra.
\]
Associated to this functor is a $\pi$-homology theory on the category of 
$\pi$-$CW$-complexes 
\begin{eqnarray}
& H_*^{\pi}(-;\bfE(p))&
\label{H_npi(-,E(p))}
\end{eqnarray}
such that
\begin{eqnarray}
H_n^{\pi}(\pi/H;\bfE(p))  & = & \pi_n\bigl(\bfE(X(\pi/H))\bigr)
\label{H_pi_at_pi/H}
\end{eqnarray}
holds for any subgroup $H \subseteq \pi$
and $n \in \IZ$, see~\cite[Sections~4 and~7]{Davis-Lueck(1998)}.

We will be interested in the two maps
\begin{eqnarray}
H_n^{\pi}(\pr;\bfE(p)) \colon H_n^{\pi}(E\pi;\bfE(p)) & \to & 
H_n^{\pi}(\pt;\bfE(p)) = \pi_n(\bfE(X));
\label{ass_Davis-Lueck_Epi}
\\
H_n^{\pi}(\pr;\bfE(p)) \colon H_n^{\pi}(\widetilde{B};\bfE(p)) & \to & 
H_n^{\pi}(\pt;\bfE(p)) = \pi_n(\bfE(X)),
\label{ass_Davis-Lueck_widetildeB}
\end{eqnarray}
where $\pr$ denotes the projection onto $\pt = \pi/\pi$.

\begin{definition}[$\NWh(p)$] \label{def:NWH(p)}
Let $p \colon X \to B $ be a $\pi_1$-surjective map of path connected spaces. Suppose that $B$ 
admits a universal covering
$q \colon \widetilde{B} \to B$. Define $\NWh_n(p)$
to be the cokernel of the assembly map \eqref{ass_Davis-Lueck_widetildeB} with 
$\bfE=\bfWh$,
\[
H_n^{\pi}(\pr;\bfWh(p)) \colon  H_n^{\pi}(\widetilde{B};\bfWh(p)) \to 
H_n^{\pi}(\pt;\bfWh(p)) = \pi_n(\bfWh(X)).
\]
We abbreviate $\NWh(p) := \NWh_1(p)$.
\end{definition}

\begin{remark}
If $B$ is aspherical, 
we can replace in Definition~\ref{def:NWH(p)} of $\NWh(p)$ the spectrum $\bfWh$ 
by either $\bfA$ or $\bfK_{\IZ\Pi}$
and replace the assembly map~\eqref{ass_Davis-Lueck_widetildeB} 
by~\eqref{ass_Davis-Lueck_Epi}.
The proof of this claim is contained in the proof of 
Theorem~\ref{the:FJC_and_Tate_cohomology}.
\end{remark}

\begin{remark}[Involutions] \label{rem:involutions} Notice that all four 
functors $\bfK_{\IZ\Pi}, \bfA, \bfWh, \bfA^\%$ take values in the category of spectra 
with (strict) involutions. See for instance~\cite{Bartels-Lueck(2009coeff),Vogell(1984),Vogell(1985),
Weiss-Williams(1998)}.
  Hence all the homology groups above come with involutions and all maps are 
compatible
  with them. In particular $\NWh_n(p)$ has an involution.
\end{remark}

We finish this subsection by giving some naturality properties. Given a 
commutative diagram
\[
\xymatrix{X_0 \ar[rd]_-{p_0} \ar[rr]^{f} 
& & 
X_1 \ar[dl]^-{p_1}
\\
& B & 
}
\]
we obtain a natural transformation of homology theories
\[
H_*^{\pi}(-,\bfE(f)) \colon H_*^{\pi}(-,\bfE(p_0)) \to H_*^{\pi}(-,\bfE(p_1))
\]
from the obvious transformation  $\underline{f} \colon \underline{X_0} \to 
\underline{X_1}$ 
of functors $\Or(\pi) \to \Spaces$.

\begin{lemma} \label{lem:weak_equi_f} If $f$ is weak homotopy equivalence, then
  $H_*^{\pi}(-,\bfE(f))$ is a natural equivalence of $\pi$-homology theories, 
i.e.,
  $H_n^{\pi}(Y,\bfE(f)) \colon H_n^{\pi}(Y,\bfE(p_0)) \to 
H_n^{\pi}(Y,\bfE(p_1))$ is
  bijective for any $\pi$-$CW$-complex $Y$ and $n \in \IZ$.
\end{lemma}
\begin{proof}
Each map $\underline{f}(\pi/H)$ is a weak homotopy equivalence since it is a homotopy pullback of the weak homotopy
equivalence $f$.
Hence by assumption each map $\bfE(f)(\pi/H) \colon \bfE(p_0)(\pi/H)  \to 
\bfE(p_1)(\pi/H)$
is a weak homotopy equivalence of spectra. Now 
apply~\cite[Lemma~4.6]{Davis-Lueck(1998)}.
\end{proof}

\begin{corollary}
The groups $H_*^\pi(X, \bfE(p))$ only depend on the homotopy class of $p$.
\end{corollary}

Let $\bfu \colon \bfE_0 \to \bfE_1$ be a transformation of functors $\Spaces \to 
\Spectra$.
It induces in the obvious way a natural transformation on homology theories
\[
H_*^{\pi}(Y,\bfu(p)) \colon H_*^{\pi}(Y,\bfE_0(p)) \to H_*^{\pi}(Y,\bfE_1(p)).
\]

\begin{lemma} \label{lem:weak_in_bfE} Let $k$ be an integer. Suppose that for 
every space $Z$ 
the homomorphisms $\pi_m(\bfu(Z))
  \colon \pi_m(\bfE_0(Z)) \to \pi_m(\bfE_1(Z))$ is bijective for $m < k$ and 
surjective for $m = k$.
  Let $Y$ be any  $\pi$-$CW$-complex.

Then $H_n^{\pi}(Y,\bfu(p)) \colon H_n^{\pi}(Y,\bfE_0(p)) \to 
H_n^{\pi}(Y,\bfE_1(p))$
is bijective for $n < k$ and surjective for $n = k$.
\end{lemma}
\begin{proof}
This follows from a standard spectral sequence comparison argument applied to 
the spectral sequence appearing in~\cite[Theorem~4.7]{Davis-Lueck(1998)}
since $\bfu$ induces a map between the spectral sequences associated to  
$H_*^{\pi}(Y,\bfE_0(p))$
and $H_*^{\pi}(Y,\bfE_1(p))$. 
\end{proof}

%%%%%%%%%%%%%%%%%%%%%%%%%%%%%%%%%%%%%%%%%%%%%%%%%%%%%%%%%%%%%%%%%%%%%%%%%%%%%%%

\subsection{Fibered assembly maps and the Farrell-Jones 
Conjecture}\label{subsec:NWh_and_FJC}

The content of this subsection is a proof the following result:

\begin{theorem}\label{the:FJC_and_Tate_cohomology}
  Let $p \colon E \to B$ be a map of path connected spaces which is  $\pi_1$-surjective.
  Suppose that $B$ is an aspherical $CW$-complex with 
fundamental group a torsionfree group $\pi$. Assume that the $K$-theoretic Farrell-Jones 
Conjecture of Definition~\ref{def:K-theoretic_Farrell-Jones_Conjecture} holds for $\pi$ 
and that the cyclic subgroups of $\pi$ are orientable in the sense of Definition~\ref{def:orientable_cyclic_subgroups}.  

  Then, for $n\leq 1$, the map
  \begin{equation}\label{eq:Wh_theory_fibered_assembly_map}
    H_n^{\pi}(E\pi;\bfWh(p)) \to H_n^{\pi}(\pt;\bfWh(p)) = \pi_n(\bfWh(E))
  \end{equation}
  induced by the projection $\pr \colon E\pi \to \pt$ is split injective, its 
cokernel is
  $\NWh_n(p)$, and we get for the Tate cohomology
  \[
  \widehat{H}^i(\IZ/2;\NWh_n(p)) = 0
  \]
  for all $i \in \IZ$.
\end{theorem}

For its proof we will need the following 

\begin{lemma} \label{lem:asmb_for_A_pt}
If $u \colon Y \to Z$ is a $\pi$-map of $\pi$-$CW$-complexes which is a weak
homotopy
equivalence after forgetting the $\pi$-action, then the induced map
\[
H_n^{\pi}(u;\bfA^\%(p)) \colon H_n^{\pi}(Y;\bfA^\%(p)) \xrightarrow{\cong}
H_n^{\pi}(Z;\bfA^\%(p))
\]
is bijective for all $n \in \IZ$.
\end{lemma}
\begin{proof}
Define the contravariant functor
\[
O(Y) \colon \Or(\pi) \to \Spaces, \quad \pi/H \mapsto \map_{\pi}(\pi/H,Y).
\]
By unravelling the definitions, we see that $H_n^{\pi}(u;\bfA^\%(p))$ is
the homomorphism on $\pi_n$ coming from the map of spectra
\[
\bigl(O(Y) \times_{\Or(\pi)} \underline{X}\bigr)_+ \wedge \bfA(\pt)
\to
\bigl(O(Z) \times_{\Or(\pi)} \underline{X}\bigr)_+ \wedge \bfA(\pt).
\]
Hence it suffices to show that the map of spaces
\[
O(Y) \times_{\Or(\pi)} \underline{X} \to O(Z) \times_{\Or(\pi)} \underline{X}
\]
is a weak homotopy equivalence. If we put $\overline{X} := X(\pi/1)$,
then $\overline{X}$ carries a proper free right $\pi$-action and
$\underline{X}$ can be identified with the functor sending $\pi/H$ to the space 
$\overline{X} \times_{\pi} \pi/H$.
There is a homeomorphism
\[
\alpha(Y) \colon O(Y) \times_{\Or(\pi)} \underline{X} \xrightarrow{\cong} 
\overline{X} \times_{\pi} Y
\]
which is induced by the various maps 
\[
\map_{\pi}(\pi/H, Y) \times \bigl(\overline{X} \times_{\pi} \pi/H\bigr) \to 
\overline{X} \times_{\pi} Y,
\quad \bigl(\sigma, (\overline{x}, wH)\bigr) \mapsto \bigl(\overline{x},\sigma(w 
H)\bigr).
\]
The inverse of $\alpha$ sends $(\overline{x},y)$ to the element in $O(Y) 
\times_{\Or(\pi)} \underline{X}$
given by 
$(\sigma_y,\overline{x})$ in $\map_{\pi}(\pi, Y) \times_{\pi}  \overline{X} 
= \map_{\pi}(\pi/\{1\},Y) \times_{\pi} \underline{X}(\pi/\{1\})$  
for $\sigma_y \colon \pi \to Y,\; w \mapsto w \cdot y$.
Hence it suffices to show that 
\[
\id_{\overline{X}} \times_{\pi} u \colon \overline{X} \times_{\pi} Y \to 
\overline{X} \times_{\pi} Z
\]
is a weak homotopy equivalence. Since $\overline{X} \to X$ is a covering, 
we obtain a commutative diagram whose rows are fibrations.
\[
\xymatrix{
Y \ar[r] \ar[d]^{u}
&
\overline{X} \times_{\pi} Y  \ar[r] \ar[d]^{\id_{\overline{X}} \times_{\pi} u} 
&
X \ar[d]^{\id_X}
\\
Z \ar[r] 
&
\overline{X} \times_{\pi} Z \ar[r]  
&
X 
}
\]
Now the long homotopy sequence associated to a fibration and the Five-Lemma 
prove that
$\id_{\overline{X}} \times_{\pi} u$ is a weak homotopy equivalence.
This finishes the proof of Lemma~\ref{lem:asmb_for_A_pt}.
\end{proof}

Now we can give the

\begin{proof}[Proof of Theorem~\ref{the:FJC_and_Tate_cohomology}]
Notice that
we can use $\widetilde{B}$ as a model for $E\pi$ for $\pi = \pi_1(B)$, as $B$ is 
aspherical. 
So the cokernel of the map~\eqref{eq:Wh_theory_fibered_assembly_map} is 
$\NWh_n(p)$ by its very definition.

Let $\underline{\underline{E}}\pi$ be the classifying space for the family of
virtually cyclic subgroups. By~\cite[Lemma 11.3]{Lueck-Steimle(2016splitasmb)}, 
the
$K$-theoretic Farrell-Jones Conjecture implies that the map  induced by the 
projection
$\underline{\underline{E}}\pi \to \pt$
\[
H_n^{\pi}(\underline{\underline{E}}\pi;\bfK_{\IZ\Pi}(p)) \xrightarrow{\cong}
H_n^{\pi}(\pt;\bfK_{\IZ\Pi}(p))
\]
is an isomorphism for all $n \in \IZ$. (Here we use the $\pi_1$-surjectivity.)
As $\pi$ is assumed to be torsionfree, 
the $\pi$-space $E\pi$ is
also a classifying space $\underline{E}\pi$ for $\pi$-actions with finite
stabilizers. We conclude from~\cite[Theorem~0.1 and 
Theorem~0.2]{Lueck-Steimle(2016splitasmb)}
that the map induced by the projection $\pr \colon E\pi \to \pt$
\begin{equation}\label{eq:K_theory_fibered_assembly_map}
H_n^{\pi}(\pr;\bfK_{\IZ\Pi}(p))  \colon H_n^{\pi}(E\pi;\bfK_{\IZ\Pi}(p)) \to 
H_n^{\pi}(\pt;\bfK_{\IZ\Pi}(p))\end{equation}
is split injective (for all $n$) and that the Tate cohomology groups of its 
cokernel
(considered as a $\IZ/2$-module under the canonical involution) vanish.

There exists a linearization map 
\begin{eqnarray*}
& \bfL(Y) \colon \bfA(Y)  \to \bfK_{\IZ\Pi}(Y) &
\end{eqnarray*}
which is natural in $Y$ and always $2$-connected.
It induces a transformation of functors $\Or(\pi) \to \Spectra$ from
$\bfA(p)$ to $\bfK_{\IZ\Pi}(p)$ whose evaluation at any object $\pi/H$ is 
$2$-connected.
From Lemma~\ref{lem:weak_in_bfE} we obtain for every $n \le 1$ and every 
$\pi$-$CW$-complex an isomorphism, natural in $Y$,
\begin{eqnarray}
&
H_n^{\pi}(Y;\bfA(p)) \xrightarrow{\cong}  H_n^{\pi}(Y;\bfK_{\IZ\Pi}(p)).
&
\label{from_A-to_K_ZPi}
\end{eqnarray}
We conclude that the map 
\begin{equation}\label{eq:A_theory_fibered_assembly_map}
H_n^\pi(\pr, \bfA(p))\colon H_n^{\pi}(E\pi;\bfA(p)) \to H_n^{\pi}(\pt,\bfA(p))
\end{equation}
is split injective at least for $n\leq 1$ with the same cokernel as 
the map~\eqref{eq:K_theory_fibered_assembly_map}.

Next we apply Lemma~\ref{lem:asmb_for_A_pt} for $Y=E\pi$ and $Z=\pt$. The long 
exact
sequence obtained from the cofiber sequence~\eqref{cofiber_sequence_A_Wh} 
implies that the
natural transformation $\bfA\to \bfWh$ induces an isomorphism
\[
H_n^\pi(E\pi \to \pt; \bfA(p))\xrightarrow{\cong} H_n^\pi(E\pi \to \pt; 
\bfWh(p))
\]
on the relative homology groups. We  obtain a commutative diagram
with exact columns whose horizontal  arrows marked with $\cong$ are bijective
\[
\xymatrix{\vdots \ar[d]_{z_{n+1}} 
& 
\vdots \ar[d]^{z_{n+1}} 
\\
H_{n+1}^{\pi}(E\pi \to \pt; \bfA(p)) \ar[r]^{\cong} \ar[d]_{\delta_{n+1}}
&
H_{n+1}^{\pi}(E\pi \to \pt; \bfWh(p)) \ar[d]^{\delta_{n+1}}
\\
H_n^{\pi}(E\pi; \bfA(p)) \ar[r] \ar[d]_{H_n^{\pi}(\pr; \bfA(p))}
&
H_n^{\pi}(E\pi; \bfWh(p)) \ar[d]^{H_n^{\pi}(\pr; \bfWh(p))}
\\
H_n^{\pi}(\pt; \bfA(p)) \ar[r] \ar[d]_{z_n}
&
H_n^{\pi}(\pt; \bfWh(p)) \ar[d]^{z_n}
\\
H_n^{\pi}(E\pi \to \pt; \bfA(p)) \ar[r]^{\cong} \ar[d]_{\delta_n}
&
H_n^{\pi}(E\pi \to \pt; \bfWh(p)) \ar[d]^{\delta_n}
\\
\vdots & \vdots
}
\]
Since each of the maps $H_n^{\pi}(\pr; \bfA(p))$ is split injective for $n \le 
1$, one
easily checks that each of the maps $z_n \colon H_n^{\pi}(\pt; \bfA(p)) \to 
H_n^{\pi}(E\pi
\to \pt; \bfA(p))$ is split surjective for $n \le 1$.  This implies that each of 
the maps
$z_n \colon H_n^{\pi}(\pt; \bfWh(p)) \to H_n^{\pi}(E\pi \to \pt; \bfWh(p))$ is 
split
surjective for $n \le 1$. Finally we conclude for $n \le 1$ that the map
$H_n^{\pi}(\pr; \bfWh(p))$ is split injective and has the same cokernel as
$H_n^{\pi}(\pr; \bfA(p))$. Since we have already shown that 
$H_n^{\pi}(\pr; \bfA(p))$ has  the same cokernel as 
the map~\eqref{eq:K_theory_fibered_assembly_map},
Theorem~\ref{the:FJC_and_Tate_cohomology} follows.
\end{proof}

%%%%%%%%%%%%%%%%%%%%%%%%%%%%%%%%%%%%%%%%%%%%%%%%%%%%%%%%%%%%%%%%%%%%%%%%%%%%%%%

\subsection{Comparison to Quinn's assembly maps}

Quinn~\cite[8. Appendix]{Quinn(1982a)} defines for a stratified system of 
fibrations $p \colon X \to B$ 
over a simplicial complex $B$ and a homotopy invariant functor $\bfE
\colon \Spaces \to \Spectra$ an assembly map
\begin{eqnarray}
\asmb_n \colon \IH_n(B;\bfE(p)) \to \pi_n(\bfE(X)).
\label{Quinn_assembly}
\end{eqnarray}
We want to compare the three assembly 
maps~\eqref{ass_Davis-Lueck_Epi},~\eqref{ass_Davis-Lueck_widetildeB},
and~\eqref{Quinn_assembly}. Roughly speaking, the 
map~\eqref{ass_Davis-Lueck_Epi}
is best suited for calculations based on the Farrell-Jones 
Conjecture,~\eqref{Quinn_assembly}
is best suited for geometric applications, 
and~\eqref{ass_Davis-Lueck_widetildeB} interpolates 
between~\eqref{ass_Davis-Lueck_Epi} and~\eqref{Quinn_assembly}. 

\begin{lemma} \label{lem:identifying_assembly_maps}
Let $p \colon X \to B$ be a $\pi_1$-surjective
map of path connected spaces such 
that $B$ is a simplicial complex.
Assume that $p$ is a stratified system of fibrations.

\begin {enumerate}
\item \label{lem:identifying_assembly_maps:map}
There is a natural homomorphism
\[
\mu_n \colon \IH_n(B;\bfE(p)) \to H_n^{\pi}(\widetilde{B};\bfE(p))
\]
from the group $\IH_n(B;\bfE(p))$ appearing in~\eqref{Quinn_assembly} to the 
$\pi$-homology group
$H_n^{\pi}(\widetilde{B};\bfE(p))$ of~\eqref{H_npi(-,E(p))}  such that the 
following diagram commutes
\[
\xymatrix{\IH_n(B;\bfE(p)) \ar[r]^{\mu_n} \ar[d]_{\asmb_n} 
& 
H_n^{\pi}(\widetilde{B} ;\bfE(p)) \ar[d]^{H_n^{\pi}(\pr;\bfE(p))}
\\
\pi_n(\bfE(X)) & 
H_n^{\pi}(\pt;\bfE(p)) \ar[l]^{\nu_n}_{\cong}
}
\]
where $\nu_n$ is the canonical isomorphism from~\eqref{H_pi_at_pi/H} for the 
homogeneous
space $\pi/\pi$;

\item \label{lem:identifying_assembly_maps:iso} Suppose that $p$ is a (Hurewicz) fibration 
and $B$ is aspherical.

Then the natural homomorphism
\[
\mu_n \colon \IH_n(B;\bfE(p)) \to H_n^{\pi}(\widetilde{B};\bfE(p))
\]
is bijective for $n \in \IZ$;

\item \label{lem:identifying_assembly_maps:same_image}
Suppose that $B$ is aspherical. Then the assembly maps
\[
\asmb_n \colon \IH_n(B;\bfE(p)) \to \pi_n(\bfE(X))
\]
of~\eqref{Quinn_assembly} and the assembly maps 
of~\eqref{ass_Davis-Lueck_widetildeB}
\[
H_n^{\pi}(\pr;\bfE(p)) \colon H_n^{\pi}(\widetilde{B} ;\bfE(p)) 
= H_n^{\pi}(E\pi;\bfE(p)) \to H_n^{\pi}(\pt;\bfE(p)) 
\]
have the same image, if we identify the targets by the isomorphism $\nu$ 
from~\eqref{H_pi_at_pi/H} for the homogeneous space $\pi/\pi$;

\item \label{lem:identifying_assembly_maps:passage_from_widetildeB_to_Epi}
Let $f \colon \widetilde{B} \to E\pi$ be the classifying $\pi$-map. Suppose that 
$\pi_1(X)$ satisfies
the $K$-theoretic FJC, see 
Definition~\ref{def:K-theoretic_Farrell-Jones_Conjecture}. Take $\bfE = \bfWh$,
defined in Subsection~\ref{subsec:fibered_assembly_maps}.

Then the left vertical arrow in the following commutative diagram is bijective 
for $n \le 0$ and surjective
for $n = 1$ 
\[
\xymatrix@!C= 15em{
H_n^{\pi}(\widetilde{B};\bfWh(p)) \ar[r]^-{H_n^{\pi}(\pr;\bfWh(p))} 
\ar[d]_{H_n^{\pi}(f;\Wh(p))}
&
H_n^{\pi}(\pt;\bfWh(p)) = \Wh_n(\pi_1(X)) \ar[d]^{\id}
\\
H_n^{\pi}(E\pi;\bfWh(p)) \ar[r]_-{H_1^{\pi}(\pr;\bfWh(p))} 
&
H_n^{\pi}(\pt;\bfWh(p)) = \Wh_n(\pi_1(X)) 
}
\]
In particular the cokernel of the upper horizontal arrow agrees with the 
cokernel of the lower
horizontal arrow for $n \le 1$. 
\end{enumerate}
\end{lemma}
\begin{proof}~\ref{lem:identifying_assembly_maps:map}
   Quinn defines a spectrum $\IE(p)$ by $\bigvee_{\sigma} \sigma_+ \wedge
  \bfE(p^{-1}(\sigma))/\sim$ where $\sigma$ runs through the simplices of $B$, 
the pointed
  space $\sigma_+$ is obtained from $\sigma$ by adding a disjoint base point, 
and the
  equivalence relation $\sim$ identifies for an inclusion $\sigma \subset \tau$ 
the
  spectrum $\sigma_+ \wedge \bfE(p^{-1}(\sigma))$ with its image under the map 
  $\sigma_+ \wedge \bfE(p^{-1}(\sigma)) \to \tau_+ \wedge \bfE(p^{-1}(\tau))$ 
coming from the
  inclusion $\sigma \subseteq \tau$.

Consider any simplex $\sigma$ of $B$. Let $\widetilde{\sigma}$ be any lift to 
$\widetilde{B}$
which will be equipped with the obvious simplicial structure coming from $B$. 
Then we obtain a map
of spectra 
\[
\widetilde{\bfa}(\widetilde{\sigma}) \colon 
\widetilde{\sigma}_+ \wedge 
\bfE\bigl(\overline{p}(\pi/1)^{-1}(\widetilde{\sigma})\bigr)
\to \widetilde{B}_+ \wedge \bfE\bigl(X(\pi/1)\bigr)
\]
by the smash product of  the inclusion $\widetilde{\sigma}_+ \to 
\widetilde{B}_+$
and the map $\bfE\bigl(\overline{p}(\pi/1)^{-1}(\widetilde{\sigma})\bigr) \to 
\bfE\bigl(X(\pi/1)\bigr)$
induced by the inclusion $\overline{p}(\pi/1)^{-1}(\widetilde{\sigma}) \to 
X(\pi/1)$.
The maps $q \colon \widetilde{B} \to B$ and
$\overline{p}(\pi/1) \colon X(\pi/1) \to X$ induce homeomorphisms 
$\widetilde{\sigma} \xrightarrow{\cong} \sigma$ and 
$\overline{p}(\pi/1)^{-1}(\widetilde{\sigma}) \xrightarrow{\cong} 
p^{-1}(\sigma)$ 
and thus an isomorphism of spectra
\[
\widetilde{\bfb}(\widetilde{\sigma}) \colon 
\widetilde{\sigma}_+\wedge 
\bfE\bigl(\overline{p}(\pi/1)^{-1}(\widetilde{\sigma})\bigr)
\xrightarrow{\cong} \sigma_+ \wedge \bfE(p^{-1}(\sigma)).
\]
For any $w \in \pi$ we obtain a commutative diagram, where the vertical maps are 
all
induced by multiplication with $w$
\[
\xymatrix@!C=12em{
 & \widetilde{\sigma}_+ \wedge 
\bfE\bigl(\overline{p}(\pi/1)^{-1}(\widetilde{\sigma})\bigr) 
\ar[ld]^{\widetilde{\bfb}(\widetilde{\sigma})}_{\cong} 
\ar[r]^-{\widetilde{\bfa}(\widetilde{\sigma})} \ar[dd]
& \widetilde{B}_+ \wedge \bfE\bigl(X(\pi/1)\bigr) \ar[dd]
\\
\sigma_+ \wedge \bfE(p^{-1}(\sigma)) 
& 
&
\\
& \bigl(\widetilde{\sigma} \cdot w\bigr)_+ 
\wedge \bfE\bigl(\overline{p}(\pi/1)^{-1}(\widetilde{\sigma} \cdot w)\bigr) 
\ar[lu]^{\widetilde{\bfb}(\widetilde{\sigma} \cdot w)}_{\cong}
\ar[r]_-{\widetilde{\bfa}(\widetilde{\sigma} \cdot w)} 
& \widetilde{B}_+ \wedge \bfE\bigl(X(\pi/1)\bigr) 
}
\]
Thus we obtain  a map of spectra
\[
\bfa(\sigma) \colon \sigma_+ \wedge \bfE\bigl(p^{-1}(\sigma)\bigr) 
\to \widetilde{B}_+ \wedge_{\pi}  \bfE\bigl(X(\pi/1)\bigr).
\]
They fit together to a map of spectra
\begin{eqnarray*}
\bfa \colon  \IE(p) := \bigvee_{\sigma} \sigma_+ \wedge 
\bfE(p^{-1}(\sigma))/\sim 
&  \to &
\widetilde{B}_+ \wedge_{\pi}  \bfE\bigl(X(\pi/1)\bigr).
\end{eqnarray*}
After taking homotopy groups it induces the desired map
\[
\mu_n \colon \IH_n(B;\bfE(p)) \xrightarrow{\cong} 
H_n^{\pi}(\widetilde{B};\bfE(p)).
\]
Notice that Quinn~\cite[8. Appendix]{Quinn(1982a)} replaces $\cals(B)$ by the 
associated
$\Omega$-spectrum but this does not matter since it does not change the homotopy
groups. In the setting of~\cite{Davis-Lueck(1998)} one does not need 
$\Omega$-spectra as
long as one is only interested in homology, 
see~\cite[Lemma~4.4]{Davis-Lueck(1998)}.  The
commutativity of the diagram appearing in
Lemma~\ref{lem:identifying_assembly_maps}~\ref{lem:identifying_assembly_maps:map} follows
directly from the definitions.
\\[2mm]~\ref{lem:identifying_assembly_maps:iso} 
If $Z$ is a space over $B$, i.e., a space with a reference map $f\colon Z\to B$, 
then the construction above 
extends and yields a map
\[
\mu_n(Z) \colon \IH_n(Z;\bfE(p))\to H_n^\pi(\overline{Z};\bfE(p))
\]
where $\overline{Z}$ is the $\pi$-covering given by the pullback
\[
\xymatrix{\overline{Z} \ar[d]
\ar[r]^{\overline{f}}
& \widetilde{B} \ar[d]
\\
Z \ar[r]^f &
B
}
\]
The map $\mu_n(Z)$ is natural in the space $Z$ over $B$. Next we show for any 
space $Z$ over $B$
that $\mu_n(Z)$ is an isomorphism for all $n\in \IZ$. Then the claim follows by 
applying it to
the case $Z = B$ with reference map $\id_B \colon B \to B$.

In fact both functors are homology theories on the category of spaces over $B$: 
For the
first one, see~\cite[Proposition~8.4]{Quinn(1982a)}, for the second one this is 
true by
construction. Hence by the standard Mayer-Vietoris and colimit argument
it is enough to check the case where $f\colon Z\to B$ is the inclusion
of a point $b\in B$. In this case $\overline{Z}$ is isomorphic to $\pi/1$ so
$H_n^\pi(\overline{Z};\bfE(p))\cong \bfE(X(\pi/1))$, and the map $\mu_n$ is the 
map 
$\pi_n \bfE(p\inv(b))\to \pi_n (\bfE(X(\pi/1)))$ induced by the inclusion 
from the point-preimage
into the homotopy fiber. But this map is a homotopy equivalence since $p$ is a
fibration and $\widetilde{B}$ is contractible. The claim follows as $\bfE$ is 
homotopy invariant.
\\[2mm]~\ref{lem:identifying_assembly_maps:same_image}
We can turn $p$ into a fibration $\widehat{p}$, i.e., there exists a commutative 
diagram
\[
\xymatrix{X \ar[r]^{u} \ar[d]^{p} 
& \widehat{X} \ar[d]^{\widehat{p}} \ar[r]^-{i_0}
& \widehat{X} \times [0,1] \ar[d]^{h}
& \widehat{X} \ar[l]_-{i_1} \ar[d]^{h_1}
\ar[r]^v
& X \ar[d]^{p} 
\\ 
B \ar[r]^{\id_B} 
&  B \ar[r]^-{i_0}
& B \times [0,1] 
& B \ar[l]_-{i_1} \ar[r]^{\id_B}
& B
}
\]
such that $\widehat{p} \colon \widehat{X} \to B$ is a fibration, $u$ and $v$ 
are homotopy inverse homotopy equivalences
and the maps $i_k$ denote the obvious inclusions coming from the inclusions 
$\{k\} \to [0,1]$, and $h$ is a homotopy between $p\circ v$ and $\widehat p$.
It induces a commutative diagram of abelian groups whose arrows marked with 
$\cong$
are bijections because of the spectral sequence due to Quinn~\cite[8. 
Appendix]{Quinn(1982a)}
and the assumption that $\bfE$ sends weak homotopy equivalences to weak homotopy 
equivalences. 
\[
\xymatrix@!C= 10em{
\IH_n(B;\bfE(p)) \ar[r]^{\asmb_n} \ar[d]_{u_*}
& \pi_n(\bfE(X)) \ar[d]_{\pi_n(\bfE(u))}^{\cong} \ar@/^{25mm}/[dddd]^{\id}
\\
\IH_n(B;\bfE(\widehat{p})) \ar[r]^{\asmb_n} \ar[d]_{(i_0)_*}^{\cong}
& \pi_n(\bfE(\widehat{X})) \ar[d]_{\pi_n(\bfE(i_0))}^{\cong}
\\
\IH_n(B\times [0,1] ;\bfE(h)) \ar[r]^{\asmb_n} 
& \pi_n(\bfE(\widehat{X} \times [0,1]))
\\
\IH_n(B;\bfE(\widehat{p})) \ar[r]^{\asmb_n} \ar[u]^{(i_1)_*}_{\cong} 
\ar[d]_{v_*}
& \pi_n(\bfE(\widehat{X})) \ar[u]^{\pi_n(\bfE(i_1))}_{\cong} 
\ar[d]_{\pi_n(\bfE(v))}^{\cong}
\\
\IH_n(B;\bfE(p)) \ar[r]^{\asmb_n} 
& \pi_n(\bfE(X)) 
}
\]
Notice that we are not claiming that the vertical arrows $u_*$ and $v_*$ are 
isomorphisms.
Nevertheless, an easy diagram chase shows that the images of 
\[
\asmb_n \colon \IH_n(B;\bfE(p)) \to  \pi_n(\bfE(X)) 
\]
and
\[
\asmb_n \colon \IH_n(B;\bfE(\widehat{p})) \to  \pi_n(\bfE(\widehat{X})) 
\]
agree if we identify their targets with the isomorphism $\pi_n(\bfE(u))$
whose inverse is  $\pi_n(\bfE(v))$.
The following diagram commutes 
\[
\xymatrix@!C= 10em{
H_n^{\pi}(\widetilde{B};\bfE(p)) \ar[r]^{\asmb_n} 
\ar[d]_{H_n^{\pi}(\widetilde{B}; \bfE(u))} 
& H_n^{\pi}(\pt; \bfE(p)) \ar[d]^{H_n^{\pi}(\pt; \bfE(u))} 
\\
H_n^{\pi}(\widetilde{B};\bfE(\widehat{p}) \ar[r]_{\asmb_n} 
& H_n^{\pi}(\pt; \bfE(\widehat{p})) }
\]
The vertical arrows are bijections because of Lemma~\ref{lem:weak_equi_f}.
Hence the images of the upper and the lower horizontal arrows agree if
we identify their targets with the right vertical isomorphisms.
We conclude that it suffices to prove 
assertion~\ref{lem:identifying_assembly_maps:same_image}
for $\widehat{p}$, or in other words, we can assume without loss of generality
that $p$ is a fibration.
But then the claim follows directly from 
assertion~\ref{lem:identifying_assembly_maps:iso}.
\\[2mm]~\ref{lem:identifying_assembly_maps:passage_from_widetildeB_to_Epi}
By assumption $\pi_1(X)$ satisfies the $K$-theoretic 
Farrell-Jones Conjecture; Theorem~\ref{the:status_of_FJC} 
implies that each of its subgroups $\pi_1(X(\pi/H))$ for 
$H \subseteq \pi$ does. Hence $\pi_q(\bfWh(X(\pi/H))$ vanishes 
for all $q \le -2$ and $H \subseteq \pi$, see for 
instance~\cite[Subsection~3.1.1]{Lueck-Reich(2005)}. The map $f \colon 
\widetilde{B} \to E\pi$ 
is a  $2$-connected map of free  $\pi$-$CW$-complexes. Now apply the spectral 
sequence 
of~\cite[Theorem~4.7]{Davis-Lueck(1998)}.
This finishes the proof of Lemma~\ref{lem:identifying_assembly_maps}.
\end{proof}

%%%%%%%%%%%%%%%%%%%%%%%%%%%%%%%%%%%%%%%%%%%%%%%%%%%%%%%%%%%%%%%%%%%%%%%%%%%%%%%
%%%%%%%%%%%%%%%%%%%%%%%%%%%%%%%%%%%%%%%%%%%%%%%%%%%%%%%%%%%%%%%%%%%%%%%%%%%%%%%
%%%%%%%%%%%%%%%%%%%%%%%%%%%%%%%%%%%%%%%%%%%%%%%%%%%%%%%%%%%%%%%%%%%%%%%%%%%%%%%

\section{The tight torsion and proof of 
Theorem~\ref{the:interpretation_in_terms_of_Q_manifolds}}
\label{sec:tight_torsion}

%%%%%%%%%%%%%%%%%%%%%%%%%%%%%%%%%%%%%%%%%%%%%%%%%%%%%%%%%%%%%%%%%%%%%%%%%%%%%%%

\subsection{Factorization to an approximate fibration}
\label{subsec:Factorization_to_an_approximate_fibration}

This subsection is devoted to the following result which will be one of the key 
ingredients in the definition of tight torsion.

\begin{theorem}[Factorization to a approximate fibration]
\label{thm:factorization_to_an_approximate_fibration}
  Let $p\colon M\to B$ be a map between topological spaces, such that $B$ is a 
finite
  simplicial complex and for each $b\in B$, the homotopy fiber of $p$ over $B$ is homotopy finite. Then 
there exists
  a homotopy commutative diagram
  \[
   \xymatrix{
    M \ar[rr]^f \ar[rd]_p && E \ar[ld]^q\\
    & B }
   \]
  where $f$ is a homotopy equivalence, $q$ is an approximate fibration, and $E$ 
is a
  compact ENR.
\end{theorem}

We are grateful to Steve Ferry for suggesting the following argument to us, 
which replaces
a more complicated proof in an earlier version.

\begin{proof}
  We may assume that $p$ is a fibration. We will construct a homotopy 
equivalence 
  $g\colon  E\to M$ such that $p\circ g=q$ strictly, and we will moreover make 
sure that for any
  subcomplex $A$ of $B$, the restriction of $q$ over $A$ is an approximate 
fibration whose
  total space is a compact ENR and that the restriction of $f$ over $A$ is a 
homotopy
  equivalence.

  The proof is by induction on the number of simplices of $B$. The case where 
$B$ consists
  of a single 0-simplex is trivial. For the inductive step, assume that $B$ is 
obtained
  from $B'$ by attaching a single $n$-simplex $\sigma$ along its boundary, and 
denote by
  $g'\colon E'\to M':=M\vert_{B'}$ and $q'\colon E'\to B'$ the maps obtained 
from the
  inductive assumption.

Let $b\in B$ be the barycenter of $\sigma$ and choose a homotopy equivalence 
\[
h\colon M_b:=p\inv(b)\to F
\]
where $F$ is a finite $CW$-complex. As $p$ is a fibration, there is a homotopy 
equivalence
\[
M\vert_\sigma\xrightarrow{\simeq} \cyl(M\vert_{\partial\sigma} \xrightarrow{t} 
M_b)
\]
relative to $M\vert_{\partial\sigma}$, where $t$ is given by fiber transport.

Let $E:=E'\cup_{E'\vert_{\partial\sigma}} 
\cyl(E'\vert_{\partial\sigma}\xrightarrow{h\circ
  t\circ g'} F)$. Then $E$ is a compact ENR 
by~\cite[Corollary~E.7]{Bredon(1997a)}
and~\cite[chapter VI, \S1]{Hu(1965retracts)}. The commutative diagram
\[
\xymatrix{E'\vert_{\partial\sigma} \ar[d]^{g'}_\simeq \ar[rr]^{h\circ t\circ g'} 
&& F\\
 M\vert_{\partial\sigma} \ar[rr]^t && M_b \ar[u]_h^\simeq
}
\]
induces a homotopy equivalence $\cyl(h\circ t\circ g')\to \cyl(t)$ which 
restricts to $g'$
on $E'\vert_{\partial\sigma}$, so we obtain a homotopy equivalence $g''\colon 
E\to M$
which extends $g'$.

We let $q\colon E\to B$ extend $q'$ by sending the mapping cylinder canonically 
to
$\sigma=\cyl(\partial\sigma\to \ast)$. It follows from the inductive assumption 
by
application of the criterion~\cite[Theorem 12.15]{Hughes-Taylor-Williams(1990)} 
that $q$
is in fact an approximate fibration (see~\cite{Steimle(2011)} for more details 
of this
argument). Moreover $p\circ g''$ and $q$ agree over $B'$ so these two maps are 
homotopic
by the straight-line homotopy in the interior of $\sigma$. As $p$ is a 
fibration, the map
$g''$ is homotopic to a map $g$ such that $p\circ g=q$, via a homotopy which is 
stationary
over $B'$.

If $A\subset B$ is a subcomplex, then either it does not contain $\sigma$, in 
which case
the claim is contained in the inductive assumption. Otherwise $A$ is obtained 
from some
$A'\subset B'$ by attaching the single simplex $\sigma$, and the same argument 
as above
applies to show that the restriction of $q$ over $A'$ is an approximate 
fibration whose
total space is a compact ENR, and that the restriction of $g$ over $A'$ is a 
homotopy
equivalence.
\end{proof}

%%%%%%%%%%%%%%%%%%%%%%%%%%%%%%%%%%%%%%%%%%%%%%%%%%%%%%%%%%%%%%%%%%%%%%%%%%%%%%%

\subsection{Definition of tight torsion}
\label{subsec:Definition_of_tight_torsion}

Recall that, given a homotopy equivalence $f\colon X\to Y$ between
compact ENRs, it is possible to define the Whitehead torsion
$\tau(f)\in\Wh(\Pi Y)$ which has the usual properties. It can be
calculated from the classical Whitehead torsion of a homotopy
equivalence between compact CW-complexes by the composition rule using the 
following
facts:
\begin{enumerate}
\item each compact ENRs receives a cell-like map $A\to X$ from a
  finite $CW$-complex $A$;

\item the Whitehead torsion of a cell-like map between compact ENRs
  (which is always a homotopy equivalence) is zero.
\end{enumerate}
See~\cite{Lacher(1977)} for a survey on cell-like maps.

In the sequel $p \colon M \to B$ is a $\pi_1$-surjective
map between closed topological
manifolds such that $B$ is triangulable.  In particular $M$ is  a compact ENR.

In order to define tight torsion below, we will need the following.

\begin{lemma}\label{lem:well_definition_of_tight_torsion}
  Given a metric on $B$, there exists an $\eps>0$, such that the following 
holds: For $i=1,2$, let
  $p\simeq q_i\circ f_i$ be two factorizations up to homotopy into a homotopy 
equivalence, followed by an 
  $\eps$-fibration whose total spaces are compact {ENR}'s. 
  Then the element
  \[
   (f_1)_*\inv\tau(f_1)-(f_2)_*\inv\tau(f_2) \in \Wh(\Pi M)
  \] 
  is in the image of the assembly map $H_1^\pi(\pr;\bfWh(p)) \colon 
H_1^{\pi}(\widetilde{B};\bfWh(p)) \to \Wh(\Pi M)$.
 \end{lemma}

 Recall that we have introduced $\NWh(p)$ in Definition~\ref{def:NWH(p)} as the 
cokernel
 of this assembly map. Hence the following definition is meaningful because of 
 Lemma~\ref{lem:well_definition_of_tight_torsion}.

\begin{definition}[Tight torsion] \label{def:tide_torsion}
  Let $p \colon M \to B$ be a $\pi_1$-surjective    map between closed topological
  manifolds such that $B$ is triangulable and such that the homotopy fiber of $p$ is homotopy finite.  The \emph{tight torsion}
  \[
   N\tau(p)\in \NWh(p)
   \] 
   is defined whenever there exists a factorization $p\simeq q\circ f$ 
  up to homotopy into a homotopy equivalence followed by an
  approximate fibration whose total space is a compact ENR; in this case it is defined to be the image of the Whitehead torsion
  $f_*\inv\tau(f)$ of $f$ in the quotient group $\NWh(p)$. 
\end{definition}

Note that $ \NWh(p)$ and the tight torsion $N\tau(p)$ depend only on the 
homotopy class of $p$. 
Our terminology comes from the case where $B$ is the circle $S^1$. Then
$N\tau(p)$ is the component of the primary torsion obstruction defined in
Farrell~\cite{Farrell(1967),Farrell(1971)} which lies in the summand
\[
\widetilde{C}(\IZ \Gamma, \alpha)\oplus \widetilde{C}(\IZ \Gamma, \alpha^{-1})
\]
of $\Wh(\pi_1(M))$ where $\Gamma=\ker \pi_1(p)$ and $\pi_1(M) = \Gamma 
\rtimes_\alpha \pi_1(S^1)$. 
In this case $N\tau(p)$ vanishes under transfer to
the finite sheeted covers $M_n\xrightarrow{p_n} S^1$ 
for all sufficiently large integers $n$ where these
covers are defined by the following Cartesian diagram:
\[
\xymatrix{
  M_n \ar[rr]^{p_n} \ar[d] && S^1 \ar[d]^{z\mapsto z^n}\\
  M \ar[rr]^p && S^1 }
\]
Our terminology comes by thinking that these covers relax the torsion more and
more, as $n\to\infty$, until it becomes zero; 
cf.~\cite{Farrell(1971b),Siebenmann(1970b)}.

%%%%%%%%%%%%%%%%%%%%%%%%%%%%%%%%%%%%%%%%%%%%%%%%%%%%%%%%%%%%%%%%%%%%%%%%%%%%%%%

\subsection{Proof of Lemma~\ref{lem:well_definition_of_tight_torsion}}
\label{subsec:Proof_of_Lemma_ref(lem:well_definition_of_tight_torsion)}

In order to guarantee that Definition~\ref{def:tide_torsion} makes sense,
we still have to prove Lemma~\ref{lem:well_definition_of_tight_torsion}.
This needs some preparations.

\begin{definition}[Control] \label{def:controlled_maps}
  Let
  \begin{equation}\label{diag:epsilon_homotopy_equivalence}
  \xymatrix{X \ar[rr]^f \ar[rd]_p && Y \ar[ld]^q\\
      & B  
    }\end{equation}
  be a (not necessarily commutative) diagram of spaces, with $B$ a metric space, 
 and let $\eps>0$. 
  \begin{enumerate}
  
\item $f$ is called \emph{$\eps$-controlled} if $d(q\circ  f(x), p(x))<\eps$
holds for all $x\in X$;
  
\item A homotopy $H\colon X\times I\to Y$ is called an $\eps$-homotopy if for
    all $x\in X$ the path $q\circ H(x,-)$ in $B$ has diameter less than $\eps$;
  
\item $f$ is an \emph{$\eps$-domination} if it is $\eps$-controlled and  there
    exists an $\eps$-controlled map $g\colon Y\to X$ and an $\eps$-homotopy
    $f\circ g\simeq \id_Y$;
 
 \item An $\eps$-domination $f$ is an \emph{$\eps$-homotopy equivalence} if, in
    addition, there exists an $\eps$-homotopy $g\circ f\simeq \id_X$.
  \end{enumerate}
\end{definition}

\begin{remark}
  It is easy to see that if $f\colon X\to Y$ and $g\colon Y\to Z$ are 
$\eps$-homotopy
  equivalences, then $g\circ f$ is an $4\eps$-homotopy equivalence. This is the 
reason to
  require that not just the homotopies, but also the maps themselves have
  $\eps$-control. Thus our convention differs from other definitions (for 
instance,
  in~\cite{Chapman(1980)}) that just require the homotopies to be controlled. 
Note that if
  a map $f$ is $\eps$-controlled and an $\eps$-homotopy equivalence in this 
weaker sense,
  then the triangle inequality implies that it is automatically a 
$2\eps$-homotopy
  equivalence in our sense.
\end{remark}

\begin{lemma}\label{lem:condition_for_epsilon_h_eq}
  Suppose that $X$, $Y$ are separable metric spaces, that $B$ is a compact ANR 
with a metric and let $\eps>0$. There is a $\delta>0$ such that the following 
holds: If 
  in~\eqref{diag:epsilon_homotopy_equivalence}, $p$ and $q$ are
  $\delta$-fibrations, and $f$ is a homotopy equivalence such that the diagram 
commutes up to homotopy, 
then $f$ is homotopic
  to an $\eps$-homotopy equivalence.
\end{lemma}

\begin{proof}
  Factor $q$ into a homotopy equivalence $\lambda\colon Y\to\cale$ followed by a 
fibration
  $r\colon \cale\to B$. By the fibration property of $r$, the composite 
$\lambda\circ f$
  is homotopic to a homotopy equivalence $f'$ such that $r\circ f'=p$, i.e., we 
get a
  commutative diagram
\[
\xymatrix{
X \ar[rr]^{f'}_\simeq \ar[rrd]_p && \cale \ar[d]^r && Y \ar[lld]^q 
\ar[ll]_\lambda^\simeq\\
&& B
}
\]
By~\cite[Proposition 2.3]{Chapman(1980)}, both $f'$ and $\lambda$ will be 
$\eps/4$-homotopy equivalences, provided
$p$ and $q$ are $\delta$-fibrations for a suitably chosen $\delta>0$. Let
$\lambda\inv$ denote a homotopy inverse which is also an $\eps/4$-homotopy 
equivalence. Then
$f$ is homotopic to the $\eps$-homotopy equivalence $\lambda\inv\circ f'$.
\end{proof}

For later use we also record the following opposite result:

\begin{lemma}\label{lem:condition_for_epsilon_fibration}
  Let $Y$ be an ANR, let $B$ be a compact  ANR with a 
metric and let $\eps>0$. There is a $\delta>0$ such that the following holds: 
If 
  in~\eqref{diag:epsilon_homotopy_equivalence}, $p$ is a $\delta$-fibration and 
$f$ is
  a $\delta$-domination, then $q$ is an
  $\eps$-fibration.
\end{lemma}

\begin{proof}
    Let $\delta>0$ and assume that there exists a $\delta$-section 
  $g\colon Y\to X$ of $f$, i.e., $g$ is $\delta$-controlled and that
  $f\circ g$ is $\delta$-homotopic to the identity map on $Y$. We are first
  going to show that $p\circ g$ is a $5\delta$-fibration. So suppose we are given a
  homotopy lifting problem
  \[
   \xymatrix{Z\times 0 \ar[rr]^{H_0} \ar@{^{(}->}[d] && Y \ar[d]^{p\circ g}\\
    Z\times I \ar[rr]^h && B }
   \]
  for $p\circ g$. By~\cite[Theorem~12.13]{Hughes-Taylor-Williams(1990)}, we can assume that $Z$ is a cell. Postcomposing with the map $g$ yields a homotopy lifting problem for $p$
  which can be solved up to $\delta$ by a map $L\colon Z\times I\to X$. Then
  $f\circ L\vert_{Z\times 0}$ is $\delta$-homotopic to $H_0$, so using the
  estimated homotopy extension property~\cite[ Proposition~2.1]{Chapman(1979)}, 
  we can replace
  $f\circ L$ by a $\delta$-homotopic
  map $H$ such that $H\vert_{Z\times 0}=H_0$. By the triangle inequality, $H$ is then a $3\delta$-lift of $h$.

  Now, since $p\circ g$ is $\delta$-close to $q$ and $p\circ g$ is a
  $3\delta$-fibration, it follows that $q$ is an $\eps$-fibration if $\delta$
  was chosen small enough. This can be seen as follows:

  Choose $\delta>0$ small enough so any two $\delta$-close maps to $B$ are
  $\eps/4$-homotopic and $3\delta<\eps/4$ holds, 
  see~\cite[Theorem~1.1 in Chapter~IV on page~111]{Hu(1965retracts)}. In 
particular $p\circ g$ and $q$ are
  $\eps/4$-homotopic. This homotopy may be used to obtain from the homotopy 
lifting
  problem displayed on the left a homotopy lifting problem as displayed on the 
right hand
  side:
\[
\xymatrix{Z\times \{0\}\ar[d] \ar[r]^{H_0} & Y \ar[d]^q && Z\times \{-1\} \ar[d] 
\ar[r]^{H_0} & Y \ar[d]^{p\circ g}\\
 Z\times [0,1] \ar[r]^h & B   && Z\times [-1,1] \ar[r]^h \ar@{.>}[ru]^L & B  
}
\]
As $p\circ g$ is a $3\delta$-fibration, we may solve the problem on the right 
hand side up
to $3\delta<\eps/4$ by a map $L$. Now by the triangle inequality the restriction 
of $L$ to $Z\times [0,1]$ is a $3\eps/4$-homotopy from the restriction of $L$ to 
$Z\times \{0\}$ to $H_0$.  The
estimated homotopy extension property (control with respect to the map $p\circ 
g$) yields a map
\[
\widehat{H}\colon Z\times [-1,1]\to Y
\]
which is $3\eps/3$-homotopic (control with respect to the map $p\circ g$) to $L$ 
and which is $H_0$
when restricted to $Z\times \{0\}$. It is easily seen that $H:=\widehat 
H\vert_{Z\times
  [0,1]}$ is an $\eps$-lift of $h$.
\end{proof}

\begin{proof}[Proof of Lemma~\ref{lem:well_definition_of_tight_torsion}]
  Let us first consider the special case where $E_1$ and $E_2$ happen to be 
compact CW
  complexes. Let $g:= f_1\circ f_2\inv$.  We have a homotopy commutative 
triangle
  \[
  \xymatrix{E_2 \ar[rr]^g_\simeq \ar[rd]_{q_2} && E_1 \ar[ld]^{q_1}\\
    & B }
   \]
  By Lemma~\ref{lem:condition_for_epsilon_h_eq}, choosing $\eps$ small enough, 
we may
  assume that $g$ is a $\delta$-homotopy equivalence, for some given 
$\delta>0$. 
  
  Factor $q_1=p'\circ \lambda$ into a homotopy equivalence $\lambda\colon E_1\to 
E'$ followed by a fibration $p'\colon E'\to B$. As $q_1$ is an approximate 
fibration $\lambda$ is a controlled homotopy equivalence, that is, an $\eps$-homotopy equivalence for every $\eps>0$.
  
  In this situation, by~\cite[1.4]{Quinn(1982a)} there is a controlled Whitehead torsion  in
  $\IH_1(\widetilde B;\bfWh(p'))$; it is mapped under the assembly map of Quinn
  \begin{eqnarray}
  \IH_1(B;\bfWh(p')) & \to & \Wh(\pi_1(E'))
  \label{Quinn_assembly_map_in_degree_1}
   \end{eqnarray}
  to the image of the Whitehead torsion of $g$ under the map induced by 
$\lambda$ (see~\cite[paragraph before~1.7]{Quinn(1982a)}). We get from
  the composition rule the following equation in $\Wh(\pi_1(E_1))$:
  \begin{eqnarray*}
  \tau(g) & = &  \tau(f_1) - (f_1)_*(f_2)_*\inv\tau(f_2).
   \end{eqnarray*}
  This implies that $\tau(f_1) - (f_1)_*(f_2)_*\inv\tau(f_2)$ lies in the image 
of the assembly 
   map~\eqref{Quinn_assembly_map_in_degree_1}. We obtain from
   Lemma~\ref{lem:weak_equi_f}\ref{lem:identifying_assembly_maps:map} and 
   Lemma~\ref{lem:identifying_assembly_maps}~\ref{lem:identifying_assembly_maps:iso}
    the following commutative diagram
   \[
    \xymatrix{\IH_1(B;\bfWh(p')) \ar[r] \ar[d]_{\mu_1}^\cong  & \Wh(\pi_1(E'))
    \\
    H_1^{\pi}(\widetilde{B};\bfWh(p')) \ar[r]  & H_1^{\pi}(\pt;\Wh(p'))  
\ar[u]_{\nu_1}^{\cong}
    \\
    H_1^{\pi}(\widetilde{B};\bfWh(p)) \ar[r] \ar[u]^{(\lambda\circ 
f_1)_*}_{\cong}]  & H_1^{\pi}(\pt;\Wh(p)) = \Wh(\pi_1(M)) \ar[u]_{(\lambda\circ 
f_1)_*}^{\cong}
    }
    \]
    where the horizontal maps are assembly maps and the vertical maps are 
bijective.
    This implies that the image of the assembly 
    map~\eqref{Quinn_assembly_map_in_degree_1} agrees with the image of the 
assembly map
   $H_1^{\pi}(\widetilde{B};\bfWh(p))  \to H_1^{\pi}(\pt;\Wh(p)) = 
\Wh(\pi_1(M))$.  
   This finishes the proof of Lemma~\ref{lem:well_definition_of_tight_torsion}  
in the special case 
   that $E_1$ and $E_2$ are compact CW-complexes.

  In the general case, choose cell-like maps $h_i\colon E'_i\to E_i$ for
  $i=1,2$, such that $E'_i$ is a compact CW-complex. As any cell-like map
  between compact ENRs is an approximate fibration, the composition $q_i\circ
  h_i$ is a, say, $2\eps$-fibration. Thus, replacing $q_i$ by $q_i\circ h_i$ and
  $f_i$ by $h_i\inv \circ f_i$, and $\eps$ by $\eps/2$, the proof of the special
  case shows that
  \[
   (f_1)_*\inv (h_1)_*\tau(h_1\inv\circ f_1) - (f_2)_*\inv (h_2)_* 
\tau(h_2\inv\circ f_2) 
  \] 
  is in the image of the assembly map. Now use the composition rule together
  with the fact that cell-like maps have zero Whitehead torsion.
  This finishes the proof of Lemma~\ref{lem:well_definition_of_tight_torsion}.
\end{proof}

%%%%%%%%%%%%%%%%%%%%%%%%%%%%%%%%%%%%%%%%%%%%%%%%%%%%%%%%%%%%%%%%%%%%%%%%%%%%%%%

\subsection{Relating tight torsion to previously defined torsion invariants}
\label{subsec:Relating-tight-torsion_to_previously_defined_torsion_invariants}

In~\cite{Farrell-Lueck-Steimle(2010)}, the authors defined obstructions
$\Theta(p)$ and $\tau_{\fib}(p)$ to (actually) fibering a given map $p\colon
M\to B$ where the homotopy fiber $F_p$ of $p$ is homotopy finite. For simplicity 
let us assume that $M$, $B$, and $F_p$ are path-connected. The element
$\tau_{\fib}(p)$ is defined whenever $\Theta(p)=0$ and lives in the cokernel of
\[
 \Wh(\pi_1F_p)\xrightarrow{\chi(B)\cdot i_*} \Wh(\pi_1(M))
 \]
where $i \colon F_p \to M$ is the canonical map and $\chi(B)\in\IZ$ denotes the 
Euler
characteristic of $B$.

\begin{lemma} \label{lem:i_ast_factories}
There is a factorization
\[
i_* \colon 
\Wh(\pi_1(F_p))  \to H_1^{\pi}(\widetilde{B} ;\bfWh(p)) 
\to H_1^{\pi}(\pt;\bfWh(p)) = \Wh(\pi_1(M))
\]
of the map induced by $i$.
\end{lemma}
\begin{proof} 
In the sequel we will use the notation of Section~\ref{sec:Assembly_maps}.
The canonical map $i\colon F_p\to M=M(\pi/\pi)$ may be lifted to a map $i'\colon 
F_p\to M(\pi/1)$ (by choosing a lift of the homotopically constant map $F_p\to 
X$ to a map $F_p\to \widetilde X$). This lift induces a map on Whitehead groups
\[\Wh(\pi_1(F_p))\to \Wh(\pi_1(M(\pi/1)))\cong H_1^\pi(\pi;\bfWh(p)).\]
As $B$ is path-connected, there is up to homotopy a unique $\pi$-map $\pi\to 
\tilde B$. It induces a map
\[H_1^\pi(\pi;\bfWh(p))\to H_1^\pi(\widetilde B;\bfWh(p)).\]
The composite of these two maps is by definition the first map in the 
factorization of $i_*$. By construction, composing this map with the map induced 
by the projection $\widetilde B\to B$ yields the map $i_*$.
\end{proof}

Because of Lemma~\ref{lem:i_ast_factories} we obtain a well-defined projection
\begin{multline}
\pr \colon \cok\big( \chi(B)\cdot i_* \colon \Wh(\pi_1(F_p)) \to 
\Wh(\pi_1(M))\bigr)
\\
\to \cok\big(i_* \colon \Wh(\pi_1(F_p)) \to \Wh(\pi_1(M))\bigr)
\\
\to 
\cok\left(H_1^{\pi}(\widetilde{B} ;\bfWh(p))  
\to H_1^{\pi}(\pt;\bfWh(p))\right) :=\NWh(p).
\label{projection_cok(i_ast)_to_NWh(p)}
\end{multline}

\begin{proposition}
  Suppose that the homotopy fiber $F_p$ of $p$ is homotopy finite, and that
  $\Theta(f)=0$. Then $\tau_{\fib}(p)$ maps to $N\tau(p)$ under 
  the projection~\eqref{projection_cok(i_ast)_to_NWh(p)}.
\end{proposition}
\begin{proof}
  Factorize $p$ into a homotopy equivalence $\lambda$ followed by a fibration
  $q\colon E\to B$. Recall from~\cite{Farrell-Lueck-Steimle(2010)} that in our
  situation $E$ carries (after making several choices) a preferred simple
  structure, i.e., a preferred homotopy equivalence $\varphi\colon X\to E$ from
  some compact {ENR}. Moreover $\tau_{\fib}(p)$ is represented by the Whitehead
  torsion of the composite homotopy equivalence from $M$ to $X$.

  Now consider the factorization of $q$ into a homotopy equivalence $\varphi$ 
and an
  approximate fibration from 
Theorem~\ref{thm:factorization_to_an_approximate_fibration}. From the
  inductive construction in the proof of that Theorem, it follows 
  (see~\cite[Proposition~8.1]{Steimle(2011)} for more details) that $\varphi$ 
represents the simple structure on
  $E$. So both $N\tau(p)$ and $\tau_{\fib}(p)$ are represented by the Whitehead 
torsion of
  the same homotopy equivalence.
\end{proof}

%%%%%%%%%%%%%%%%%%%%%%%%%%%%%%%%%%%%%%%%%%%%%%%%%%%%%%%%%%%%%%%%%%%%%%%%%%%%%%%

\subsection{Proof of Theorem~\ref{the:interpretation_in_terms_of_Q_manifolds}}
\label{subsec:Proof_of_Theorem_ref(the:interpretation_in_terms_of_Q_manifolds)}

We conclude this section by giving the 

\begin{proof}[Proof of Theorem~\ref{the:interpretation_in_terms_of_Q_manifolds}]
  Choose a factorization up to homotopy
  \[
   \xymatrix{
    M \ar[rr]^f \ar[rd]_p && E \ar[ld]^q\\
    & B }
   \]
  into a homotopy equivalence and an approximate fibration, where
  $E$ is a compact {ENR}, see 
Theorem~\ref{thm:factorization_to_an_approximate_fibration}.
  Since any compact ENR receives a cell-like map (and
  hence an approximate fibration) from a compact manifold with 
  boundary~\cite[Corollary~11.2]{Lacher(1977)}, we may assume that $E$ is a 
compact manifold with
  boundary. By assumption,
  $\tau(f)$ is the image of some element $\tau'$ under the assembly 
map~\eqref{ass_Davis-Lueck_widetildeB}
  and hence by 
Lemma~\ref{lem:identifying_assembly_maps}~\ref{lem:identifying_assembly_maps:same_image}  
  under the assembly map~\eqref{Quinn_assembly}.
  By Quinn's Thin $h$-Cobordism Theorem~\cite[1.2]{Quinn(1982a)}, there is a
  controlled $h$-cobordism $W$ from $E$ to some other compact manifold $E'$,
  such that the controlled torsion of $(W, E)$ equals $-\tau'$. By 
  Lemma~\ref{lem:condition_for_epsilon_fibration}, the map $W\to B$ 
  is an approximate fibration, hence the composite $W\times Q\to
  W\to B$ is also an approximate fibration. As the resulting map $M\to W$ has
  Whitehead torsion zero, the map $M\times Q\to W\times Q$ is homotopic to a
  homeomorphism~\cite[Main Theorem]{Chapman(1974)} and hence $M\times Q$ 
approximately fibers  over $B$.
  This finishes the proof of 
Theorem~\ref{the:interpretation_in_terms_of_Q_manifolds}.
\end{proof}

%%%%%%%%%%%%%%%%%%%%%%%%%%%%%%%%%%%%%%%%%%%%%%%%%%%%%%%%%%%%%%%%%%%%%%%%%%%%%%%
%%%%%%%%%%%%%%%%%%%%%%%%%%%%%%%%%%%%%%%%%%%%%%%%%%%%%%%%%%%%%%%%%%%%%%%%%%%%%%%
%%%%%%%%%%%%%%%%%%%%%%%%%%%%%%%%%%%%%%%%%%%%%%%%%%%%%%%%%%%%%%%%%%%%%%%%%%%%%%%

\section{Proof of the Stabilization Theorem~\ref{the:theorem_is_stably_true}  
for finite homotopy fiber}
\label{sec:stably_true_finite_fiber}

Let $p \colon M \to B$ be a $\pi_1$-surjective
map of closed manifolds such that 
$B$ is PL and 
aspherical, and  we assume that the $L$-theoretic FJC  holds for $\pi_1(B)$, 
see Definition~\ref{def:L-theoretic_Farrell-Jones_Conjecture}. 
This section is entirely devoted to the proof of 
Theorem~\ref{the:theorem_is_stably_true}
under the stronger assumption that the homotopy fiber of $p$ is finite.

%%%%%%%%%%%%%%%%%%%%%%%%%%%%%%%%%%%%%%%%%%%%%%%%%%%%%%%%%%%%%%%%%%%%%%%%%%%%%%%

\subsection{$s$-split factorization}
\label{subsec:s-split_factorization}

\begin{notation}
  Let $\overline{M}$ denote the normal covering space of $M$ corresponding to
  $\ker(p_*\colon \pi_1(M) \to \pi_1(B))$ and $\widetilde{B}$ be the universal 
covering of
  $B$. Put $G=\pi_1(B)$ and  $E=\widetilde{B}\times_G \overline{M}$. Let 
$\widehat{p}\colon E\to B$ be the induced
  fiber bundle with fiber $\overline{M}$. Finally let $q\colon E\to M$ be the 
induced fiber
  bundle with fiber $\widetilde{B}$. Since $\widetilde{B}$ is contractible, $q$ 
is a homotopy
  equivalence.
  
  Finally, let $\calt$ be a combinatorial triangulation of $B$, determining a PL structure on $B$.
\end{notation}

Consider the following diagram:
\[ 
\xymatrix{
  & \overline{M}\ar[d]\\
  {\widetilde{B}}\ar[r] & E=\widetilde{B}\times_G \overline{M} 
\ar[r]^(.65)q_(.65)\simeq \ar[d]^{\widehat{p}} & M \ar[ld]^p\\
  & B }
\]
An easy exercise shows that the triangle in the diagram commutes after applying
$\pi_1$. It follows that it commutes up to homotopy since $B$ is aspherical.

\begin{definition}[$s$-split] \label{def:s-split}
  Let $s$ be a positive integer. 
The map
  $q$ is \emph{$s$-split} (with respect to $\mathcal T$) if there exists a 
homotopy
  inverse
  \[
   f\colon M\times T^s\to E\times T^s
  \] 
  of $q_s=q\times\id_{T^s}$ (called an
  $s$-splitting of $q$ relative to $\mathcal T$) and a collection of compact 
submanifolds
  $M_\sigma$, $\sigma\in\mathcal T$ of $M\times T^s$, so that for all simplices
  $\sigma\in\calt$ we have:
  \begin{enumerate}
  \item If $\tau$ is a proper face of $\sigma$, then $M_\tau\subset\partial 
M_\sigma$,
  \item $M_\sigma=f\inv (\widehat{p}\inv(\sigma)\times T^s)$, and
  \item $f$ restricts to a homotopy equivalence of pairs
    \[
    f_\sigma\colon (M_\sigma, \partial 
M_\sigma)\to(\widehat{p}\inv(\sigma)\times T^s, 
    \widehat{p} \inv(\partial \sigma)\times T^s).
    \]
  \end{enumerate}
\end{definition}

The following diagram illustrates the situation, the dotted map $g$ will appear 
later.
\begin{equation*}
  \xymatrix{E\times T^s \ar[r]^{q_s} \ar@/^3ex/@{.>}[r]^g \ar[d]^{\widehat{p}_s}
& 
M\times T^s \ar@/^/[l]^f  \ar[ld]^{p_s}\\
    B\times T^s
  }\end{equation*}

\begin{notation}[Suppressing orientation homomorphisms]
  \label{rem:Suppressing_orientation_homomorphisms}
  We will surpress the orien\-ta\-tion homomorphisms in the notation of the 
$L$-groups and for
  instance write just $L_n(\IZ[\pi_1(M)])$ instead of
  $L_n(\IZ[\pi_1(M)],w_1(M))$  for a connected closed manifold $M$.
\end{notation}

\begin{theorem}\label{the:existence_of_s_splitting}
  Suppose that the homotopy fiber $F_p$ is finitely dominated.  
Then there exists a positive integer $s$ such that $q\colon E\to M$ is 
$s$-split.
\end{theorem}

\begin{remark}\label{rem:subdividing_leaves_s_fixed}
By using Wall's 2-sided separating codimension 1 splitting theorem~\cite[Theorem 12.1]{Wall(1999)}, 
we see that the triangulation $\calt$ has a subdivision $\calt'$ of arbitrarily small mesh 
such that $q$ is also $s$-split (for the same integer $s$) with respect to $\calt'$.
\end{remark}

\begin{proof}
  We will proceed by an argument similar to one originated by Quinn in his
  thesis~\cite{Quinn(1969)}, cf.~\cite{Quinn(1970)}. Let $n=\dim M$ and 
$\varphi\colon
  M\to E$ be a homotopy inverse to $q$. We may assume, after a homotopy, that 
$\varphi$ is
  transverse to each submanifold $\widehat{p}\inv(\sigma)$ of $E$, where 
$\sigma$ denotes
  a simplex of $\calt$; i.e., $\widehat{p}\circ\varphi$ is transverse to each
  $\sigma\subset B$~\cite{Freedman-Quinn(1990), Kirby-Siebenmann(1977)}.  Put 
$M_\sigma:=(\widehat{p}\circ
  \varphi)\inv(\sigma)$ and let
  \[
  \varphi_\sigma\colon M_\sigma\to \widehat{p}\inv(\sigma)
  \]
  be the restriction of $\varphi$.

  We first complete the proof of Theorem~\ref{the:theorem_is_stably_true} under 
the extra
  assumption that the homotopy fiber $F_b$ is a finite complex; we will show 
afterwards
  how the general case follows easily from the restricted case.

  We conclude from~\cite[Lemma 11.3]{Lueck-Steimle(2016splitasmb)}, whose proof 
carries over to the $L$-theory case word by word,
  that the assembly map
  \[
  H_*^\pi(\underline{\underline{E}}\pi;\bfLinfty(\widehat{p}))\to 
\Linfty_{*}(\IZ[\pi_1(M)])
  \] 
  is an isomorphism for all $* \in \IZ$. Since every virtually cyclic subgroup of the 
torsionfree group $\pi$ is infinite cyclic,
  the relative assembly map
  \[
  H_*^\pi(E\pi ;\bfLinfty(\widehat{p}))    \xrightarrow{\cong} 
H_*^\pi(\underline{\underline{E}}\pi;\bfLinfty(\widehat{p}))
  \] 
  is bijective for all $* \in \IZ$ by~\cite[Lemma~4.2]{Lueck(2005heis)}.
 This together with the assumption that $B$ is aspherical implies
  that the assembly map
   \[
   H_*^\pi(\widetilde{B};\bfLinfty(\widehat{p}))\to \Linfty_{*}(\IZ[\pi_1(M)]) 
   \]
    is bijective for all $* \in \IZ$. Since this
   assembly map is isomorphic to Quinn's assembly map  by
  Lemma~\ref{lem:identifying_assembly_maps}~\ref{lem:identifying_assembly_maps:iso},
  we can assume in the sequel that  Quinn's assembly map 
  \[
   \IH_*(B;\bfLinfty(\widehat{p}))\to  \Linfty_{*}(\IZ[\pi_1(M)])
  \]
  is bijective for all $* \in \IZ$.

Although the manifold dimension of $E$ is greater than $n$, our assumption 
implies that
it is a $n$-dimensional Poincar\'e complex. In this situation,
\[
\calp=\{\varphi_\sigma\colon M_\sigma \to \widehat{p}\inv(\sigma); 
\sigma\in\calt\}
\]
is a conglomerate surgery problem to which we can assign a conglomerate surgery 
obstruction
$\sigma(\calp) \in \IH_n(B; \bfL^h(\widehat{p}))$.

Since $\calp$ assembles to $\varphi$, which is a homotopy equivalence, the image 
of $\sigma(\calp)$
under the assembly map
\[
\IH_n(B;\bfL^h(\widehat{p}))\to  L^h_n(\pi_1(M))
\]
is zero. 

There is a sequence of homomorphisms
\[
\IH_n(B;\bfL^h(\widehat{p}))  = \IH_n(B;\bfL^{\langle 0 \rangle}(\widehat{p})) 
\to \IH_n(B;\bfL^{\langle -1 \rangle}(\widehat{p}))  \to 
\IH_n(B;\bfL^{\langle -2\rangle}(\widehat{p})) \to \cdots
\]
whose colimit is $\IH_n(B;\bfL^{\langle -\infty \rangle}(\widehat{p}))$. The 
assembly maps fit into a commutative diagram
\[
\xymatrix{\IH_n(B;\bfL^h(\widehat{p})) \ar[r] \ar[d]  
&
L^h_n(\pi_1(M)) \ar[d]
\\
\IH_n(B;\bfL^{\langle -i \rangle}(\widehat{p})) \ar[r] \ar[d]  
&
L^{\langle -i \rangle}_n(\pi_1(M)) \ar[d]
\\
\IH_n(B;\bfL^{\langle -\infty \rangle}(\widehat{p})) \ar[r] 
&
L^{\langle -\infty \rangle}_n(\pi_1(M))
}
\]
Since the bottom horizontal arrow is injective, there exists $m \ge 0$ such that
the image of $\sigma(\calp)$ under the map
\[
\IH_n(B;\bfL^h(\widehat{p}))  \to \IH_n(B;\bfL^{\langle -m 
\rangle}(\widehat{p}))
\]
is trivial. Crossing with $T^m$ yields a map
\[
\IH_n(B;\bfL^{\langle -m \rangle}(\widehat{p})) \to 
\IH_{n+m}(B;\bfL^{h}(\widehat{p}_m)).
\]
Consider the new conglomerate surgery problem
\[
\calp_m=\{\varphi_\sigma\times\id_{T^m}\colon M_\sigma\times T^m \to \widehat
p\inv(\sigma)\times T^m; \sigma\in\calt\}
\] 
Its  conglomerate surgery obstruction $\sigma(\calp_m) \in  
\IH_{n+m}(B;\bfL^{h}(\widehat{p}_m))$
is the image of $\sigma(\calp)$ under the composite
\[
\IH_n(B;\bfL^h(\widehat{p}))  \to \IH_n(B;\bfL^{\langle -m 
\rangle}(\widehat{p}))
\to \IH_{n+m}(B;\bfL^{h}(\widehat{p}_m))
\]
and hence trivial. This implies  that there exists a conglomerate surgery 
problem
\[
\IP=\{\psi_\sigma \colon W_\sigma\to\widehat{p}\inv(\sigma)\times T^m\times 
[0,1];
\sigma\in \calt\}
\]
such that $\partial^- W_\sigma= M_\sigma\times T^m$ and
\[
\psi_\sigma\vert_{\partial^- W_\sigma}\colon M_\sigma\times T^m\to \hat
p\inv(\sigma)\times T^m\times 0\] is $\varphi_\sigma\times\id_{T^m}$. 

Furthermore, if we put
\[
\cals=\{\varphi'_\sigma:=\psi_\sigma\vert_{\partial^+ W_\sigma}\colon \partial^+
W_\sigma \to \widehat{p}\inv(\sigma)\times T^m\times 1; \sigma \in \calt\}
\] 
then each $\varphi'_\sigma$ is a homotopy equivalence; i.e., $\cals$ is a 
conglomerate
homotopy-topological structure, and we denote the homotopy-topological structure 
on
$E\times T^m$ that $\cals$ assembles to by
\[
\varphi'\colon (M')^{n+m}\to E\times T^m.
\]

By topologically assembling $\IP$, we obtain a (single) surgery problem
\[
\psi\colon W^{n+m+1}\to E\times T^m\times [0,1]
\]
satisfying
\begin{enumerate}
\item $\psi\vert_{\partial^- W}\colon M\times T^m\to E\times T^m\times 0$ is
  $\varphi\times\id_{T^m}$;
\item $\psi\vert_{\partial^+ W}\colon (M')^{n+m}\to E\times T^m$ is $\varphi'$.
\end{enumerate}
This surgery problem determines an element $\theta\in L^h_{n+m+1}(\pi_1(M)\times 
\IZ^m)$.

Since the assembly map 
\[
\IH_{n+m+1}(B;\Linfty(\widehat{p}_m))\to \Linfty_{n+m+1}(\pi_1(M) \times \IZ^m)
\]
is an epimorphism, we see by an argument similar to the one for $T^m$
that after taking the product with an additional torus $T^t$ of
sufficiently large dimension $t$, there is another conglomerate surgery problem
\[
\IP'=\{\eta_\sigma\colon W'_\sigma\to \widehat{p}\inv(\sigma)\times T^m\times 
T^t\times
[1,2];\sigma\in\calt\}
\] 
such that $\partial^- W'_\sigma = \partial^+ W_\sigma$ and
\[
\eta_\sigma\vert_{\partial^- W'_\sigma}\colon \partial^+ W_\sigma \to \widehat
p\inv(\sigma)\times T^m\times T^t \times 1
\] 
is $\varphi'_\sigma\times \id_{T^t}$. Furthermore
\[
\cals'=\{\varphi''_\sigma:=\eta_\sigma\vert_{\partial^+
  W'_\sigma}\colon \partial^+W'_\sigma
  \to \widehat{p} \inv(\sigma)\times T^m\times T^t\times 2; \sigma\in\calt\}
\] 
is a conglomerate homotopy-topological structure which assembles to a
homotopy-topological structure on $E\times T^m\times T^t$ denoted by
\[
\varphi''\colon (M'')^{n+m+t}\to E \times T^m\times T^t.
\]
Assembling $\IP'$ yields a surgery problem
\[
\eta\colon (W')^{n+m+t+1}\to E\times T^m\times T^t\times [1,2]
\]
which represents the image of $-\theta$ in 
$L^h_{n+m+t+1}(\pi_1(M)\times\IZ^m\times \IZ^t)$
(under the natural map in the Wall-Shaneson formula), and satisfies
\begin{enumerate}
\item $\eta\vert_{\partial^- W'}\colon (M')^{n+m}\times T^t\to E\times T^m\times 
T^t$ is
  $\varphi'\times\id_{T^t}$;
\item $\eta\vert_{\partial^+ W'}\colon (M'')^{n+m+t} \to E\times T^m\times T^t$ 
is
  $\varphi''$.
\end{enumerate}

Glueing together the two surgery problems $W\times T^t$ and $W'$ along 
$\partial^+ W\times T^t=\partial^- W'$ yields a surgery problem representing 0 
in 
$L^h_{n+m+t+1}(\pi_1(M)\times \IZ^{m+t})$. Hence $W\times T^t\cup W'$ can be 
surgered, 
without touching its boundary, so
as to yield an $h$-cobordism $C$ between $\partial^- C= M\times T^{m+t}$ and 
$\partial^+ C= \partial^+ W'=(M'')^{n+m+t}$. Since $C\times S^1$ is an 
$s$-cobordism, 
$M\times T^{m+t}\times S^1$ is homeomorphic to $(M'')^{n+m+t}\times S^1$ via a 
homeomorphism
$g\colon M\times T^{m+t+1}\to M''\times S^1$ such that the composition
\[
f\colon M\times T^{m+t+1}\xrightarrow{g} M''\times S^1 
\xrightarrow{\varphi''\times\id}
E\times T^{m+t+1}
\] 
is homotopic to $\varphi\times\id_{T^{m+t+1}}$. Thus $q$ is $s$-split
for $s=m+t+1$ by the map $f$.

This completes the proof of the restricted case of
Theorem~\ref{the:theorem_is_stably_true}, i.e., assuming that $F_b$ is a finite
complex. However in the general case we assume that $F_b$ is dominated by a 
finite
complex. But this at least implies that $F_b\times S^1$ has the homotopy type of 
a finite
complex and consequently the homotopy fiber of the composite map
\[
M\times S^1\to M\xrightarrow{p} B
\] 
is a finite complex. Therefore the restricted form
of Theorem~\ref{the:theorem_is_stably_true} applies, showing that 
$q\times\id_{S^1}$ is
$s$-split. Consequently, $q$ is $(s+1)$-split. This finishes the proof of
Theorem~\ref{the:existence_of_s_splitting}.
\end{proof}

\begin{remark}[Bounds for $s$]\label{rem:number_of_stabilizations_needed}
  An analysis of the proof of Theorem~\ref{the:existence_of_s_splitting} 
 shows that if $\dim(M)\geq 5$, then $s$ can be taken to be
  $2(k+1)+2$ if $k$ has the property that $H_*^\pi(\widetilde{B}; \bfL^{\langle 
-k\rangle}(p))
  \xrightarrow{\cong} H_*^\pi(\widetilde{B};\bfL^{\langle -\infty\rangle}(p))$ 
and
  $L_*^{\langle -k\rangle}(\pi_1(M))\xrightarrow{\cong} \Linfty_*(\pi_1(M))$ 
under the
  canonical maps. If the \emph{total space} of $p$ satisfies the $K$-theoretic 
FJC, then
  $k=1$ such that $s$ can be taken to be $6$ (and even $5$ in the homotopy finite case).
\end{remark}

%%%%%%%%%%%%%%%%%%%%%%%%%%%%%%%%%%%%%%%%%%%%%%%%%%%%%%%%%%%%%%%%%%%%%%%%%%%%%%%

\subsection{Gaining control over the torus}
\label{subsec:Gaining_control_over_the_torus}

Now fix a triangulation $\calt$ and an $s$-splitting $f$ of $q$ relative to 
$\calt$. Using a cofibration argument, there
is a left homotopy inverse
\[
g\colon E\times T^s\to M\times T^s
\] 
to $f$ and a homotopy $g\circ f\simeq_H
\id_{M\times T^s}$, such that the restriction 
\[
g_\sigma \colon (\widehat{p}\inv(\sigma)\times
T^s, \widehat{p}\inv(\partial\sigma)\times T^s)\to (M_\sigma, \partial M_\sigma)
\] 
of $g$ is a left homotopy inverse to $f_\sigma$, via a homotopy $H_\sigma$ which 
is the restriction of
$H$.  If the mesh of
$\mathcal T$ is less than $\eps$, it follows the map $g$ is an $\eps$-domination 
over $B$
in the sense of Definition~\ref{def:controlled_maps},
where the control map from $M\times T^s$ to $B$ is the composite of $f$ with the
projection from $E\times T^s$ to $B$.

If we knew that $g$ was an $\eps$-domination over $B\times T^s$, then we could 
apply
Lemma~\ref{lem:condition_for_epsilon_fibration} to conclude that 
$\widehat{p}_s\circ f$ were a
$\delta$-fibration. But unfortunately there is no control over the torus yet.

\begin{lemma}\label{lem:bounded_over_torus}
  Suppose that the homotopy fiber $F_p$ is homotopy finite.
  There is a $4\eps$-domination $g$ over $B$ in such a way that $g$ is 
\emph{bounded} over
  the torus, i.e., for one (and hence any) lift $\widetilde{g}\colon E\times 
\IR^s\to M\times \IR^s$ of $g$ along 
   the universal coverings $\IR^s\to \IR^s/\IZ^s=T^s$ of the torus and for one 
(and hence all)
   lifts $\widetilde{f} \colon M \times \IR^s \to E \times \IR^s$ of $f$ 
  there exists $N>0$ such that
  \[
  d(\pi\circ\widetilde{f}\circ\widetilde{g}(e,x), x)<N\quad \forall (e,x)\in 
E\times \IR^s.
  \]
  where $\pi$ is the projection $E\times \IR^s\to \IR^s$.
\end{lemma}

\begin{proof}
  By Theorem~\ref{thm:factorization_to_an_approximate_fibration} and
  Lemma~\ref{lem:condition_for_epsilon_h_eq}, the map $\widehat{p}_s$ is 
$\eps$-homotopy
  equivalent to an approximate fibration $E'\to B\times T^s$ where $E'$ is 
compact. In
  particular the identity map on $E\times T^s$ is $\eps$-homotopic (over 
$B\times T^s$) to
  a map $z \colon E\times T^s \to E\times T^s$ which factors 
  into a compact subset $C\subset E\times T^s$. Then, by the
  triangle inequality, the map $g':=g\circ z$ is $3\eps$-controlled over $B$, 
and $g'\circ
  f$ is $4\eps$-homotopic to the identity map (over $B$). Thus, $g'$ is a
  $4\eps$-domination over $B$. Choosing a lift $\widetilde{z}$ of $z$ along the 
universal
  covering of the torus, we have
\[
  d(\pi\circ\widetilde{f}\circ\widetilde{g'}(e,x),x) \leq 
d(\pi\circ\widetilde{f}\circ\widetilde{g}
  \circ\widetilde{z}(e,x),\pi\circ \widetilde{z}(e,x)) + d(\pi\circ 
\widetilde{z}(e,x),x).
 \]
The second summand is less than $\eps$. Choose a compact set 
$\widetilde{C}\subset E\times \IR^s$ 
which surjects onto $C$. Then we have
\[
 \sup_{(e,x)\in E\times \IR^s} d(\pi\circ\widetilde{f}\circ\widetilde{g}
\circ\widetilde{z}(e,x),\pi\circ \widetilde{z}(e,x))\leq 
\sup_{\widetilde{c}\in \widetilde{C}} d(\pi\circ \widetilde{f}
\circ \widetilde{g}(\widetilde{c}), \pi(\widetilde{c}))<\infty.
 \]
\end{proof}

Note that as $M\times T^s$ is compact, the homotopy $H$ between $g\circ 
f$ and the identity
map is an $N$-homotopy for some $N>0$. (Here we measure the diameter of a path 
in $T^s$
also by first lifting it to the universal cover $\IR^s$.) We might call $g$ 
from 
Lemma~\ref{lem:bounded_over_torus} a ``bounded domination'' over the torus.

Lifting such a map $g$ to a map $g_k$ between coverings over the torus of index 
$k$ (i.e., to the coverings whose covering projections are determined by the expanding self-maps $x\mapsto x^k$, where $x\in T^s$), we
can improve the bound by a factor of $k$. Choosing $k$ large enough, it follows that $g_k$ 
is an
$5\eps$-domination over $B\times T^s$ (where $M\times T^s$ is controlled by 
$\widehat
p_s\circ f_k$ now). Since $E\times T^s$ is a fiber bundle over $B\times T^s$, we 
obtain
from Lemma~\ref{lem:condition_for_epsilon_fibration}:

\begin{corollary} \label{cor:getting_epsilon_fibration}
  Suppose that the homotopy fiber $F_p$ is homotopy finite.
  Then for all $\eps>0$ there is a $\delta>0$ and an $k>0$ such that the 
composite
  \[
  M\times T^s \xrightarrow{f_k} E\times T^s \xrightarrow{\widehat{p}_s} B\times 
T^s
 \]
  is an $\eps$-fibration provided that the mesh of $\calt$ is less than 
$\delta$.
\end{corollary}

Now we can complete the proof of Theorem~\ref{the:theorem_is_stably_true} with 
the help of
Corollary~\ref{cor:getting_epsilon_fibration} and Remark~\ref{rem:subdividing_leaves_s_fixed}.
Since $M\times T^s$ is a closed manifold, an approximation theorem by
Chapman~\cite{Chapman(1981)} shows that $\widehat{p}_s\circ f_k$ is homotopic to 
an
approximate fibration, provided $\eps$ was chosen small enough. Now $f_k$ and 
$f$ induce
the same map on fundamental groups, namely $\pi_1(q)\inv\times\id_{\IZ^s}$. 
Since $B$ is
aspherical, $p_s\simeq\widehat{p}_s\circ f$ is homotopic to $\widehat{p}_s\circ 
f_k$. 
This finishes the proof of Theorem~\ref{the:theorem_is_stably_true}.

%%%%%%%%%%%%%%%%%%%%%%%%%%%%%%%%%%%%%%%%%%%%%%%%%%%%%%%%%%%%%%%%%%%%%%%%%%%%%%%
%%%%%%%%%%%%%%%%%%%%%%%%%%%%%%%%%%%%%%%%%%%%%%%%%%%%%%%%%%%%%%%%%%%%%%%%%%%%%%%
%%%%%%%%%%%%%%%%%%%%%%%%%%%%%%%%%%%%%%%%%%%%%%%%%%%%%%%%%%%%%%%%%%%%%%%%%%%%%%%

\section{Proof of the Splitting Theorem~\ref{the:stably_implies_unstably}}
\label{sec:Proof_of_the_Destabilization_Theorem_ref(the:stably_implies_unstably)}

The content of this section is the proof of 
Theorems~\ref{the:stably_implies_unstably} 
and~\ref{the:obstruction_to_destabilization}.

%%%%%%%%%%%%%%%%%%%%%%%%%%%%%%%%%%%%%%%%%%%%%%%%%%%%%%%%%%%%%%%%%%%%%%%%%%%%%%%

\subsection{Proof of 
Theorem~\ref{the:stably_implies_unstably}~\ref{the:stably_implies_unstably:fac_prop}}
\label{subsec:Proof_of_Theorem_ref(the:stably_implies_unstably)_(i)}

  Let $\rho$ denote the composition
  \[
  \rho \colon M\times\IR\to M\times S^1\xrightarrow{\overline{p_1}} B\times S^1 
\to B
  \] 
   where
  $\overline{p_1}$ is the approximate fibration homotopic to $p_1$ which is 
assumed to exist. 
 Then $\rho$ is also   an approximate fibration. We wish to apply Quinn's End 
Theorem 1.1 
  from~\cite{Quinn(1982a)} to complete $\rho$ at ``$+\infty$''. 
  The end is tame since $M\times\IR$ is an approximate fibration over
  $B\times\IR$. 
  For the same reason, the end has a locally constant fundamental groupoid, 
which is the fundamental groupoid of the homotopy fiber of $\rho$, and hence, of 
$p$.
  
  Denoting by $p'\colon E\to B$ a fibration controlled equivalent to $p$, the 
end obstruction to obtain a completion is therefore an element
  \[
   q_0(\rho)\in \IH_0(\widetilde B;\bfWh(p'))
  \]
  whose image under Quinn's assembly map
  \[
  \IH_0(\widetilde B;\bfWh(p')) \to  
\Wh_0(\pi_1(E))\cong\widetilde{K_0}(\IZ[\pi_1(M)])
  \]
  is Siebenmann's obstruction to adding a boundary (see~\cite{Siebenmann(1965)} and the paragraph preceding Proposition 1.7 
in~\cite{Quinn(1982a)}).

The
  latter obstruction is zero since $M\times\IR$ compactifies to $M\times
  [-\infty,+\infty]$.  Since we assume that the $K$-theoretic FJC holds for 
$\pi_1(B)$ and
  $B$ is aspherical, the assembly map $H_0^{\pi}(\widetilde B;\bfWh(p')) \to
  \Wh_0(\pi_1(M))$ is injective. Now we conclude from
  
Lemma~\ref{lem:identifying_assembly_maps}~\ref{lem:identifying_assembly_maps:iso} that
  Quinn's assembly map $\IH_0(\widetilde B;\bfWh(p')) \to \Wh_0(\pi_1(M))$
     is injective and hence
  $q_0(\rho) = 0$. Let
  \[
   \overline{\rho}\colon\calw\to B
  \]
  denote a completion of $\rho\colon M\times\IR\to B$ given by the End
  Theorem. It is also an approximate fibration. Now let $W$ denote the compact
  $h$-cobordism connecting $M$ to $\partial\calw$ inside $\calw$. (Identify $M$
  with $M\times 0$.)
  
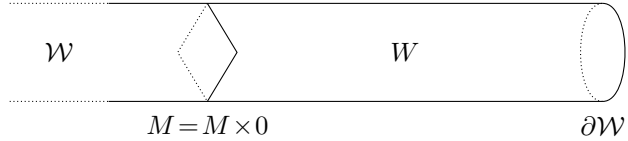
\begin{center}
\begin{tikzpicture}
[scale=1.3,
% Styles
line/.style={thin}, curved line/.style={decorate,
  decoration={snake,amplitude=.4mm}}, dashed line/.style={loosely dashed},
dotted line/.style={densely dotted}]

\draw [dotted line] (0,0) -- (1,0); \draw [dotted line] (0,1) -- (1,1); \draw
[line] (1,0) -- (6,0); \draw [line] (1,1) -- (6,1); \draw (.5,.5) node{$\calw$};
\draw (4,.5) node{$W$}; \draw (2,-.25) node{$M\!=\!M\!\times\!0$};

\draw [line] (2,0) -- (2.3,.5) -- (2,1); \draw [dotted line] (2,1) -- (1.7,.5)
-- (2,0);

\draw [line] (6,0) .. controls (6.3,0) and (6.3,1) .. (6,1); \draw [dotted line]
(6,0) .. controls (5.7,0) and (5.7,1) .. (6,1); \draw (6,-.25)
node{$\partial\calw$};

\end{tikzpicture}
\end{center}

So there is an $h$-cobordism from $M$ to some closed manifold 
$N:=\partial\calw$,
which is a compact ENR and maps to $B$ by the approximate fibration
\[
\overline{\rho}\vert_{\partial\calw}\colon\partial\calw\to B.
\]
 This finishes the proof of 
Theorem~\ref{the:stably_implies_unstably}~\ref{the:stably_implies_unstably:fac_prop}.

%%%%%%%%%%%%%%%%%%%%%%%%%%%%%%%%%%%%%%%%%%%%%%%%%%%%%%%%%%%%%%%%%%%%%%%%%%%%%%%

\subsection{Proof of 
Theorem~\ref{the:stably_implies_unstably}~\ref{the:stably_implies_unstably_approx}}
\label{subsec:Proof_of_Theorem_ref(the:stably_implies_unstably)_(ii)}

(ii) Denote by
\[
\bar{}\;\colon \Wh(G)\to\Wh(G)
\]
the involution coming from $w_1(M)$-twisted involution $a\mapsto
\overline{a}$ on $\IZ G$. It sends a matrix to its conjugate-transpose.

If we let $x=\tau(W,M)$, then
\[
\tau(f)=x - (-1)^n \overline{x} 
\]
where $f\colon M\to \partial\calw$ is the homotopy equivalence determined by $W$
and $n=\dim M$. Next note that the class $\tau(f)$ in the cokernel of the 
assembly map
\[
\alpha=H_1^\pi(\pr;\bfWh(p))\colon H_1^\pi(\widetilde B;\bfWh(p))\to 
\Wh(\pi_1(M))
\]
is $N\tau(p)$ and hence vanishes by assumption.

Hence the class of $x$ in $\NWh(p)$ defines an element in the Tate 
cohomology group $\widehat H^n(\IZ/2;\NWh(p))$. It follows that we can write
\[
x=z+(-1)^n\overline{z} +a
\]
where $a\in\im(\alpha)$ because of our
assumption on the vanishing of the Tate cohomology group. We can change
the embedding of $M$ into $M\times\IR$ so that 
\[
x=a\in\im(\alpha)
\]
by inserting an $h$-cobordism with torsion $-z-(-1)^n\overline{z}$ to the left 
of $M\times 0$.

\begin{center}
\begin{tikzpicture}
[scale=1.3,
% Styles
line/.style={thin}, curved line/.style={decorate,
  decoration={snake,amplitude=.4mm}}, dashed line/.style={loosely dashed},
dotted line/.style={densely dotted}]

\draw [line] (6,0) -- (9,0); \draw [line] (6,1) -- (9,1);

\draw (7.5,.5) node{old $W$};

\draw [line] (6,0) -- (6.3,.5) -- (6,1); \draw [dotted line] (6,1) -- (5.7,.5)
-- (6,0); \draw (6,-.25) node{$M\!=\!M\!\times\!0$};

\draw (6.4,.5) node{$\longrightarrow$}; \draw (6.45,.3) node{$x$};

\draw (5.6,.5) node{$\longleftarrow$}; \draw (5.55,.3) node{$-z$};

\draw [line] (9,0) .. controls (9.3,0) and (9.3,1) .. (9,1); \draw [dotted line]
(9,0) .. controls (8.7,0) and (8.7,1) .. (9,1); \draw (9,-.25)
node{$\partial\calw$};

\draw [line] (6,0) -- (4.5,-.25) -- (3,0); \draw [line] (6,1) -- (4.5,1.25) --
(3,1);

\draw [line] (4.5,-.25) .. controls (5,-.25) and (5,1.25) .. (4.5,1.25); \draw
[dotted line] (4.5,-.25) .. controls (4,-.25) and (4,1.25) .. (4.5,1.25);

\draw [line] (3,0) -- (3.3,.5) -- (3,1); \draw [dotted line] (3,1) -- (2.7,.5)
-- (3,0); \draw (3,-.25) node{$M$};

\draw (3.4,.5) node{$\longrightarrow$}; \draw (3.45,.3) node{$-z$};

\draw [snake=brace,segment amplitude=1ex] (3,1.4) -- (9,1.4); \draw (6,1.65)
node{new $W$};

\end{tikzpicture}
\end{center}

Lemma~\ref{lem:identifying_assembly_maps}~\ref{lem:identifying_assembly_maps:same_image} shows
  that Quinn's assembly map $\IH_1(\widetilde B;\bfWh(p)) \to \Wh(\pi_1(M))$ 
and 
  $\alpha=H_1^\pi(\pr;\bfWh(p))\colon H_1^\pi(\widetilde B;\bfWh(p))\to 
\Wh(\pi_1(M))$ have the same image.
  Hence $x$ lies in the image of Quinn' assembly map as well.

  We now use Quinn's $h$-cobordism Theorem 1.2 (b) together with his 
  End Theorem~1.1 (b) from~\cite{Quinn(1982a)} to change the completion of
  \[
\rho\colon M\times \IR\to B
 \]
 so that the torsion $\tau(W,M)$ changes by adding to it the element $-a$.
 Hence we may assume after doing this that $x=0$ so $W\cong M\times [0,1]$; in
 particular that $M\cong \partial\calw$.

This finishes the proof of Theorem~\ref{the:stably_implies_unstably}.

%%%%%%%%%%%%%%%%%%%%%%%%%%%%%%%%%%%%%%%%%%%%%%%%%%%%%%%%%%%%%%%%%%%%%%%%%%%%%%%
%%%%%%%%%%%%%%%%%%%%%%%%%%%%%%%%%%%%%%%%%%%%%%%%%%%%%%%%%%%%%%%%%%%%%%%%%%%%%%%
%%%%%%%%%%%%%%%%%%%%%%%%%%%%%%%%%%%%%%%%%%%%%%%%%%%%%%%%%%%%%%%%%%%%%%%%%%%%%%%

\section{Proof of Theorem~\ref{the:obstruction_to_destabilization} about the 
MAF-Rigidity Conjecture and the splitting obstruction}
\label{sec:Proof_of_Theorem_ref(the:obstruction_to_destabilization)}

This section is devoted to the proof of 
Theorem~\ref{the:obstruction_to_destabilization}.

%%%%%%%%%%%%%%%%%%%%%%%%%%%%%%%%%%%%%%%%%%%%%%%%%%%%%%%%%%%%%%%%%%%%%%%%%%%%%%%

\subsection{The MAF Rigidity Conjecture}
\label{subsec:The_MAF_Rigidity_Conjecture}

Let us first recall the statement of the ``MAF Rigidity Conjecture'' by
Hughes-Taylor-Williams~\cite[page~568]{Hughes-Taylor-Williams(1995)}:

\begin{conjecture}[MAF Rigidity Conjecture]\label{con:MAF_rigidity_conjecture}
  Let $B$ be a closed aspherical manifold. Then two MAFs
  $p\colon M\to B$ and $q\colon N\to B$ are controlled homeomorphic if and only
  if there is a homeomorphism $h\colon M\to N$ such that $q\circ h$ is homotopic
  to $p$.
\end{conjecture}

Here, $p$ and $q$ being controlled homeomorphic means that there is a locally
trivial fiber bundle $E\to [0,1]$ with $E\times 0=M$, $E\times 1=N$, and a
fiberwise map $H\colon E\to B\times [0,1]$ which is an approximate fibration,
such that $H_0=p$ and $H_1=q$. 
(Compare~\cite[Proposition~12.17]{Hughes-Taylor-Williams(1990)} 
and~\cite[page~567]{Hughes-Taylor-Williams(1995)}.)

The next result is taken from~\cite[Theorem~1.2]{Hughes-Taylor-Williams(1995)}.

\begin{theorem}
  The MAF Rigidity Conjecture holds when $B$ is non-positively curved.
\end{theorem}

%%%%%%%%%%%%%%%%%%%%%%%%%%%%%%%%%%%%%%%%%%%%%%%%%%%%%%%%%%%%%%%%%%%%%%%%%%%%%%%

\subsection{The splitting obstruction}
\label{subsec:The_splitting_obstruction}

Suppose now that the MAF Rigidity Conjecture holds for $B$. Recall the
construction of an $h$-cobordism $W$ from $M$ to $\partial \calw$ in the proof
of Theorem~\ref{the:stably_implies_unstably}, (i). Recall also from the
beginning of the proof of part (ii) that if we let $x=\tau(W, M)$, we have
$[x]-(-1)^n[\overline{x}]=N\tau(p)=0\in \NWh(p)$.

\begin{definition}[Splitting obstruction] \label{def:splitting_obstruction}
  The \emph{splitting obstruction} 
  \[
  \kappa_0\in \widehat H^n(\IZ/2;\NWh(p))
  \] 
  is the class determined by $x=\tau(W,M)$.
\end{definition}

We have to show that the element $\kappa_0$ is well-defined, i.e., does not
depend on the choice of approximate fibration homotopic to $p_1$ nor on the
choice of completion. So suppose that $\overline{p_1}$ and $\overline{p_1}'$ are 
two
approximate fibrations homotopic to $p_1$. By passing to the infinite cyclic
cover, we obtain two maps $\rho, \rho'\colon M\times \IR \to B$; suppose that
$(\calw,\overline{\rho})$ and $(\calw',\overline{\rho}')$ are completions of 
$(M\times\IR,\rho)$ 
and $(M\times\IR,\rho')$ at $\infty$.

By the MAF Rigidity Conjecture, we obtain an approximate fibration $H\colon E\to
B\times S^1\times [0,1]$ interpolating between $\overline{p_1}$ and 
$\overline{p_1}'$. If
$\overline{E}$ denotes the infinite cyclic cover corresponding to the kernel of
$\pi_1(H)$, we obtain an approximate fibration $F\colon \overline{E}\to B$
interpolating between $\rho$ and $\rho'$. Applying Quinn's relative End Theorem
to $F$, there is a completion $(\mathcal V,\overline{F})$ of $F$ that extends 
the given
completions $(\calw,\overline{\rho})$ and $(\calw',\overline{\rho}')$, and the 
boundary component 
$\partial_0\mathcal V$ is a thin $h$-cobordism from $\partial\calw$ to $\partial 
\calw'$. (See the
left picture.)

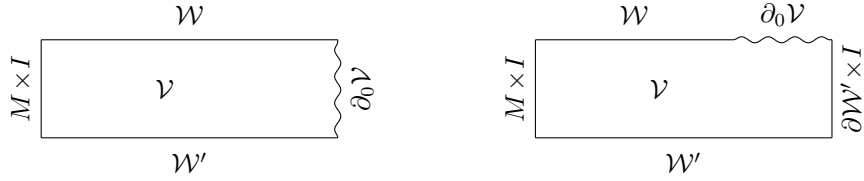
\begin{tikzpicture}
[scale=1.3,
% Styles
line/.style={thin},
curved line/.style={decorate, decoration={snake,amplitude=.4mm}}]

\draw [line] (0,0) -- (3,0);
\draw (1.5, -.25) node{$\calw'$};
\draw [line] (0,0) -- (0,1);
\draw (-.2,.5) node[rotate=90]{$M\!\times \!I$};
\draw [curved line] (3,0) -- (3,1);
\draw (3.25,.5) node[rotate=90]{$\partial_0 \mathcal V$};
\draw [line] (0,1) -- (3,1);
\draw (1.5,1.25) node{$\calw$};
\draw (1.25,.5) node {$\mathcal V$};

\draw [line] (5,0) -- (8,0);
\draw (6.5, -.25) node{$\calw'$};
\draw [line] (5,0) -- (5,1);
\draw (4.8,.5) node[rotate=90]{$M\!\times \!I$};
\draw [line] (8,0) -- (8,1);
\draw (8.2,.5) node[rotate=90]{$\partial\calw'\!\times \!I$};
\draw [line] (5,1) -- (7,1);
\draw (6,1.25) node{$\calw$};
\draw [curved line] (7,1) -- (8,1);
\draw (7.5,1.25) node{$\partial_0\mathcal V$};
\draw (6.25,.5) node {$\mathcal V$};
\end{tikzpicture}

We have
\[
\tau(\calw', M)+\tau(\calv, \calw') =  \tau(\calw, M)+\tau(\calv, \calw),
\]
as both are the Whitehead torsions of the inclusion from $M$ to $\calv$.  So,
\[
x-x'=\tau(\calv, \calw') - \tau(\calv,\calw).
\] 
But the torsions of $(\calv,\calw')$
and $(\calv,\calw)$ are related by the involution: In fact, bending around the 
corners
transforms the left picture into the right picture, where we see that
\[
\tau(\calv,\calw')=(-1)^{(n+1)}\overline{\tau(\calv,\calw\cup\partial_0\calv)}.
\]

Now, since $\partial_0\calv$ is a thin $h$-cobordism, its Whitehead torsion
becomes zero in $\NWh(p)$. Hence
\[
[\tau(\calv,\calw')]=(-1)^{(n+1)}\overline{[\tau(\calv,\calw)]}\in\NWh(p).
\]
so that 
\[
[x]-[x']=0\in \widehat H^n(\IZ/2;\NWh(p)),
\]
establishing that $\kappa_0$ is well-defined.

%%%%%%%%%%%%%%%%%%%%%%%%%%%%%%%%%%%%%%%%%%%%%%%%%%%%%%%%%%%%%%%%%%%%%%%%%%%%%%%

\subsection{Finishing the proof of 
Theorem~\ref{the:obstruction_to_destabilization}}
\label{subsec:Finishing_the_proof_of_Theorem_ref(the:obstruction_to_destabilization)}

  Suppose first that $p$ is homotopic to an approximate fibration $q$. Then the
  map $q_1=q\times \id_{S^1}$ is an approximate fibration homotopic to $p_1$,
  and (again in the notation of the proof of 
  Theorem~\ref{the:stably_implies_unstably}) we may take $\rho$ to be the 
composite
  \[
  M\times \IR\to M\xrightarrow{q} B
  \] of the projection and $q$. This map can
  obviously be completed by $M\times (-\infty, \infty]$ so that $W$ is just
  $M\times [0,\infty]$ which is a trivial $h$-cobordism. So $\tau(W, M)=0$ and
  $\kappa_0$ vanishes.

  On the other hand, if $\kappa_0=0$, then we can write 
$x=z+(-1)\inv\overline{z}$
  modulo the image of the assembly map and the very same argument as in the 
proof of part (ii) of
  Theorem~\ref{the:stably_implies_unstably} shows that $p$ is homotopic to an
  approximate fibration. Note that for $n>0$, $N\tau(p_n)=0$ since it is given by
  \[\tau(\varphi\times \id_{S^1})=\tau(\varphi)\cdot \chi(S^1)=0\]
  by~\cite{Kwun-Szczarba(1965)}.

%%%%%%%%%%%%%%%%%%%%%%%%%%%%%%%%%%%%%%%%%%%%%%%%%%%%%%%%%%%%%%%%%%%%%%%%%%%%%%%
%%%%%%%%%%%%%%%%%%%%%%%%%%%%%%%%%%%%%%%%%%%%%%%%%%%%%%%%%%%%%%%%%%%%%%%%%%%%%%%
%%%%%%%%%%%%%%%%%%%%%%%%%%%%%%%%%%%%%%%%%%%%%%%%%%%%%%%%%%%%%%%%%%%%%%%%%%%%%%%

\section{Orientability of cyclic subgroups}\label{sec:orientability}

If a torsionfree group $G$ contains the fundamental group of the Klein bottle 
$K:=Z\rtimes Z$ as a
subgroup, then the cyclic subgroups are certainly not orientable in the 
sense of Definition~\ref{def:orientable_cyclic_subgroups}.  
We conclude from~\cite[Lemma~8.7 and Lemma~8.8]{Lueck-Steimle(2016splitasmb)} 
a converse in a special situation for $G$.

\begin{theorem}\label{the:fundamental_groups_of_npc_manifolds_are_orientable}
  Let $G$ be a torsionfree group satisfying one of the following conditions
  \begin{enumerate}

  \item $G$ is hyperbolic;

  \item $G$ is a  CAT(0)-group and  satisfies the Klein bottle condition, i.e., 
$G$ does not contain the fundamental group
  $K=\IZ\rtimes \IZ$ of the Klein bottle.

\end{enumerate}

  Then the cyclic subgroups of $G$ are orientable in the sense of  
Definition~\ref{def:orientable_cyclic_subgroups}.
\end{theorem}

\begin{lemma} \label{lem:extensions_and_orientable_subgroups}
Let $1 \to K \xrightarrow{i} G \xrightarrow{p}  Q \to 1$ be an extension of 
torsionfree groups.
Suppose that for any $g \in G$ the conjugation automorphisms $K \to K$ is an 
inner automorphism
of $K$. If the cyclic subgroups of both $K$ and $Q$ are orientable in the sense 
of
Definition~\ref{def:orientable_cyclic_subgroups}, then the same is true for $G$.
\end{lemma}

\begin{proof} Fix orientations for the cyclic subgroups of $K$ and of $Q$.  Let 
$C \subset
  G$ be a cyclic subgroup of $G$. If $C= i(C')$ for $C'\subset K$, we let $g_C$ 
be the
  image of the chosen generator of $C'$. Otherwise $C\cap K=\emptyset$ so $p$ 
defines an
  isomorphism $C\to p(C)$, in which case we let $g_C$ be the preimage of the 
chosen
  generator of $p(C)$. We leave it to the reader to check that these choices are
  compatible with the requirements in 
Definition~\ref{def:orientable_cyclic_subgroups}.
\end{proof}

We close this section with a result  stated in 
Proposition~\ref{prop:klein_bottle_and_locally_symmetric_space}
and mentioned in the introduction. It is not needed elsewhere in the paper.
The key ingredient in its  proof is~\cite[Theorem 6.11]{Raghunathan(1972)}, 
which implies the following:

\begin{theorem}\label{the:theorem_6_11}
  Let $\Gamma \subset\GL(n,\IC)$ be a finitely generated subgroup. Then $\Gamma$
  contains a subgroup $\pi$ of finite index such that for each matrix $A\in
  \pi$ the eigenvalues of $A$ contain no non-trivial root of unity.
\end{theorem}
 
\begin{proposition}\label{prop:klein_bottle_and_locally_symmetric_space}
  Let $B$ be a non-positively curved closed locally symmetric space, then there
  is a finite sheeted cover $\widetilde{B}\to B$ such that 
$\pi_1(\widetilde{B})$ does not
  contain a subgroup isomorphic to the fundamental group of the Klein bottle.
\end{proposition}

\begin{proof}
Put $\Gamma=\pi_1(B)$. There exists a linear centerless semi-simple Lie
group $G$ containing $\Gamma$ as a discrete cocompact subgroup. Consider the
adjoint representation of $G$ onto its Lie algebra $\mathfrak g$. Via this
representation $\Gamma$ embeds in $\GL_n(\IR)=\GL(\mathfrak g)$ where
$n=\dim_\IR \mathfrak g$. By passing to subgroups of finite index
we may assume that $G$ is connected; i.e., is analytic. See 
Mostow~\cite[\S~2]{Mostow(1973)} for
terminology. We will apply Theorem~\ref{the:theorem_6_11} to this situation,
i.e.,  $\Gamma\subset\GL_n(\IR)\subset\GL_n(\IC)$.

Let $\pi\subset\Gamma$ be the finite index subgroup given by 
Theorem~\ref{the:theorem_6_11}, and let $\widetilde{B}\to B$ be the covering 
space
corresponding to $\pi\subset\Gamma$. It remains to show that $\pi$ does not
contain a pair of matrices $A,B$ satisfying $BAB^{-1}=A^{-1}$ and $A\neq I$. We
do this by showing that the existence of such a pair leads to a
contradiction. (Now glance at~\cite[\S~3 on page~12]{Mostow(1973)} and also note
that $\pi$ is a discrete subgroup of $\GL_n(\IR)$.)

By the first paragraph on page 76 of Mostow's book both $A$ and $B$ are
semi-simple matrices, i.e.,  diagonalizable in $\GL_n(\IC)$. And since $A, B^2$
commute they are simultaneously diagonalizable. By page 10 of Mostow,
\[
A= k p = p k
\] 
where $p$ has positive real eigenvalues and $k$ has all
eigenvalues of length 1. Furthermore $p, k\in G\subset \GL_n(\IC)$, see p.~12 of
Mostow. Also $p,k$ are uniquely defined by $A$. Since $A$ has infinite order,
$p\neq\id$. (Otherwise $A=k$ which lies in a compact subgroup of $\GL_n(\IC)$.)
Furthermore $p^t$, $t\in\IR$, is a 1-parameter subgroup of $G$ passing through
$p$ (see p.~12 of Mostow). Now $B^2$ commutes with $p$ since it commutes with
$A$ (see p.~10 of Mostow). Likewise $B^2$ commutes with every matrix $p^t$ since
$p$ and $B^2$ are simultaneously diagonalizable. Now
\[
p^t=\exp(tv),\quad t\in\IR
\] 
for some $v\in\mathfrak g$, $v\neq 0$. And
\[
t\mapsto B p^t B^{-1}
\] 
is the 1-parameter subgroup $\exp(tu)$ where $u=B(v)$
-- through the adjoint action of $B$ on $\IR^n$, i.e., through
$B\in\Gamma\subset\GL_n(\IR)$. And since
\[
p^t=B^2 p^t B^{-2}= B\exp(tu) B^{-1},
\] 
we have that $B(u)=v$.

Let $V$ be the subspace of $\mathfrak g=\IR^n$ spanned by $u$ and $v$. Then
$\dim V$ is 1 or 2, and $V$ is left invariant by $B$. We now proceed to show
that each of the two possibilities leads to a contradiction, thus completing the
proof of Proposition~\ref{prop:klein_bottle_and_locally_symmetric_space}.

\emph{Case 1. }$\dim V=2$. Then $\{u,v\}$ form a basis for $V$ and with respect
to this basis $B\vert_V$ is represented by the matrix $\begin{pmatrix} 0 & 1 \\
  1 & 0
\end{pmatrix}$ which has characteristic polynomial $\lambda^2-1$. Hence
$\lambda=-1$ is an eigenvalue of $B$, contradicting 
Theorem~\ref{the:theorem_6_11}.

\emph{Case 2. }$\dim V= 1$. Then $u=\lambda v$ for $\lambda\in\IR$. Hence
\[
v=Bu=\lambda Bv=\lambda u=\lambda^2 v.
\] 
Since $v\neq 0$,
$\lambda^2=1$. Therefore $\lambda=\pm 1$. If $\lambda=1$, then
\[p^t=\exp(tv)=\exp(tu)=B p^t B^{-1}\quad\forall t\in \IR.\] Setting $t=1$
yields $p=BpB^{-1}$. Therefore
\[p(BkB^{-1})=(BpB^{-1})(BkB^{-1})=BAB^{-1}=A^{-1}=p^{-1} k^{-1}\] and by the
uniqueness of the polar decomposition (Mostow p.~10) we have that
\[
p=p^{-1}\quad\mathrm{and}\quad BkB^{-1}=k^{-1}.
\] 
In particular $p^2=\id$, so
$p=\id$, which contradicts the assumptions on the decomposition $A=kp$.

Therefore $\lambda=-1$, which is an eigenvalue of $B$, contradicting 
Theorem~\ref{the:theorem_6_11}. This finishes the proof of 
Proposition~\ref{prop:klein_bottle_and_locally_symmetric_space}
\end{proof}

%%%%%%%%%%%%%%%%%%%%%%%%%%%%%%%%%%%%%%%%%%%%%%%%%%%%%%%%%%%%%%%%%%%%%%%%%%%%%%%
%%%%%%%%%%%%%%%%%%%%%%%%%%%%%%%%%%%%%%%%%%%%%%%%%%%%%%%%%%%%%%%%%%%%%%%%%%%%%%%
%%%%%%%%%%%%%%%%%%%%%%%%%%%%%%%%%%%%%%%%%%%%%%%%%%%%%%%%%%%%%%%%%%%%%%%%%%%%%%%

\section{The case of a non-finite homotopy fiber}\label{sec:non_finite_fiber}

%%%%%%%%%%%%%%%%%%%%%%%%%%%%%%%%%%%%%%%%%%%%%%%%%%%%%%%%%%%%%%%%%%%%%%%%%%%%%%%

\subsection{The homotopy fiber of an approximate fibration}
\label{subsec:The_homotopy_fiber_of_an_approximate_fibration}

\begin{lemma}\label{lem:finite_domination_for_approximate_fibration}
  Let $q\colon E\to B$ be an approximate fibration, where $E$ is a compact {ENR} and $B$
  is a connected topological manifold. Then the homotopy fiber of $q$ is dominated by a
  finite complex.
\end{lemma}

\begin{remark}
  This result is probably well-known to the experts but we couldn't find a reference so we
  include a proof for the reader's convenience. We are grateful to the referee for
  providing us with the following, alternative argument: By Ferry~\cite{Ferry(1977a)} any
  CW model for the homotopy fiber is shape equivalent to the actual fiber $q\inv(b)$,
  which is compact, so it follows from~\cite[Corollary 3.2]{Edwards-Geoghegan(1975)} that
  the homotopy fiber is finitely dominated.
\end{remark}

\begin{proof}[Proof of Lemma~\ref{lem:finite_domination_for_approximate_fibration}]
  Let $U_1\subset U_2\subset B$ two open balls such that $\overline U_1\subset
  U_2$. We form the pull-backs
  \[
   \xymatrix{
    E\vert_{U_1} \ar[r] \ar[d]^q & E\vert_{U_2} \ar[r] \ar[d]^q & E \ar[d]^q\\
    U_1 \ar[r] & U_2 \ar[r] & B }
   \]
   It follows from the fact that $q$ is an approximate fibration that the 
inclusion
   $E\vert_{U_1}\to E\vert_{U_2}$ is a homotopy equivalence and that both spaces 
have the
   homotopy type of the homotopy fiber of $q$.
   (See~\cite[\S10]{Hughes-Taylor-Williams(1990)}~\cite[Corollary~2.14 and
     Theorem~12.15]{Hughes-Taylor-Williams(1990)}.)

  We first complete the argument under the extra assumption that $E$ is a
  simplicial complex. In this case there is a finite subcomplex $K$ lying in
  between $E\vert_{U_1}$ and $E\vert_{U_2}$ (subdividing $E$ if necessary). So 
$K$
  dominates $E\vert_{U_1}$ which is the homotopy type of the homotopy fiber of
  $q$.

  In the general case, $E$ receives a cell-like map from a compact topological manifold~\cite[11.2]{Lacher(1977)}, 
  which in turn admits a disk bundle which is triangulated (see~\cite[\S III.4]{Kirby-Siebenmann(1977)}). 
  This implies that $E$ receives a cell-like map
  from a finite simplicial complex $X$.
 
  Since $X$ is homotopy equivalent to $E$ and the map from $X$ to
  $B$ is still an approximate fibration, we may just replace $E$ by $X$ and use
  the argument of the special case above.
\end{proof}

\begin{remark} \label{rem:homotopy_fiber_for_aspherical_base} Let $p \colon X 
\to B$ be a
  $\pi_1$-surjective map of $CW$-complexes such that $B$ is aspherical and $X$ 
finite
  dimensional.  Then the homotopy fiber of $p$ is the total space of the 
covering $q
  \colon \overline{X} \to X$ associated to the kernel of $\pi_1(p)$. In 
particular it is
  homotopy equivalent to a finite dimensional $CW$-complex.  This implies that 
the
  homotopy fiber is finitely dominated if and only if it has the homotopy type 
of a
  $CW$-complex of finite type, i.e., a $CW$-complex all whose skeleta are 
finite.
\end{remark}

%(Here we use the ANR condition on
%  $Y$.)

%%%%%%%%%%%%%%%%%%%%%%%%%%%%%%%%%%%%%%%%%%%%%%%%%%%%%%%%%%%%%%%%%%%%%%%%%%%%%%%

\subsection{Factorization into an $\eps$-fibration}
\label{subsec:Factorization_to_an_epsilon_fibration}

We do not know whether 
Theorem~\ref{thm:factorization_to_an_approximate_fibration} holds when the
homotopy fiber is only required to be finitely dominated. (See~\cite[Theorem 
1.1]{Ferry(1977a)} for a related statement when $B=S^1$.)

\begin{theorem}[Factorization to an $\eps$-fibration]
\label{the:factorization_for_nonfinite_fiber}
  Let $p\colon M\to B$ be a map of topological spaces, such that $B$ is a finite
  simplicial complex  and the homotopy fiber of $p$ is finitely dominated. Fix a
metric on $B \times S^1$. 

  Then for each  $\eps>0$ there is a homotopy commutative diagram
  \[
  \xymatrix{
    M\times S^1 \ar[rd]_{p_1 := p\times \id_{S^1}} \ar[rr]^f_\simeq && C 
\ar[ld]^{q}  \\
    & B\times S^1 }
   \]
  where $C$ is a compact ENR, $f$ is a homotopy equivalence, and $q$ is an
  $\eps$-fibration.
\end{theorem}
\begin{proof}  We can assume that
  $p$ is a fibration. Given any $\eps>0$, we will construct a space $C$ which is 
a compact ENR,  a
  continuous map $q\colon C\to B\times S^1$, and an $\eps$-homotopy equivalence
  $\varphi\colon C\to M\times S^1$. (This is good enough by
  Lemma~\ref{lem:condition_for_epsilon_fibration}.) The map $\varphi$ will be a 
composite
  of several homotopy equivalences which we are going to define now.

Let $Z$ be a finite CW-complex of Euler characteristic 0 with a chosen
basepoint. Consider the inclusion and projection maps
\[
M\xrightarrow{i} M\times Z \xrightarrow{r} M.
\] 
Since the composite $r\circ i$ is the identity map, its mapping torus is given 
by
\[
T(r\circ i)=M\times S^1.
\]

Denote by $T(r,i)$ the space $M\times [0,\frac12]\amalg M\times Z\times
[\frac12,1]/{\sim}$, where we glue $M\times\frac12$ to $M\times Z\times \frac12$
using the map $i$ and we glue $M\times Z\times 1$ to $M\times 0$ via $r$. There
are well-known homotopy equivalences
\[
\varphi_1\colon T(r\circ i)\to T(r,i),\quad\varphi_2\colon T(i\circ r)\to 
T(r,i).
\] 
For instance, $\varphi_1$ is the identity on $M\times [0,\frac12]$ and
$i\times \id$ on $M\times [\frac12,1]$. Its homotopy inverse $\psi_1$ collapses
$M\times Z\times [\frac12,1]$ to $M$ and expands $M\times [0,\frac12]$ linearly
to $M\times [0,1]$. As both $i$ and $r$ commute with the projection to $B$, both
homotopies $\psi_1\circ\varphi_2\simeq\id$ and $\varphi_1\circ\psi_2\simeq \id$
are stationary over $B$. As $T(r,i)=T(i,r)$, the same is true for the
homotopy equivalence $\varphi_2$.

Now, as $Z$ has Euler characteristic zero and the fibers of $p$ are assumed to 
be finitely
dominated, the fibers of the composite fibration $p'\colon M\times Z\to 
M\xrightarrow{p} B$
have finiteness obstruction zero and hence are homotopy finite. Choose a 
triangulation
$\calt$ of $B$. By the proof of Theorem~\ref{thm:factorization_to_an_approximate_fibration} 
there exists a commutative diagram
\[
\xymatrix{
 X \ar[rr]^g \ar[rd]_\xi && M\times Z   \ar[ld]^{p'}\\
 & B
}
\]
where $X$ is a compact ENR and $g$ is a homotopy equivalence over each simplex. 
Moreover
it follows from the construction of $X$ that for all simplices of $B$, the 
inclusions
$X\vert _{\partial\sigma}\to X\vert_\sigma$ is a cofibration. As $p'$ is a 
fibration, the inclusion
$M\vert_{\partial\sigma}\to M\vert_\sigma$ is a cofibration for each simplex 
$\sigma$, too. We
may therefore choose a self-map $k\colon X\to X$ such that $i\circ r\circ 
g\simeq g\circ
k$ by a homotopy which restricts to a homotopy over each simplex of $B$. It 
follows that
there is a map $\varphi_3\colon T(k) \to T(i\circ r)$ which is a homotopy 
equivalence over
each simplex of $B$. (Here $T(k)$ is controlled over $B$ via $p'\circ 
\varphi_3$.)

Hence we obtain a chain of maps
\[
E\times S^1\cong T(r\circ i)\xrightarrow{\varphi_1} T(r,i) 
\xleftarrow{\varphi_2} T(i\circ r) \xleftarrow{\varphi_3} T(k)=:C
\]
where each map is a homotopy equivalence over each simplex of $B$. Note that $C$ 
is a compact ENR.

Each of the spaces $T(r\circ i)$, $T(r,i)$, $T(i\circ r)$, and $T(k)$ naturally maps to
$S^1$ and hence also to $B\times S^1$. Now we pull back along self-coverings 
$\id_B\times n\colon B\times S^1\to B\times S^1$ of index $n$. Geometrically this corresponds to
replacing the mapping tori (such as $T(k)$) by $n$-fold multiple mapping tori (such as
$T(k,\dots,k)$). The pull-backs of the corresponding homotopy equivalences $\varphi_1$,
$\varphi_2$, $\varphi_3$ are by construction pieced together from homotopy equivalences
over each simplex of $B$ and over a single cylinder piece within the multiple mapping
torus. Thus, by choosing $n$ and the mesh of the triangulation of $B$ small enough, we can
get each of the homotopy equivalences as controlled over $B\times S^1$ as we wish.

This concludes the proof of~Theorem~\ref{the:factorization_for_nonfinite_fiber}.
\end{proof}

%%%%%%%%%%%%%%%%%%%%%%%%%%%%%%%%%%%%%%%%%%%%%%%%%%%%%%%%%%%%%%%%%%%%%%%%%%%%%%%

\subsection{Proof of the Stabilization Theorem~\ref{the:theorem_is_stably_true} 
for finitely dominated homotopy fiber}
\label{subsec:stably_true_in_general}

The proof of Theorem~\ref{the:theorem_is_stably_true} for the special case of a 
finite homotopy fiber
Section~\ref{sec:stably_true_finite_fiber}   carries over to the finitely 
dominated case. We just need to replace the use of  
Theorem~\ref{thm:factorization_to_an_approximate_fibration} by
  Theorem~\ref{the:factorization_for_nonfinite_fiber} in the proof of 
Lemma~\ref{lem:bounded_over_torus}; it follows that 
Lemma~\ref{lem:bounded_over_torus} and hence 
Corollary~\ref{cor:getting_epsilon_fibration}
are still true under the weaker assumption that the homotopy fiber is finitely 
dominated, provided $s\geq 1$.

%%%%%%%%%%%%%%%%%%%%%%%%%%%%%%%%%%%%%%%%%%%%%%%%%%%%%%%%%%%%%%%%%%%%%%%%%%%%%%%

\subsection{Tight torsion for finitely dominated homotopy fiber}
\label{subsec:Tight_torsion_for_finitely_dominated_homotopy_fiber}

Next we define the tight torsion of $p_1 \colon M \times S^1 \to B \times S^1$, 
provided
that the homotopy fiber of $p$ is finitely dominated.  Fix a metric on $B \times 
S^1$.
Choose $\epsilon > 0$ such that Lemma~\ref{lem:well_definition_of_tight_torsion} 
(for 
$B \times S^1$ instead of $B$) applies to it.  We define the tight torsion 
$N\tau(p_1)$ of
$p_1 = p \times \id_{S^1}$ as the class of the Whitehead torsion of $f$ in 
$\NWh(p_1)$,
where $f$ is the map appearing in the factorization $p_1 = q \circ f$ for a
$\eps$-fibration $q$ in Theorem~\ref{the:factorization_for_nonfinite_fiber}. 
This is
independent of the factorization $p_1 = q \circ f$ by
Lemma~\ref{lem:well_definition_of_tight_torsion}. One easily checks that it is 
also
independent of the choice of the metric on $B \times S^1 $ since $B \times S^1$ 
is
compact.  This definition reduces to the Definition~\ref{def:tide_torsion} of 
tight
torsion $N\tau(p_1)$ in the special case that the homotopy fiber of $p$ is 
finite since an
approximate fibration over $B \times S^1$ is an $\eps$-fibration for any 
$\epsilon > 0$ and any metric
on $B$.

%%%%%%%%%%%%%%%%%%%%%%%%%%%%%%%%%%%%%%%%%%%%%%%%%%%%%%%%%%%%%%%%%%%%%%%%%%%%%%%

\subsection{The approximate fiber problem for finitely dominated homotopy fiber}
\label{subsec:The_approximate_fiber_problem_for_finitely_dominated_homotopy_fiber}

The following theorem extends Theorem~\ref{the:main_result_1_generalized} to the 
case of a finitely dominated fiber.

\begin{theorem}[A criterion in the case of a finitely dominated 
fiber]\label{the:main_result_3}
  Let $B$ be an aspherical closed triangulable manifold. Suppose that $\pi_1(B)$ 
satisfies FJC and the cyclic subgroups
  of $\pi_1(B)$ are orientable.  Let $M$ be a closed connected manifold of 
dimension $\neq 4$. 

  Then a $\pi_1$-surjective map $p\colon M\to B$ is homotopic to a MAF if and only if
  the following conditions are satisfied:
  \begin{enumerate}
  \item \label{the:main_result_3:(1)} The homotopy fiber of $p$ is finitely 
dominated;
  \item \label{the:main_result_3:(2)} If~\ref{the:main_result_3:(1)} is 
satisfied, the element $N\tau(p_1) \in \NWh(p_1)$ is defined
  and we require it to vanish;
  \item \label{the:main_result_3:(3)} If~\ref{the:main_result_3:(1)} 
and~\ref{the:main_result_3:(2)}  are  satisfied,
  the element $N\tau(p) \in \NWh(p)$ is defined and we require it to vanish.
  \end{enumerate}
\end{theorem}

\begin{proof}
  Suppose first that $p$ approximately fibers. Then the homotopy fiber of $p$ is 
finitely
  dominated by Lemma~\ref{lem:finite_domination_for_approximate_fibration}. 
  Note that in this case $N\tau(p)$ is defined and both
  $N\tau(p)$ and $N\tau(p_1)$ clearly vanish.

  Suppose conversely that conditions (i) to (iii) hold. 
By Theorem~\ref{the:theorem_is_stably_true} (see section~\ref{subsec:stably_true_in_general}) and 
Theorem~\ref{the:stably_implies_unstably}~\ref{the:stably_implies_unstably:fac_prop}, the
  map $p$ is $h$-cobordant to an approximate fibration. Hence there is a 
homotopy equivalence $f \colon M \to M'$
  an approximate fibration $q \colon M' \to B$ for a closed manifold $M'$ such 
that $q \circ f \simeq p$. 
  Hence $N\tau(p)$ is  defined, namely by the Whitehead torsion of $f$ 
considered in $\NWh(p)$.
  (This is well-defined by Lemma~\ref{lem:well_definition_of_tight_torsion}).
  Condition (iii) together with  
Theorem~\ref{the:stably_implies_unstably}~\ref{the:stably_implies_unstably_approx} implies  
  that $p$ is homotopic to a {MAF}. This finishes the proof of 
Theorem~\ref{the:main_result_3}.
\end{proof}

\begin{remark}\label{rem:approximate_fibration_with_nonfinite_fiber}
  Let $M$ be a closed manifold and $f\colon M\to B$ be a manifold approximate
  fibration whose homotopy fiber is not homotopy equivalent to a finite CW
  complex. Then $f$ cannot have an (even topological) regular value.

  In fact, suppose that $b\in B$ was a regular value, so that $f\inv(U)\cong
  f\inv(b)\times U\simeq f\inv(b)$ for a small contractible neighborhood $U$ of
  $b$ in $B$. Then $f\inv(U)$ is an ANR, being an open subset of an ANR, and
  $f\inv(b)$ is a retract of $f\inv(U)$, hence also an {ANR}. By 
  West~\cite{West(1977)}, $f\inv(b)$ is homotopy equivalent to a finite CW
  complex. But $f$ is an approximate fibration, so the homotopy fiber of $f$ is
  homotopy equivalent to $f\inv(U)$ and thus to a finite $CW$-complex,
  contradicting the assumption.
\end{remark}

%%%%%%%%%%%%%%%%%%%%%%%%%%%%%%%%%%%%%%%%%%%%%%%%%%%%%%%%%%%%%%%%%%%%%%%%%%%%%%%

\subsection{Approximate fibrations with non-finite homotopy fiber}
\label{subsec:The_approximate_fibrations_with_non-finite_homotopy_fiber}

Next we use our results obtained so far to give an example of an approximate 
fibration
$g\colon M\to T^2$ whose total space $M$ is a closed smooth manifold, such that 
the
homotopy fiber of $g$ is \emph{not} homotopy equivalent to a finite CW
complex. Chapman-Ferry~\cite{Chapman-Ferry(1983)} have produced similar examples 
over
arbitrary base manifolds of Euler characteristic zero.

Let $p$ be an odd prime such that $h_1(p)$ has an odd prime factor. (Here
$h_1(p)$ denotes the first factor of the class number of the ring $\IZ[e^{2\pi
  i/p}]$. For instance, the primes $p=23,31,37,41,43$, and $47$ will do
\cite[Table 8]{Borevich-Shafarevich(1966)}. Any irregular prime $p$ satisfies
this condition on $h_1(p)$.) Then Wall has 
shown~\cite[Corollary~5.4.2]{Wall(1967)} 
that there exists a Poincar\'e 4-complex $Y$ with
$\pi_1Y=\IZ/p$ whose Wall finiteness obstruction $\sigma(Y)\in \widetilde 
K_0(\IZ[\IZ/p])$
is non-zero and in fact not of order two.

We claim that $Y\times S^1\times S^2$ is homotopically equivalent to a closed
smooth 7-manifold $M'$. In fact, the only obstruction to lifting the Spivak
fibration of $Y$ to a stable vector bundle is in
\[
H^3(Y; \IZ/2)=H_1(Y; \IZ/2)=\IZ/p\otimes \IZ/2=0.
\] 
The Wall obstruction is
$\sigma(Y\times S^1)=\chi(S^1)\sigma(Y)=0$. Therefore there exists a closed
smooth 5-manifold $N$ and a normal map $N\to Y\times S^1$. Crossing $f$ with
$\id\colon (D^3, S^2)\to (D^3, S^2)$ and using the $\pi$-$\pi$-theorem yields
that
\[
f\times \id\colon N\times (D^3, S^2)\to Y\times S^1\times (D^3, S^2)
\] 
can be surgered to a homotopy equivalence. Restricting this to the boundary 
yields a
closed smooth 7-manifold $M'$ homotopy equivalent to $Y\times S^1\times S^2$.

Notice that $M'$ comes with a canonical map $g'\colon M'\to S^1$, corresponding
to $\pi_1(M') = \IZ/p\times \IZ\to \IZ$. Its homotopy fiber $F_{g'}$ is 
dominated by a finite
complex but \emph{not} homotopy equivalent to a finite complex, as can be seen
by the Wall obstruction: 
\[
\sigma(F_{g'})=\sigma(Y\times S^2)=\chi(S^2)\cdot \sigma(Y) =  2 \cdot  
\sigma(Y) \neq
0.
\]

Following the proof of Theorem~\ref{the:main_result_3}, we see that the map 
\[g'\times \id_{S^1}\colon M'\times S^1\to T^2\]
is $h$-cobordant to an approximate fibration $g\colon M\to T^2$. Of course the homotopy fiber $F_g$ of $g$ is homotopy equivalent to $F_{g'}$.

%%%%%%%%%%%%%%%%%%%%%%%%%%%%%%%%%%%%%%%%%%%%%%%%%%%%%%%%%%%%%%%%%%%%%%%%%%%%%%%
%%%%%%%%%%%%%%%%%%%%%%%%%%%%%%%%%%%%%%%%%%%%%%%%%%%%%%%%%%%%%%%%%%%%%%%%%%%%%%%
%%%%%%%%%%%%%%%%%%%%%%%%%%%%%%%%%%%%%%%%%%%%%%%%%%%%%%%%%%%%%%%%%%%%%%%%%%%%%%%

\section{Three-dimensional manifolds}\label{sec:Three-dimensional_manifolds}

\begin{theorem}[$3$-manifolds as source]\label{the_dim_three}
Let $M$ be a closed 3-manifold, $B$ be an aspherical closed manifold, and 
$p\colon M\to B$.
% be a $\pi_1$-surjective map.  
Then the following assertions are equivalent:
\begin{enumerate}

\item The kernel of the $p_* \colon \pi_1(M) \to \pi_1(B)$ is 
finitely generated and its image has finite index;

\item The homotopy fiber of $p$ is finitely dominated;

\item The homotopy fiber of $p$ is homotopy finite;

\item $p$ is homotopic to MAF;

\item $p$ is homotopic to the projection of a locally trivial fiber bundle of 
closed manifolds.

\end{enumerate}
\end{theorem}

\begin{proof}
  Obviously it suffices to show that the first condition implies the last one. Following Remark~\ref{rem:pi_1_surjective}, 
it is enough to consider the case where $p$ is $\pi_1$-surjective.

  We begin with the case, where $M$ is homotopy equivalent to $\IR\IP^2 \times 
S^1$.  Then
  $\pi_1(B)$ is a quotient of $\IZ/2 \times \IZ$ and hence $\pi_1(B) \cong \IZ$ 
and $B =
  S^1$. Since Thurston's Geometrization Conjecture is true by the work of 
Perelman
  (see~\cite{Morgan-Tian(2014)}), $M$ is actually homeomorphic to $\IR\IP^2 
\times S^1$
  by~\cite[page~457]{Scott(1983)}. Since any map $\IR\IP^2 \to S^1$ factors up 
to homotopy
  through the projection $\pr\colon \IR P^2\times S^1 \to S^1$ and $p$ is
  $\pi_1$-surjective, we conclude that $p$ is homotopic to either $\pr$ or 
$c\circ \pr$
  where $c\colon S^1\to S^1$ is the reflection. Both maps are fiber bundle 
projections.
  Hence we can assume in the sequel that $M$ is not homotopy equivalent to 
$\IR\IP^2
  \times S^1$.

Next we treat the case, where $H=\ker p_*$ is not isomorphic to $\IZ$.
By~\cite[Theorem~11.1~(2) on page~100]{Hempel(1976)} the group $\pi_1(B)$ is
virtually cyclic.  Since $B$ is by assumption a closed aspherical manifold,
$\pi_1(B) \cong \IZ$ and $B = S^1$.  We conclude from~\cite[Theorem~11.6~(i) on
page~104]{Hempel(1976)} and Perelman's proof of the Poincar\'e Conjecture
(see~\cite{Kleiner-Lott(2008)}~or\cite{Morgan-Tian(2008_Ricci_poincare)})  that there
exists a fiber bundle $F \xrightarrow{i} M \xrightarrow{q} S^1$ with a connected
$2$-manifold as fiber such that the image of $\pi_1(i)$ is $H$. Since $q$ or $c
\circ q$ induce the same map on the fundamental group  as $p$, the map $p$ is 
homotopy
equivalent to the one of the bundle projections $q$ or $c \circ q$.

It remains to treat the case, where $H \cong \IZ$.  If $\pi_1(B) \cong S^1$, 
then
$\pi_1(M)$ is $\IZ^2$ or $\IZ \rtimes \IZ$ which is not possible for a closed 
$3$-manifold
$M$ since otherwise by~\cite[Lemma~3.13 on page~28 and Theorem~7.1 on
page~66]{Hempel(1976)} the manifold $M$ is irreducible and hence 
by~\cite[Theorem
4.3]{Hempel(1976)} aspherical and therefore the cohomological dimension of 
$\pi_1(M)$ is
precisely $3$.  Hence $B$ is an aspherical  closed $2$-manifold. We conclude 
from~\cite[Theorem~11.10
on page~112]{Hempel(1976)} that $p$ is homotopy equivalent to a bundle 
projection since
any automorphism of $\pi_1(B)$ is induced by a homeomorphism.
\end{proof}

\begin{remark}
  The equivalence between (iv) and (v) in this theorem holds even if $B$ is not
  aspherical: For $B$ of dimension 1 or 2, this follows from~\cite[Theorem A]{Husch(1977)}
  together with Perelman's solution of the Poincar\'e conjecture. For $B$ of dimension 3
  this is due to~\cite{Armentrout(1971)} plus the Poincar\'e conjecture. We are grateful
  to the referee for pointing this out to us.
\end{remark}

%%%%%%%%%%%%%%%%%%%%%%%%%%%%%%%%%%%%%%%%%%%%%%%%%%%%%%%%%%%%%%%%%%%%%%%%%%%%%%%
%%%%%%%%%%%%%%%%%%%%%%%%%%%%%%%%%%%%%%%%%%%%%%%%%%%%%%%%%%%%%%%%%%%%%%%%%%%%%%%
%%%%%%%%%%%%%%%%%%%%%%%%%%%%%%%%%%%%%%%%%%%%%%%%%%%%%%%%%%%%%%%%%%%%%%%%%%%%%%%

\section{A counterexample to the Naive Conjecture~\ref{con:Naive_Conjecture}}
\label{sec:counterexample_to_Naive_Conjecture}

Let $B$ denote the Klein bottle and $K$ its fundamental group.

\begin{theorem}[Counterexample to the Naive 
Conjecture~\ref{con:Naive_Conjecture}] 
\label{the:important_example}
  For $n\geq 6$ there exists a continuous map $p\colon M\to B\times S^1$,
 where $M$ is a closed smooth $(n+1)$-manifold whose homotopy fiber $F_p$ 
is a finite complex and whose tight
  torsion $N\tau(p)=0$, but where $p$ is not homotopic to an approximate
  fibration.
\end{theorem}

\begin{addendum}
  In fact $M$ is of the form $\calm\times S^1$ for a smooth $n$-manifold 
$\calm$,
  and $p$ factors as the composite $(q\circ \phi)\times\id_{S^1}$ where $q\colon
  B\times \caln\to B$ is the projection onto the first
  factor, $\caln$ is a closed smooth manifold, and $\phi\colon \calm\to B\times
  \caln$ is a homotopy equivalence.
\end{addendum}

Here is the strategy of proof of Theorem~\ref{the:important_example}. We will
construct the homotopy equivalence $\phi\colon\calm\to B\times \caln$ such that
$\tau(\phi)$ is not in the image of the assembly map
\begin{equation}
\asmb = H_1^\pi(\pr, \bfWh(q))\colon H_1^\pi(\widetilde{B};\bfWh(q))  \to 
\Wh(\pi_1(B\times\caln))
\label{special_asmb}
\end{equation}
and hence the tight torsion
\[
y=N\tau(q\circ\phi)\neq 0\in \NWh(\pi_1(B\times\caln)).
\]

As $\overline{y}=(-1)^{n+1} y$, where $n=\dim\calm$, $y$ defines a Tate cohomology 
class
\[ 
[y]\in\widehat H^{n+1}(\IZ/2;\NWh(q))
\]
which, in our construction, will be non-zero. We then set $M=\calm\times S^1$,
$N=\caln\times S^1$, and $p=(q\circ\phi)\times\id_{S^1}$. Note that $N\tau(p)=0$ 
since it can be calculated using $f=\varphi\times \id_{S^1}$ in \eqref{diag:factorization_property} and since
\[\tau(\varphi\times\id_{S^1})=\tau(\varphi)\cdot \chi(S^1)=0\]
by~\cite{Kwun-Szczarba(1965)}.

Provided we can implement our construction, our proof is by contradiction, i.e., 
we assume
that $p\colon \calm \times S^1\to B\times S^1$ is homotopic to an approximate 
fibration
$\widehat{p}$. Then by Theorem~\ref{the:stably_implies_unstably}~(i), we obtain 
an
$h$-cobordism $W$ between $q \circ \phi \colon \calm \to B$ and an approximate
fibration. Here $y=x-(-1)^n\overline{x}$ where $x$ is the image of 
$\tau(W,\calm)$ in
$\NWh(q)$ and therefore
\[
[y]=0\in \widehat H^{n+1}(\IZ/2; \NWh(q)).
\] 
This
contradiction proves Theorem~\ref{the:important_example}, provided we can
implement our construction.

The following lemma is a consequence of work of Bass and Murthy
\cite{Bass-Murthy(1967)}. 

\begin{lemma}\label{lem:on_C_4}
  The group $S=\IZ/4$ has the properties that 
  \[
\NK_0(\IZ[S])\cong \bigoplus_{i=1}^\infty \IZ/2
\]
  and
  \[
{\widehat{~}}\colon \NK_0(\IZ[S])\to \NK_0(\IZ[S])
\]
  is the identity where $\widehat{~}$ is the involution induced by the 
involution of
  the polynomial ring $\IZ[S][x]$ determined by $x\mapsto x$ and $s\mapsto   
s\inv$ for $s\in S$.
\end{lemma}
\begin{proof}
  There is a Cartesian square
  \[
   \xymatrix{
    {\IZ}[\IZ/4] \ar[rr] \ar[d] && {\IZ}[i] \ar[d]\\
    {\IZ}[\IZ/2] \ar[rr] && R:={\IZ_2}[x]/(x^2) }
   \]
  Applying the functor $\NK_*$ to it gives a Mayer-Vietoris sequence
  \[
  \NK_1(\IZ[\IZ/2])\oplus \NK_1(\IZ[i]) \to \NK_1(R) \to \NK_0(\IZ[\IZ/4])\to 
\NK_0(\IZ[\IZ/2])\oplus
  \NK_0(\IZ[i])
  \] 
  Now it is known that $\NK_i(\IZ[\IZ/2])=0$ for $i=0,1$ and
  $\NK_i(\IZ[i])=0$ for $i=0,1$ since the Gaussian integers are a regular ring.

  Hence the map
  \[
  \NK_1(R)\to \NK_0(\IZ[\IZ/4])
   \] 
   is an isomorphism, which respects the involution
  since the maps in the Cartesian square do.  On the other hand the only
  involution on $R$ is $\id$. And since $R$ is quasi-regular, i.e., $R/(x)$ is
  regular and $x$ is nilpotent, $\NK_1(R)$ is easy to compute; namely it is
  $(R[t])^\times/\{1,1+x\}$.  Thinking of these as $(1,1)$-matrices, we see that
  the involution on $\NK_1(R)$ is $\id$.
\end{proof}

Choose an oriented closed $(n-2)$-manifold $\caln$ so that 
$\pi_1(\caln)=S=\IZ/4$. 
To construct $\phi\colon\calm\to
B\times\caln$ we examine the Rothenberg exact sequence and in particular the map
\[
L^h_{n+1}(\IZ[S\times K])\to \widehat H^{n+1}(\IZ/2;\Wh(S\times K)).
\]

\begin{remark}
  All the $L$-groups occurring in the following discussion are with respect to
  $w\colon S\times K\to \IZ/2$ given by the composition of projection $S\times
  K\to K$ with the first Stiefel-Whitney class of the Klein bottle.
\end{remark}

The quotient map
\[
\Wh(S\times K)\to\NWh(q)
\]
preserves the involutions and hence induces a homomorphism
\[
\widehat H^{n+1}(\IZ/2;\Wh(S\times K))\to\widehat H^{n+1}(\IZ/2;\NWh(q)).
\] 
Denote its composition with the Rothenberg homomorphism by
\[
\eta\colon L^h_{n+1}(\IZ[S\times K])\to \widehat H^{n+1}(\IZ/2;\NWh(q)).
\]
If $\eta$ is not identically zero, just select $x$ such that $\eta(x)\neq 0$ and 
it
will determine $\phi\colon \calm\to B\times\caln$ such that 
$[N\tau(\overline{q}\circ \phi)]\neq 0$ 
showing that our construction can be implemented.

We proceed to show that such an element $x$ exists. Let $\calc$ be the class of
all finitely generated abelian groups. The claim clearly follows from the two
assertions:
\begin{description}
\item[Assertion (i)] $\eta$ is a mod-$\calc$ monomorphism, i.e., $\ker \eta$ is
  in $\calc$;
\item[Assertion (ii)] $L^h_{n+1}(\IZ[S\times K])$ is not in $\calc$.
\end{description}

To prove the assertions, we use a variant of the Rothenberg exact sequence where
we replace $L^s$ by $L^X$ and $X\subset\Wh(S\times K)$ is the image of the
homomorphism
\[
\Wh(S\times \IZ)\to\Wh(S\times K)
\]
induced by the inclusion $S\times \IZ\to(S\times \IZ)\rtimes \IZ=S\times K$ (see
\cite[page~4]{Wall(1976)}):
\begin{multline}\label{eq:variant_of_rothenberg}
  \cdots \to L^X_{n+1}(\IZ[S\times K])\to L^h_{n+1}(\IZ[S\times 
K])\xrightarrow{\psi}
\widehat H^{n+1}(\IZ/2;\Wh(S\times K)/X)
\\
\to L^X_n(\IZ[S\times K]) \to \cdots
\end{multline}

We start by showing

\begin{description}
\item[Assertion (iii)] $\psi$ is a mod-$\calc$ isomorphism, i.e., both the 
kernel
  and cokernel of $\psi$ are in $\calc$.
\end{description}

For this is suffices, because of~\eqref{eq:variant_of_rothenberg}, to show that
$L^X_i(\IZ[S\times K])$ is finitely generated for both $i=n$ and $i=n+1$. And 
these
groups can be analyzed by using a variant of the Wall-Shaneson exact sequence;
in particular by using Ranicki~\cite[Theorem~5.2]{Ranicki(1973c)} where
$A= \IZ[S\times \IZ]$, $\alpha \colon A \xrightarrow{\cong} A$ is the ring 
automorphism 
induced by the group automorphism
by $\id_S \times -\id_{\IZ} \colon S\times \IZ \xrightarrow{\cong} S\times \IZ$,
and $R=K_1(\IZ[S\times \IZ])$. Then $\widetilde{V}_n^{\overline{\varepsilon} 
R}(A_\alpha)=L^X_n(\IZ[S\times K])$ 
and it suffices to show that
both $V_n^R(A)=L^h_n(\IZ[S\times \IZ])$ and $V_{n-1}^{(1-\alpha)\inv
  R}(A)=L^h_{n-1}(\IZ[S\times \IZ])$ are finitely generated. 
Now Wall~\cite{Wall(1976)}
showed that both $L^h_i(\IZ[F])$ and $L^s_i(\IZ[F])$ are finitely generated for 
all $i$
and every finite group $F$. Hence the Wall-Shaneson Theorem shows that
$L^s_i(\IZ[S\times \IZ])$ is in $\calc$ for all $i$. And using the Rothenberg 
exact
sequence together with Bass's result that $\Wh(F)$ is in $\calc$ for every
finite group $F$, we see that $L^h_i(\IZ[S\times \IZ])$ is in $\calc$, 
completing the
verification of Assertion (iii).

We use this result to show Assertion (i). Consider the following commutative
diagram of $\IZ/2$-modules:
\[
\xymatrix{
  {\Wh}(S\times K) \ar[rr] \ar[d] && {\Wh(S\times K)}/X \ar[d]^\gamma \\
  {\Wh(S\times K)}/\im(\asmb) \ar[rr]^\beta && {\Wh(S\times K)}/(X+\im(\asmb)) }
\]
where $\asmb$ is the assembly map of~\eqref{special_asmb}.
Because of it, it clearly suffices to show that $\gamma$ induces a mod-$\calc$
isomorphism
\[
\gamma_*\colon \widehat H^{n+1}(\IZ/2;\Wh(S\times K)/X)\to\widehat
H^{n+1}(\IZ/2;\Wh(S\times K)/(X+\im(\asmb))).
\] 
Using the exact sequence in Tate
cohomology induced from a short exact sequence of $\IZ/2$-modules, we see that
this is true provided we can verify the following Assertion:

\begin{description}
\item[Assertion (iv)] For all $n\in \IZ$, $\widehat H^{n+1}(\IZ/2;\im(\asmb)/ 
(X\cap\im(\asmb)))$ is in
  $\calc$.
\end{description}

In fact, the domain of $\asmb$ is finitely generated because of the
Atiyah-Hirzebruch spectral sequence and the fact that $K_i(\IZ[F])$ is finitely
generated when $F$ is a finite group and $i\leq 1$. Therefore
$\im(\asmb)/(X\cap\im(\asmb))$ and any of its sub-quotients are also finitely
generated, which establishes Assertion (iv).

So Assertion (i) is proven; let us now show Assertion (ii), which, by Assertion
(iii), is equivalent to showing that $\widehat H^{n+1}(\IZ/2;\Wh(S\times K)/X)$ 
is
not in $\calc$.

By the main result of~\cite{Farrell-Hsiang(1970)}, $\Wh(S\times K)/X$ is the
direct sum of two $\IZ/2$-modules $W$ and $Z$ where $\widehat 
H^{n+1}(\IZ/2;Z)=0$ and
$W$ is $\IZ/2$-isomorphic to $\NK_0(\IZ[S])$ with involution ${\widehat~}=\id$ 
because of
Lemma~\ref{lem:on_C_4}. Therefore
\[
\widehat H^{n+1}(\IZ/2;\Wh(S\times K)/X)=\NK_0(\IZ[S])/ 2\cdot \NK_0(\IZ[S]).
\]
By Lemma~\ref{lem:on_C_4}, $\NK_0(S)/ 2\NK_0(S)$ is not in $\calc$.

This concludes the proof of Theorem~\ref{the:important_example}.

%%%%%%%%%%%%%%%%%%%%%%%%%%%%%%%%%%%%%%%%%%%%%%%%%%%%%%%%%%%%%%%%%%%%%%%%%%%%%%%
%%%%%%%%%%%%%%%%%%%%%%%%%%%%%%%%%%%%%%%%%%%%%%%%%%%%%%%%%%%%%%%%%%%%%%%%%%%%%%%
%%%%%%%%%%%%%%%%%%%%%%%%%%%%%%%%%%%%%%%%%%%%%%%%%%%%%%%%%%%%%%%%%%%%%%%%%%%%%%%

\section{An example on block fibering}
\label{sec:An_example_on_block_fibering}

The following example shows that the vanishing of our obstructions does, in
general, \emph{not} imply that a map is homotopic to a block bundle.

\begin{theorem}[MAF versus block bundle]\label{the:example_on_block_fibering}
  Let $m\geq 6$ be even. There exist a pair of smooth closed (connected) 
manifolds $M$ and $N$ of dimension $m+1$, 
  a  simple homotopy equivalence $f\colon M\to N$, and a smooth fibre bundle map
  $p\colon N \to S^1\times S^1$ such that the composite map $p\circ f\colon M\to
  S^1\times S^1$ is not homotopic to a block bundle projection but is homotopic to 
an approximate fibration.
\end{theorem}

\begin{remark}\label{rem:explicite_projection}
  In fact $N=L\times(S^1\times S^1)$ where $L$ is closed connected  smooth 
manifold, 
  and $p$ is  the projection onto the second factor.
\end{remark}

We will need the following result to prove 
Theorem~\ref{the:example_on_block_fibering}.

\begin{lemma}\label{lem:main_lemma}
  There exist a pair of closed  connected smooth manifolds $L$ and $\mathcal M$
  with $m=\dim(\mathcal M)$ being any given even integer $\geq 6$, and a
  homotopy equivalence $\varphi\colon\mathcal M\to L\times S^1$ such that
  $q(\tau(\varphi))$ cannot be expressed as $x+\overline{x}$ for some element 
$x\in
  \widetilde{K}_0(\IZ[\pi_1(L)])$. Here $q\colon\Wh(\pi_1(L\times 
S^1))\to\widetilde{K}_0(\IZ[\pi_1(L)])$ 
  denotes the projection map in the Bass-Heller-Swan formula.
\end{lemma}
\begin{proof}
  We begin by recalling some needed calculations. These calculations can be
  found in~\cite{Bak(1974)} and~\cite[page~30]{Milnor(1971)}. Let $p$ be an odd 
prime.

  \begin{enumerate}
  \item $L^s_i(\IZ[\IZ/p])=L^h_i(\IZ[\IZ/p])=0$ if $i$ is odd;
  \item $\widehat H^i(\IZ/2;\Wh(\IZ/p))=0$ if $i$ is odd;
  \item $\widetilde{K}_0(\IZ[\IZ/29])=\IZ/2\oplus \IZ/2\oplus \IZ/2$.
  \end{enumerate}
  The following result is a consequence of (iii) and the fact that 
  $\IZ/2\oplus   \IZ/2\oplus \IZ/2$ equipped with some $\IZ/2[\IZ/2]$-module 
structure has 
  either the trivial   $\IZ/2$-action or is $\IZ/2[\IZ/2]$-isomorphic the direct 

sum of $\IZ/2[\IZ/2]$ 
  and the  trivial $\IZ/2[\IZ/2]$-module $\IZ/2$:
  \begin{enumerate}\setcounter{enumi}{3}
  \item $\widehat H^i(\IZ/2;\widetilde{K}_0(\IZ[\IZ/29]))\neq 0$ for all $i$.
  \end{enumerate}

  Next consider the Rothenberg exact sequence for $\IZ[\IZ\times \IZ/29]$:
  \begin{multline}\label{rothenberg_sequence}
    \cdots \to L^h_{m+1}(\IZ[\IZ\times \IZ/29]) \xrightarrow{\beta} \widehat 
H^{m+1}(\IZ/2;\Wh(\IZ\times \IZ/29))
   \\
   \to L^s_m(\IZ[\IZ\times \IZ/29])\xrightarrow{\alpha_m} L^h_m(\IZ[\IZ\times 
\IZ/29]) \cdots \to.
    \end{multline}

  Let us examine the map $\alpha_m$ of~\eqref{rothenberg_sequence} in terms of
  the Wall-Shaneson formula. Since
  \[
  L^s_m(\IZ[\IZ\times \IZ/29])= L^s_m(\IZ[\IZ/29])\oplus 
L^h_{m-1}(\IZ[\IZ/29])=L^s_m(\IZ[\IZ/29])
  \]
  because of calculation (i) and our assumption that $m$ is even, $\alpha_m$
  factors as the composition of the Rothenberg map
  \[
   \alpha\colon L^s_m(\IZ[\IZ/29])\to L^h_m(\IZ[\IZ/29])
   \]
  and the functorial inclusion
  \[
  L^h_m(\IZ[\IZ/29])\to L^h_m(\IZ[\IZ\times \IZ/29|).
  \] 
  But $\alpha$ is monic by the
  Rothenberg exact sequence for $\IZ[\IZ/29]$ and calculation (ii); namely, that
  $\widehat H^{m+1}(\IZ/2;\Wh(\IZ/29))=0$.

  Consequently, $\alpha_m$ is monic and hence the map $\beta$ 
  in~\eqref{rothenberg_sequence} is a epimorphism.

  Let $L$ be any closed smooth (connected) $(m-1)$-dimensional manifold with
  $\pi_1(L)=\IZ/29$ and fix an element $y\in\widetilde{K}_0(\IZ[\IZ/29])$ 
satisfying
  $y=\overline{y}$ but not of the form $x+\overline{x}$ where 
  $x\in\widetilde{K}_0(\IZ[\IZ/29])$. This is possible because of statement 
(iv). Let
  \[
  \sigma\colon\widetilde{K}_0(\IZ[\IZ/29])\to\Wh(\IZ\times \IZ/29)
  \]
  be the usual splitting of $q$ and let $z=\sigma(y)$. Then
  \[
  \overline{z}=-z
   \]
  since $\overline{\sigma(y)} = -\sigma(\overline{y})$. Since $\beta$ is onto, 
there
  exists a closed smooth manifold $\mathcal M$ together with a homotopy
  equivalence
  \[
   \varphi\colon\mathcal M\to L\times S^1
   \]
  such that $\tau(\varphi)=z$. But
  \[
   q(\tau(\varphi))=q(\sigma(y))=y
  \] 
so we have completed the construction. 
This finishes the proof of Lemma~\ref{lem:main_lemma}.
\end{proof}

\begin{proof}[Proof of Theorem~\ref{the:example_on_block_fibering}]
  Let $N=L\times S^1\times S^1$, $M=\mathcal M\times S^1$, $f=\varphi\times\id$,
  and $p$ be the projection described in Remark~\ref{rem:explicite_projection}. 
Since
  \[
   \tau(f)=\chi(S^1)\,\tau(\varphi)=0
  \]
  by the Kwun-Szczarba formula~\cite{Kwun-Szczarba(1965)}. Hence $p \circ f \colon \calm \times S^1 \to S^1 
\times S^1$ is homotopic to an 
  approximate fibration by Theorem~\ref{the:main_result_1_generalized},
  since the homotopy fiber of $f \circ p$ is $L$ and $\pi_1(S^1 \times S^1)$ 
satisfies
  FJC by Theorem~\ref{the:status_of_FJC}.

   Hence it remains to show that $p\circ f$ is
  \emph{not} homotopic to a block bundle projection. We do this by assuming the
  opposite and showing this assumption leads to a contradiction. Our assumption
  clearly forces $M$ to contain a closed connected (locally flat) codimension
  one submanifold $\mathcal M_0$ where $\mathcal M_0$ also contains a closed
  connected (locally flat) codimension one submanifold $\mathcal L_0$ with the
  following properties:
  \begin{enumerate}
  \item $\mathcal M_0$ lifts to the covering space $\mathcal M\times \IR\to
    \mathcal M\times S^1$ and the inclusion $\mathcal M_0\subset\mathcal
    M\times\IR$ is a homotopy equivalence;
  \item The composite map
    \[
    \mathcal L_0\subset\mathcal M_0\subset\mathcal M\times\IR\to \mathcal
    M\xrightarrow{\varphi}L\times S^1
    \] 
   lifts to the covering space $L\times\IR\to L\times S^1$ and this lift induces 
a homotopy equivalence.
  \end{enumerate}

\begin{remark}
  We can assume that $\mathcal M_0$ and $\mathcal L_0$ are smooth submanifolds
  satisfying properties (i) and (ii) by applying Hirsch's codimension one
  smoothing theorem in~\cite{Hirsch(1961)}.
\end{remark}

The following diagram~\eqref{diag1} is useful in ``visualizing'' properties (i)
and (ii): the dotted arrows in this diagram are the lifts posited in (i) and
(ii). Note that all the horizontal arrows are homotopy equivalences.

\begin{equation}\label{diag1}
  \xymatrix{{\mathcal M}\times S^1\\
    {\mathcal M_0} \ar@{^{(}->}[u] \ar@{.>}[r] & {\mathcal M}\times \IR \ar[lu] 
\ar[r] 
   & {\mathcal{M}} \ar[r]^(.4)\varphi & L\times S^1\\
    {\mathcal L_0} \ar@{^{(}->}[u] \ar@{.>}[rrr] &&&L\times\IR \ar[u] \ar[r] & L
  }
\end{equation}

Denote the composite of the horizontal arrows in line 2 of~\eqref{diag1} by
\[
\psi\colon\mathcal M_0\to L\times S^1
\]
and identify $L$ with the codimension one submanifold $L\times 1$ of $L\times
S^1$. It is easily seen using property (ii) that $\psi$ is homotopic to a smooth
homotopy equivalence $\eta$ which is split along $L$; i.e.
\begin{enumerate}
\item $\eta$ is transverse to $L$, and
\item $\eta\vert_{\eta^{-1}(L)}\colon\eta^{-1}(L)\to L$ is a homotopy
  equivalence.
\end{enumerate}

It was shown in~\cite{Farrell-Hsiang(1973)} (see also
\cite{Farrell-Hsiang(1968)}) that $\eta$ split forces its Whitehead torsion
$\tau(\eta)$ to lie in the image of $\Wh(\pi_1 (L))$ in $\Wh(\pi_1(L\times 
S^1))$;
consequently $q(\tau(\eta))=0$. And therefore
\begin{equation}\label{eq1}
  q(\tau(\psi))=0
\end{equation}
since $\tau(\psi)=\tau(\eta)$.

Now observe that, for $r$ a sufficiently large real number, $\mathcal M\times r$
is disjoint from $\mathcal M_0$ and the region $W$ between them is a compact
smooth $h$-cobordism with $\partial_- W=\mathcal M_0$ and $\partial_+ W=\mathcal
M\times r$. Let $y$ denote $\tau(W,\partial_- W)$ and let $g\colon \partial_-
W\to\partial_+ W$ be the composite of the inclusion $\partial_- W\subset W$ with
a retraction $W\to\partial_+ W$, then
\[
\tau(g)=y -\overline{y}
\]
because $m$ is assumed to be even. After identifying $\mathcal M\times r$ with 
$\mathcal M$ in the natural
(diffeomorphic) way, the composite $\varphi\circ g$ is homotopic to $\psi$ since
$g$ is easily seen to be homotopic to the composition
\[
\mathcal M_0\subset \mathcal M\times\IR\to \mathcal M
\]
of the first two arrows on line 2 of~\ref{diag1}. Therefore
\[
\tau(\psi)=\tau(\varphi)+\varphi_*(y-\overline{y});
\]
applying $q$ to this equation yields
\[
q(\tau(\psi))=q(\tau(\varphi)) - (x+\overline{x})
\] 
where $x=q(\varphi_*(\overline{y}))$. 
(Note that $q(\overline{z})=-\overline{q(z)}$.)

Substituting identity~\eqref{eq1} into this equation yields
\[
q(\tau(\varphi))=x+\overline{x}
\] 
which is the contradiction proving the 
Theorem~\ref{the:example_on_block_fibering}.
\end{proof}

%%%%%%%%%%%%%%%%%%%%%%%%%%%%%%%%%%%%%%%%%%%%%%%%%%%%%%%%%%%%%%%%%%%%%%%%%%%%%%%
%%%%%%%%%%%%%%%%%%%%%%%%%%%%%%%%%%%%%%%%%%%%%%%%%%%%%%%%%%%%%%%%%%%%%%%%%%%%%%%
%%%%%%%%%%%%%%%%%%%%%%%%%%%%%%%%%%%%%%%%%%%%%%%%%%%%%%%%%%%%%%%%%%%%%%%%%%%%%%%

\typeout{-----------------------  References ------------------------}

%\bibliographystyle{abbrv}
%\bibliography{dbpub,dbpre}

%\bibliography{../../dbpub,../../dbpre}

%\version{16.03.2017 (Wolfgang L.)}

\end{document}